\documentclass{amsart}
\usepackage{graphicx}
\usepackage[all]{xy}
\usepackage{anyfontsize}
\usepackage{amsmath}
\usepackage{amssymb}
\usepackage{amsthm}
\usepackage{geometry}
\usepackage{mathrsfs}
\usepackage{bbm}
\usepackage{bm}
\usepackage{tikz}
\usetikzlibrary{cd}
\usepackage[colorlinks,linkcolor=black,anchorcolor=black,citecolor=blue,urlcolor=black]{hyperref}
\usepackage{marvosym}
\usepackage{stmaryrd}
\usepackage{enumitem}
\usepackage{mathtools}

\numberwithin{equation}{section}

\theoremstyle{definition}
\newtheorem{theorem}{Theorem}[section]
\newtheorem{corollary}[theorem]{Corollary}

\newtheorem{lemma}[theorem]{Lemma}
\newtheorem{definition}[theorem]{Definition}
\newtheorem{definition-proposition}[theorem]{Definition-Proposition}
\newtheorem{proposition}[theorem]{Proposition}
\newtheorem{example}[theorem]{Example}
\newtheorem{remark}[theorem]{Remark}

\newtheorem*{claim}{Claim}
\newtheorem{theorem-def}[theorem]{Theorem-Definition}

\allowdisplaybreaks[4]

\newcommand{\nc}{\newcommand}
\nc{\ad}{\mathrm{ad}}
\nc{\Alt}{\mathrm{Alt}}
\nc{\bbB}{\mathbb{B}}
\nc{\bbK}{\mathbb{K}}
\nc{\bbN}{\mathbb{N}}
\nc{\bbP}{\mathbb{P}}
\nc{\bbQ}{\mathbb{Q}}
\nc{\bbR}{\mathbb{R}}
\nc{\bbZ}{\mathbb{Z}}
\nc{\bfk}{{\mathbf k}}
\nc{\calA}{{\mathcal A}}
\nc{\calB}{{\mathcal B}}
\nc{\calC}{{\mathcal C}}
\nc{\calD}{\mathcal{D}}
\nc{\calE}{{\mathcal E}}
\nc{\calF}{{\mathcal F}}
\nc{\calG}{{\mathcal G}}
\nc{\calH}{{\mathcal H}}
\nc{\calPCH}{{\mathcal{PCH}}}
\nc{\calI}{{\mathcal I}}
\nc{\calJ}{{\mathcal J}}
\nc{\calK}{{\mathcal K}}
\nc{\calL}{{\mathcal L}}
\nc{\calM}{{\mathcal M}}
\nc{\calN}{{\mathcal N}}
\nc{\calO}{{\mathcal O}}
\nc{\calP}{{\mathcal P}}
\nc{\calPH}{\mathcal {PH}}
\nc{\calQ}{{\mathcal Q}}
\nc{\calr}{{\mathcal R}}
\nc{\calS}{{\mathcal S}}
\nc{\calT}{{\mathcal T}}
\nc{\calU}{{\mathcal U}}
\nc{\calV}{{\mathcal V}}
\nc{\calW}{{\mathcal W}}
\nc{\calX}{{\mathcal X}}
\nc{\calZ}{{\mathcal Z}}
\nc{\coH}{\mathrm{co}H}
\nc{\cop}{\mathrm{cop}}
\nc{\dR}{\mathrm{dR}}
\nc{\End}{\mathrm{End}}
\nc{\Ext}{\mathrm{Ext}}
\nc{\FP}{\mathrm{FP}_\infty}
\nc{\fraka}{{\mathfrak a}}
\nc{\frakb}{\mathfrak{b}}
\nc{\frakB}{{\frak B}}
\nc{\frakC}{\mathfrak{C}}
\nc{\frakg}{{\mathfrak g}}
\nc{\frakHP}{{\mathfrak{HP}}}
\nc{\frakl}{{\frak l}}
\nc{\fk}[1]{\mathfrak{#1}}
\nc{\frakm}{{\frak m}}
\nc{\frakM}{{\frak M}}
\nc{\frakP}{{\frak P}}
\nc{\frakW}{{\frak W}}
\nc{\frakX}{{\frak X}}
\nc{\frakt}{{\frak{T}}}
\nc{\frakA}{{\frak A}}
\nc{\gr}{{\mathrm{Gr}}}
\newcommand{\HH}{\operatorname{HH}}
\newcommand{\Hom}{\operatorname{Hom}}
\newcommand{\id}{\operatorname{Id}}
\newcommand{\im}{\operatorname{Im}}
\newcommand{\Ker}{\operatorname{Ker}}
\newcommand{\Mod}{\text{-}\operatorname{Mod}}
\nc{\ob}{\operatorname{ob}}
\newcommand{\ol}{\overline}
\newcommand{\op}{{\mathrm{op}}}
\newcommand{\ot}{\otimes}
\newcommand{\rmB}{{\mathrm{B}}}
\newcommand{\rmC}{{\mathrm{C}}}
\newcommand{\rmd}{{\mathrm{d}}}
\newcommand{\rmE}{\mathrm{E}}
\newcommand{\rmF}{\mathrm{F}}
\newcommand{\rmH}{{\mathrm{H}}}
\newcommand{\rmHH}{{\mathrm{HH}}}
\newcommand{\rmZ}{{\mathrm{Z}}}
\newcommand{\Span}{\operatorname{Span}}
\newcommand{\Tor}{\operatorname{Tor}}
\newcommand{\Tot}{\operatorname{Tot}}
\newcommand{\Vect}{\mathrm{Vect}_\bfk}
\newcommand{\Rmnum}[1]{\uppercase\expandafter{\romannumeral#1}}

\newlist{myenumerate}{enumerate}{1}
\setlist[myenumerate]{label=(\alph*), noitemsep,topsep=0pt}

\usepackage[english]{datetime2}
\DTMnewdatestyle{yearmonth}{%
}
\DTMsetdatestyle{yearmonth}

\begin{document}

\title{\c{S}tefan spectral sequence for faithfully flat Hopf Galois extensions is multiplicative}

\date{\today}

\begin{abstract}
    Let $H$ be a Hopf algebra with a bijective antipode  over a field $\bfk$ and $B/A$ be a flat right $H$-Galois extension.
    \c{S}tefan constructed a spectral sequence converging to the Hochschild cohomology $\HH^{p+q}(B, N)$ with $\rmE_2^{p,q} = \rmH^p(H, \HH^q(A, N))$. 
    We show that when $B/A$ is faithfullt flat,  the \c{S}tefan  spectral sequence  is multiplicative. 
    
More precisely,   a new formalism of functorial multiplicative spectral sequences via lax monoidal functors over monoidal categories is introduced.
   A functorial multiplicative spectral sequence is constructed for  any faithfully flat Hopf Galois extension. 
    Identification with \c{S}tefan's spectral sequence from the $E_2$-page follows from K\"unzer's comparison theorems.
    Applications include strongly graded algebras, smash products, crossed products, group and Lie algebra extensions, with new multiplicativity results in several cases.
\end{abstract}

\author{Liyu Liu}
\address[Liyu Liu]{School of Mathematics, Yangzhou University, No. 180 Siwangting Road, 225002 Yangzhou, Jiangsu Province,  P.R.China}
\email[Liyu Liu]{lyliu@yzu.edu.cn}

\author{Hongguang Nie}
\author{Guodong Zhou}
\address[Hongguang Nie and Guodong Zhou]{School of Mathematical Sciences,
     Laboratory of MEA (Ministry of Education),
	 Shanghai Key Laboratory of PMMP,
	 East China Normal University,
	 Shanghai 200241,
	 P.R.China}
\email[Hongguang Nie]{52285500007@stu.ecnu.edu.cn}
\email[Guodong Zhou]{gdzhou@math.ecnu.edu.cn}

\author{Ruipeng Zhu}
\address[Ruipeng Zhu]{School of Mathematics,
     Shanghai University of Finance and Economics,
	 Shanghai 200433,
	 P.R.China}
\email[Ruipeng Zhu]{zhuruipeng@sufe.edu.cn}

\date{\today}

	\renewcommand{\thefootnote}{\alph{footnote}}
	\setcounter{footnote}{-1} 
	\renewcommand{\thefootnote}{\alph{footnote}}
	\setcounter{footnote}{-1} 

	\maketitle

\tableofcontents

\section*{Introduction}

Let $H$ be a Hopf algebra with a bijective antipode  over a field $\bfk$ and $B/A$ be a flat right $H$-Galois extension. In his influential work \cite{Ste95}, \c{S}tefan constructed a convergent spectral sequence
\begin{equation}\label{eqn: Stefan-SS}
\rmE_2^{p,q} = \rmH^p(H, \HH^q(A, N)) \Longrightarrow \HH^{p+q}(B, N),
\end{equation}
where $N$ is a $(B,B)$-bimodule, $\HH^*(A,-)$ denotes the Hochschild cohomology and $\rmH^*(H,-)$ denotes the cohomology of the Hopf algebra $H$.
This spectral sequence provides a powerful tool for computing the Hochschild cohomology of $B$ in terms of the cohomology of $H$ and the Hochschild cohomology of the coinvariant subalgebra $A$.
It unifies and generalizes several classical spectral sequences in homological algebra, including the Lyndon-Hochschild-Serre spectral sequence for group extensions \cite{HS53}, the Hochschild-Serre spectral sequence for Lie algebra extensions \cite{HochschildSerre1953}, and their analogues for smash product algebras \cite{Neg15}.

A fundamental question concerning any spectral sequence arising from cohomology theories is whether it admits a multiplicative structure, that is, whether the cohomology groups on each page carry a natural algebra structure compatible with the differentials and the convergence.
In the aforementioned special cases, the multiplicativity of the corresponding spectral sequences has been established by Beyl \cite{Beyl81}, Fuks \cite{Fuks86}, and Negron \cite{Neg15}, respectively.
However, the multiplicative property of the \c{S}tefan spectral sequence in full generality has remained open.
The main obstruction lies in the noncommutative nature of the setting: unlike the classical cases, a general Hopf Galois extension involves a possibly noncommutative coinvariant subalgebra $A$ and a noncocommutative Hopf algebra $H$.

The purpose of the present paper is to establish the multiplicative property of the \c{S}tefan spectral sequence for faithfully flat Hopf Galois extensions.
Our approach is based on a new formalism of functorial multiplicative spectral sequences, which we develop systematically in Subsection~\ref{subsec: Multiplicative functorial spectral sequences}.
This formalism extends the classical theory of multiplicative spectral sequences (see, e.g., McCleary \cite{McC01}) to the functorial setting over a monoidal category, where the multiplicative structures are encoded in the monoidal constraints of the underlying lax monoidal functors.
The key concepts introduced include graded lax monoidal functors, differential graded lax monoidal functors, and filtered differential graded lax monoidal functors, which provide the appropriate categorical framework for constructing multiplicative spectral sequences in a functorial manner.
We show that a filtered differential graded lax monoidal functor gives rise to a functorial multiplicative spectral sequence converging multiplicatively to its cohomology (Theorem \ref{thm:construction}).
Moreover, a bicomplex lax monoidal functor induces two functorial multiplicative spectral sequences via the standard column and row filtrations. 

In Section~\ref{section: Constructing Cup products in cohomology theories}, we recall the construction of cup products in Hochschild cohomology of associative algebras and in the cohomology theory of Hopf algebras.
We interpret these cup products from the perspective of differential graded corings and lax monoidal functors.
The classical bar resolution $B_\bullet = B_\bullet(R;R;R)$ of an algebra $R$ carries a natural differential graded coalgebra structure, and the Hochschild cochain complex $\Hom_{R^e}(B_\bullet, -)$ becomes a differential graded lax monoidal functor (Proposition \ref{prop:hom-lax}).
Similarly, for a Hopf algebra $H$, the one-sided bar resolution $B'_\bullet = B_\bullet(\bfk;H;H)$ admits a differential graded coalgebra structure, and the cochain complex $\Hom_H(B'_\bullet, -)$ is also a differential graded lax monoidal functor.
These constructions provide the algebraic foundation for the multiplicative structures appearing in the \c{S}tefan spectral sequence.

Section~\ref{section: Hopf Galois extensions} collects the necessary background on Hopf Galois extensions.
The notion of a Hopf-Galois extension provides a far-reaching noncommutative generalization of the classical Galois theory of field extensions and, more geometrically, of the theory of principal bundles.
We review the definition and basic properties of Hopf Galois extensions, the Ehresmann-Schauenburg bialgebroid $\Gamma$, and the $H$-module structures on certain $\Hom$ spaces and tensor products following \cite{LRW25}.
A key result is that when $B/A$ is a faithfully flat $H$-Galois extension, the Ehresmann-Schauenburg bialgebroid $\Gamma$ is projective as a left $A^e$-module (Proposition \ref{proposition: Gamma is projective as A^e-module}).
This projectivity ensures that the resolutions needed for constructing the spectral sequence behave well with respect to the monoidal structures.

The main technical work is carried out in Section~\ref{Section: An isomorphism of differential bigraded algebras}.
We construct an isomorphism of bicomplexes
\begin{equation}\label{eqn: isom of bicomplexes}
    \Phi^{\bullet,\bullet}_N: \Hom_H(L_\bullet,\Hom_{A^e}(K_\bullet,N))\cong \Hom_{B^e}(L_\bullet\otimes_H(B^e\otimes_{A^e}K_\bullet),N),
\end{equation}
where $L_\bullet$ is a projective resolution of the trivial right $H$-module $k$ and $K_\bullet$ is a projective resolution of the left $\Gamma$-module $A$ (Proposition \ref{prop: isom of bicomplexes}).
The total complex $\Tot(L_\bullet \otimes_H (B^e \otimes_{A^e} K_\bullet))$ is shown to be a projective resolution of $B$ as a $(B,B)$-bimodule (Proposition \ref{Prop: the total complex is a projective resolution of (B,B) bimodule}).
The crucial point is that both sides of \eqref{eqn: isom of bicomplexes} carry natural exterior products, and the isomorphism $\Phi^{\bullet,\bullet}$ intertwines these products (Proposition \ref{proposition: Phi commutes with cup product}).
To ensure strict associativity of the exterior products (rather than associativity up to homotopy), we make specific choices of resolutions: $L_\bullet$ and $K_\bullet$ are taken to be the (relative) one-sided bar resolution $B_\bullet(\bfk; H; H)$ and $B_\bullet^{A^\op}(\Gamma; \Gamma; A)$, respectively. We construct explicit coassociative diagonal maps on these resolutions (Proposition \ref{prop: the products are coassociative}).
Consequently, the isomorphism $\Phi^{\bullet,\bullet}$ becomes an isomorphism of first quadrant bicomplex lax monoidal functors (Theorem \ref{thm: isom of bicomplexes compatible with associative cup products}).
In Theorem \ref{theorem: A new proof of Stefan spectral sequence}, we obtain a functorial multiplicative spectral sequence
\begin{equation}\label{eqn: mult-SS}
    \rmE^{p,q}_{2}=\rmH^p(H,\HH^q(A,N))\Rightarrow \HH^{p+q}(B,N).
\end{equation}

In Section \ref{Section: The multiplicative property of Stefan Spectral Sequence}, we identify the spectral sequence \eqref{eqn: mult-SS} with the \c{S}tefan spectral sequence.
Using K\"unzer's comparison theorems for spectral sequences involving bifunctors \cite{Kunzer08}, we show that our spectral sequence is isomorphic to the \c{S}tefan spectral sequence from the second page (Theorem \ref{theorem: two contructions are isomorphic}). The key ingredients are the acyclicity conditions established in Lemmas \ref{lemma:  K is (Hom(-,N), Hom_H(k,-))-acyclic} and \ref{lemma: I is (Hom(A,-), Hom_H(k,-))-acyclic}, which allow us to apply K\"unzer's theorems.
Combining this identification with the multiplicative property established in Theorem \ref{theorem: A new proof of Stefan spectral sequence}, we obtain our main result:
\begin{theorem}[Theorem \ref{Stefan spectral sequence is multiplicative}]
    Let $H$ be a Hopf algebra with a bijective antipode and  $B/A$ be a faithfully flat $H$-Galois extension. Then the \c{S}tefan spectral sequence \eqref{eqn: Stefan-SS} is a functorial multiplicative convergent spectral sequence.
\end{theorem}

We conclude the paper with several examples illustrating our main theorem.
These include strongly graded algebras (Example \ref{ex: Stefan spectral Sequence for strongly graded algebras}), smash product algebras (Example \ref{ex: Stefan spectral Sequence for smash product algebras}), crossed product algebras (Example \ref{ex: Stefan spectral Sequence for crossed product algebras}), group extensions recovering the Lyndon-Hochschild-Serre spectral sequence (Example \ref{ex: Lyndon-Hochschild-Serre spectral sequence for a group extension}), Lie algebra extensions recovering the Hochschild-Serre spectral sequence (Example \ref{ex: Chevalley-Eilenberg spectral sequence for an extension of Lie algebras}), and extensions of Hopf algebras (Example \ref{ex: Stefan spectral sequence for an extension of Hopf algebras}).
Notably, the multiplicative property for strongly graded algebras and crossed product algebras seems to be new.

Throughout this paper, $\bfk$ denotes a field, and unadorned tensors and Hom are taken over $\bfk$.

We will use Sweedler's notation. For a coalgebra $C=(C, \Delta, \epsilon)$, we write
$\Delta(c)=\sum c_{(1)}\ot c_{(2)}$ for any $c\in C$.
For a left $C$-comodule $M=(M, \lambda)$, we write  $\lambda(m)=\sum m_{(-1)}\ot m_{(0)}$ for any $m\in M$, and a right $C$-comodule $N=(N, \rho)$, we write $\rho(n)=\sum n_{(0)}\ot n_{(1)}$ for any $n\in N$.

\section{The formalism of functorial multiplicative spectral sequences}

This section recalls some preliminaries on multiplicative spectral sequences
and introduces a new formalism for functorial multiplicative spectral
sequences.

\subsection{Multiplicative spectral sequences}

For basics about spectral sequences, we refer the reader to \cite{McC01} or \cite[Chapter 5]{Weibel94}.

We introduce several definitions in order to fix the notations.
Let $\Vect$ be the monoidal category of vector spaces over $\bfk$. 

\begin{definition}
A \textbf{cohomological spectral sequence} (starting from $a\in\bbZ$) in $\Vect$ ($\Vect$ could be replaced by the category of modules over a $\bfk$-algebra) consists of the following data:
\begin{itemize}
\item[(a)] For each $r\geq a$, a bigraded object $\rmE_r=\{\rmE^{p,q}_r\}_{p,q\in \bbZ}$   of $\Vect$;
\item[(b)] for each $r\geq a$, a bigraded morphism $\rmd_r=\{\rmd_r^{p,q}:\rmE_r^{p,q}\rightarrow \rmE_r^{p+r,q-r+1}\}_{p, q \in \bbZ}$  of degree $(r, -r+1)$ which forms a  differential  in the sense that $\rmd_r^{p+r, q-r+1}\circ \rmd_r^{p, q}=0$, $\forall p, q\in \bbZ$;
\item[(c)] for each $r\geq a$, bigraded isomorphisms between $\rmE_{r+1}$ and the cohomology group  $\rmH^{*,*}(\rmE_r, \rmd_r)$, that is, $$\rmE^{p,q}_{r+1}\cong \Ker(\rmd_r^{p,q})/\im(\rmd_r^{p-r,q+r-1}), \quad \forall\ p, q\in \bbZ.$$
\end{itemize}
\end{definition}

Given a spectral sequence $\{(\rmE_r, \rmd_r)\}_{r\geq a}$, for any $p,q\in \bbZ$, denote
$$\rmZ^{p,q}_{a}=\Ker(\rmd^{p,q}_a: \rmE_a^{p,q}\rightarrow \rmE_a^{p+a,q-a+1}), \quad \rmB^{p,q}_{a}=\im(\rmd^{p-a,q+a-1}_a: \rmE_a^{p-a,q+a-1}\rightarrow \rmE_a^{p,q}),$$
then $\rmE^{p,q}_{a+1}\cong \rmZ^{p,q}_{a}/\rmB^{p,q}_{a}$; there exist subobjects $\rmZ^{p,q}_{a+1}$ and $\rmB^{p,q}_{a+1}$ of $\rmE^{p,q}_{a}$ such that 
$$\Ker(\rmd^{p,q}_{a+1}: \rmE_{a+1}^{p,q}\rightarrow \rmE_{a+1}^{p+a+1,q-a})\cong \rmZ^{p,q}_{a+1}/\rmB^{p,q}_a, \quad \im(\rmd^{p-a-1,q+a}_{a+1}: \rmE_{a+1}^{p-a-1,q+a}\rightarrow \rmE_{a+1}^{p,q})\cong \rmB^{p,q}_{a+1}/\rmB^{p,q}_{a}.$$
 In this way, we obtain bigraded subfamilies $\rmB_r=\{\rmB^{p,q}_{r}\}_{p,q\in \bbZ}$ and 
 $\rmZ_r=\{\rmZ^{p,q}_{r}\}_{p,q\in \bbZ}$ with $r\geq a$
of $\rmE_a$ such that they form a sequence
$$0\subseteq \rmB_a\subseteq \rmB_{a+1}\subseteq \cdots\subseteq \rmB_{r}\subseteq \rmB_{r+1} \subseteq \cdots\subseteq \rmZ_{r+1}\subseteq \rmZ_{r}\subseteq\cdots \subseteq \rmZ_{a+1}\subseteq \rmZ_a\subseteq \rmE_a$$ and
such that for each $r\geq a$, $\rmE_{r+1}\cong \rmZ_r/\rmB_r$. Denote $\rmB_\infty=\cup_{r\geq a} \rmB_r$ and $\rmZ_\infty=\cap_{r\geq a} \rmZ_r$, and define  a bigraded object $\rmE_\infty=\rmZ_\infty/\rmB_\infty$.

\begin{definition}\label{definition: Classical  convergence}
Let $\rmH^*=\{\rmH^n\}_{n\in \bbZ}$ be a family of objects of $\Vect$ endowed with decreasing filtrations:
$$\rmF^*(\rmH^n)=(\cdots \subseteq \rmF^p (\rmH^n)\subseteq \rmF^{p-1} (\rmH^n)\subseteq \cdots).$$
A spectral sequence $\{(\rmE_r, \rmd_r)\}_{r\geq a}$ is said to \textbf{converge} to $(\rmH^*, \rmF^*)$ if there exists an isomorphism $\rmE_\infty^{p,q}\cong \rmF^p (\rmH^{p+q})/\rmF^{p+1}(\rmH^{p+q})$ for any $p,q\in \bbZ$. In this case, we write $\rmE_a^{p,q}\Longrightarrow \rmH^{p+q}$.
\end{definition}

There are several versions of convergence theorems for spectral sequences. Since all the spectral sequences appearing in this paper satisfy the hypotheses of the Classical Convergence Theorem, we will confine ourselves to this kind of convergence theorem.
\begin{theorem}[{see for instance, \cite[Theorem 2.6]{McC01}}]
\label{thm: constructing spectral sequence from filtered complexes}
Let $\rmC^\bullet=(\rmC^\bullet, \rmd_C)$ be a complex in $\Vect$. Let $\rmF^*(\rmC^\bullet)$ be a decreasing filtration by subcomplexes of $\rmC^\bullet$, which is bounded in the sense that for each $n\in \bbZ$, there exist $s(n)\leq t(n)$ such that $\rmF^{s(n)}(\rmC^n)=C^n, \rmF^{t(n)}(\rmC^n)=0$.
Then there exists a spectral sequence
$$\rmE^{p,q}_1=\rmH^{p+q}(\rmF^p(\rmC^\bullet)/\rmF^{p+1}(\rmC^\bullet))\Longrightarrow \rmH^{p+q}(\rmC^\bullet),$$
which converges to $\rmH^*(\rmC^\bullet)$, whose filtration is induced from that of $\rmC^\bullet$.
\end{theorem}

Recall that a first quadrant bicomplex is a bigraded object $\rmC=\{\rmC^{p,q}\}_{p,q\geq 0}$ in $\Vect$ endowed with a horizontal differential $\rmd_h$ and a vertical differential $\rmd_v$ such that $$\rmd_h\circ \rmd_h=\rmd_v\circ \rmd_v=\rmd_h\circ \rmd_v+\rmd_v\circ \rmd_h=0.$$
Its total complex $\Tot(\rmC)$ is defined by $\Tot(\rmC)^n=\oplus_{p+q=n} \rmC^{p,q}$ with differential $\rmd_h+\rmd_v$.

\begin{corollary}[{see for instance, \cite[Theorem 2.15]{McC01}}]
\label{corollary: spectral sequence induced by a bicomplex}
Let $(\rmC, \rmd_h, \rmd_v)$ be a first quadrant bicomplex.
Then the two induced spectral sequences (by the standard column/row filtrations) converge to $\rmH(\Tot(C))$.
\end{corollary}

\begin{definition}\label{def: multiplicative spectral sequence}
A \textbf{multiplicative spectral sequence} is a spectral sequence $\{(\rmE_r,\rmd_r)\}_{r\geq a}$, together with, for each $r\geq a$, a family of morphisms $\mu_r=\{\mu_r^{p,q; u,v}: \rmE^{p,q}_r\otimes \rmE^{u,v}_r \rightarrow \rmE^{p+u,q+v}_r\}_{p, q, u, v\in \bbZ}$ such that:
\begin{itemize}
\item[(a)] For each $r\geq a$, $(\rmE_r,\mu_r,\rmd_r)$ is a differential bigraded algebra;
\item[(b)] the product $\mu_{r+1}$ on $\rmE_{r+1}$ is induced from $\mu_r$ via the bigraded isomorphism $\rmH^{*,*}(\rmE_r,\rmd_r)\cong \rmE_{r+1}$.
\end{itemize}
\end{definition}

\begin{definition}\label{definition: Classical multiplicative convergence}
A multiplicative spectral sequence $\{(\rmE_r,\rmd_r, \mu_r)\}_{r\geq a}$ \textbf{converges multiplicatively} to a filtered graded algebra $\rmH=(\rmH^*, \mu_H, \rmF^*)$ if it converges to $(\rmH^*, \rmF^*)$ and the family of isomorphisms $\rmE_\infty^{p,q}\cong \rmF^p (\rmH^{p+q})/\rmF^{p+1}(\rmH^{p+q})$ forms an isomorphism of bigraded algebras $\rmE_\infty\cong \gr_{\rmF^*}(\rmH^*)$.
\end{definition}

\begin{theorem}[{see for instance, \cite[Theorem 2.14]{McC01}}]\label{thm: constructing multiplicative spectral sequence from filtered dga}
Let $A=(A^\bullet, \rmd_A, \mu)$ be a differential $\bbN$-graded $\bfk$-algebra in $\Vect$. Let $\rmF^*(A)$ be a decreasing bounded filtration by subcomplexes, which is multiplicative.
Then there exists a multiplicative spectral sequence
$$\rmE^{p,q}_1=\rmH^{p+q}(\rmF^p(A)/\rmF^{p+1}(A))\Longrightarrow \rmH^{p+q}(A),$$
which converges multiplicatively to $\rmH^*(A)$.
\end{theorem}

\begin{corollary}\label{corollary: multiplicative for spectral sequence induced by a bicomplex}
Let $\rmC=(\rmC^{\bullet, \bullet}, \rmd_h, \rmd_v)$ be a first quadrant bicomplex.
Assume that $\rmC$ is endowed with a differential bigraded algebra structure $\mu$.
Then the two induced spectral sequences are multiplicative and converge multiplicatively to $\rmH(\Tot(C))$.
\end{corollary}

\subsection{Functorial multiplicative  spectral sequences}\label{subsec: Multiplicative functorial spectral sequences}
Let $\calM$ be an additive category and let $\Vect$ be the monoidal category of vector spaces over $\bfk$.

\begin{definition}\label{def: M indexed cohomological spectral sequence}
A  \textbf{functorial   spectral sequence} (starting from $a\in\bbZ$) indexed by $\calM$ consists of:
\begin{itemize}
\item[(a)] for each $r\geq a$, a bigraded family $\calE_r=\{\calE^{p,q}_r: \calM\to \Vect\}_{p,q\in \bbZ}$ of additive functors;
\item[(b)] for each $r\geq a$, a functorial   differential $\rmd_r=\{\rmd_r^{p,q}:\calE_r^{p,q}\rightarrow \calE_r^{p+r,q-r+1}\}_{p,q\in \bbZ }$   of degree $(r, -r+1)$ formed by   natural transformations;
\item[(c)] for each $r\geq a$, bigraded  functorial isomorphisms $\calE^{p,q}_{r+1}\cong \Ker(\rmd_r^{p,q})/\im(\rmd_r^{p-r,q+r-1})$.
\end{itemize}
\end{definition}

Similarly, given a spectral sequence $\{(\calE_r, \rmd_r)\}_{r\geq a}$, for any $p,q\in \bbZ$ and $M\in \calM$, denote
$$\calZ^{p,q}_{a}(M)=\Ker(\rmd^{p,q}_a(M): \calE_a^{p,q}(M)\rightarrow \calE_a^{p+a,q-a+1}(M)), \quad \calB^{p,q}_{a}(M)=\im(\rmd^{p-a,q+a-1}_a(M): \calE_a^{p-a,q+a-1}(M) \rightarrow \calE_a^{p,q}(M)),$$
then $\calE^{p,q}_{a+1}(M)\cong \calZ^{p,q}_{a}(M)/\calB^{p,q}_{a}(M)$; there exist subfunctors $\calZ^{p,q}_{a+1}$ and $\calB^{p,q}_{a+1}$ of $\calE^{p,q}_{a}$ such that 
$$\Ker(\rmd^{p,q}_{a+1}: \calE_{a+1}^{p,q}\rightarrow \calE_{a+1}^{p+a+1,q-a})\cong \calZ^{p,q}_{a+1}/\calB^{p,q}_a, \quad \im(\rmd^{p-a-1,q+a}_{a+1}: \calE_{a+1}^{p-a-1,q+a}\rightarrow \calE_{a+1}^{p,q})\cong \calB^{p,q}_{a+1}/\calB^{p,q}_{a}.$$
In this way, we obtain     bigraded subfamilies $\calB_r=\{\calB^{p,q}_{r}\}_{p,q\in \bbZ}$ and 
 $\calZ_r=\{\calZ^{p,q}_{r}\}_{p,q\in \bbZ}$ with $r\geq a$
of $\calE_a$ such that they form a sequence
$$0\subseteq \calB_a\subseteq \calB_{a+1}\subseteq\cdots\subseteq \calB_{r}\subseteq \calB_{r+1} \subseteq \cdots\subseteq \calZ_{r+1}\subseteq  \calZ_{r}\subseteq\cdots \subseteq  \calZ_{a+1}\subseteq\calZ_a \subseteq \calE_a$$
and such that for each $r\geq a,p,q\in \bbZ$, there are functorial isomorphisms $\calE^{p,q}_{r+1}\cong \calZ^{p,q}_r/\calB^{p,q}_r$. For any $p,q\in \bbZ$ and $M\in \calM$,  denote $\calB^{p,q}_\infty(M)=\cup_{r\geq a} \calB^{p,q}_r(M)$ and $\calZ^{p,q}_\infty(M)=\cap_{r\geq a} \calZ^{p,q}_r(M)$, and  define an additive  functor $\calE^{p,q}_\infty: \calM\to \Vect$ by  $\calE^{p,q}_\infty(M)=\calZ^{p,q}_\infty(M)/\calB_\infty^{p,q}(M)$.

\begin{definition}\label{def: convergence of M indexed cohomological spectral sequence}
Let $\calH^*=\{\calH^n\}_{n\in \bbZ}$ be a graded  family of additive functors from $\calM$ to $\Vect$ endowed with decreasing exhaustive Hausdorff  filtrations $\calF^*(\calH^n)$ by subfunctors:
$$\calF^*(\calH^n)=(\cdots \subseteq \calF^p (\calH^n)\subseteq \calF^{p-1} (\calH^n)\subseteq \cdots).$$

An $\calM$-indexed functorial spectral sequence is said to \textbf{converge} to $(\calH^*, \calF^*)$ if there exist natural isomorphisms $\calE_\infty^{p,q}\cong \calF^p \calH^{p+q}/\calF^{p+1}\calH^{p+q}$ of functors.
\end{definition}

Next, in order to deal with  multiplicative spectral sequences indexed by a category, we introduce the formalism of functorial multiplicative spectral sequences. This formalism seems to be new, although it should be folklore.

We shall consider lax monoidal functors between monoidal   categories. Examples of monoidal  categories include the category of vector spaces, the category of modules or comodules over a Hopf algebra, the category of bimodules over a (noncommutative) associative $\bfk$-algebra. Note that the latter  monoidal category is NOT symmetric.

Let us introduce several monoidal   categories of vector spaces.

\begin{itemize}
\item[(a)]
Let $\mathrm{GrVect}_{\bfk}$ denote the category of $\bbZ$-graded $\bfk$-vector spaces, with objects $V^{*} = \{V^{n}\}_{n \in \bbZ}$ and morphisms being grading-preserving linear maps. It is equipped with the monoidal structure
\[
(V^{*} \otimes W^{*})^{k} = \bigoplus_{m+n=k} V^{m} \otimes_{\bfk} W^{n},
\]
with unit $\underline{\bfk} = \bfk[0]$ (the field $\bfk$ concentrated in degree $0$).
A nonzero element $v\in V^k$ is said to be homogeneous of degree $k$, denoted by  $|v|=k$.
\item[(b)]
Let $\mathrm{DGVect}_{\bfk}$ denote the category of differential graded $\bfk$-vector spaces (or complexes of $\bfk$-vector spaces). Its objects are of the form $$V^\bullet=(V^{*}, \rmd_{V})=(\{V^n\}_{n\in \bbZ}, \{\rmd_V^n: V^{n} \to V^{n+1}\}_{n\in \bbZ})$$ with $\rmd_V^2=0$ (that is,  $\rmd_V^{n+1}\circ \rmd_{V}^{n} = 0, \forall n\in \bbZ$) and morphisms being chain maps. The graded tensor product  of $V^\bullet=(V^{*}, \rmd_{V})$ and $W^\bullet=(W^{*}, \rmd_{W})$ has its underlying graded space $V^{*} \otimes W^{*}$ and carries the differential
\[
\rmd_{V \otimes W}(v \otimes w) = \rmd_{V}(v) \otimes w + (-1)^{|v|} v \otimes \rmd_{W}(w), 
\]
where $v\in V, w\in W$ are homogeneous elements.  
The unit is $\underline{\bfk}=\bfk[0]$ with zero differential.

\item[(c)]  A filtered graded vector space is a graded vector space $V^{*}=\{V^n\}_{n\in \bbZ}$ equipped with a decreasing exhaustive Hausdorff filtration by graded subspaces
$$\rmF^*(V^*)=  (\cdots \subseteq \rmF^{p+1} (V^*) \subseteq \rmF^p (V^*)\subseteq \rmF^{p-1} (V^*)\subseteq \cdots).  $$  
Let $\mathrm{FGVect}_{\bfk}$ denote the category of filtered graded vector spaces whose morphisms are filtration preserving graded  maps of degree $0$.   
 The tensor product of two filtered graded vector spaces $(V^*, \rmF^*(V^*))$ and  $(W^*, \rmF^*(W^*))$ is given by  the  tensor product of graded vector spaces $V^*\ot W^*$ with induced filtration
 \[
\rmF^{p}(V^{*} \otimes W^{*}) = \sum_{i+j=p} \rmF^{i}(V^{*}) \otimes \rmF^{j}(W^{*}).
\] 
The tensor  unit is $\underline{\bfk}=\bfk[0]$ with trivial filtration
$$\rmF^{p}(\bfk[0])=\left\{\begin{array}{ll} k & p\leq 0, \\ 0 & p>0\end{array}\right. $$.

 \item[(d)] A filtered differential graded vector space is a differential graded vector space $$V^\bullet=(V^{*}, \rmd_{V})=(\{V^n\}_{n\in \bbZ}, \{\rmd_V^n: V^{n} \to V^{n+1}\}_{n\in \bbZ})$$ equipped with a decreasing exhaustive Hausdorff filtration by subcomplexes
$$\rmF^*(V^\bullet)=  (\cdots \subseteq \rmF^{p+1} (V^\bullet) \subseteq \rmF^p (V^\bullet)\subseteq \rmF^{p-1} (V^\bullet)\subseteq \cdots).  $$   Let $\mathrm{FDGVect}_{\bfk}$ denote the category of filtered differential graded vector spaces.
The tensor product of two filtered differential graded vector spaces $(V^\bullet, \rmF^*(V^\bullet))$ and  $(W^\bullet, \rmF^*(W^\bullet))$ is given by  the  tensor product of filtered differential graded vector spaces $(V^\bullet\ot W^\bullet, \rmF^*(V^\bullet\ot W^\bullet) )$  
and the tensor unit is $\underline{\bfk}=\bfk[0]$ with trivial filtration and trivial differential. 

\item[(e)]Let $\mathrm{BiGrVect}_{\bfk}$ denote the category of  bigraded $\bfk$-vector spaces, with objects $E^{*,*} = \{E^{p,q}\}_{p,q \in \bbZ}$ and whose morphisms are bigraded maps of degree $(0,0)$. Its monoidal structure is
\[
(E^{*,*} \otimes F^{*,*})^{p,q} = \bigoplus_{\substack{s+u=p \\ t+v=q}} E^{s,t} \otimes_{\bfk} F^{u,v},
\]
with unit $\underline{\bfk} = \bfk[0,0]$ (the base field $\bfk$  put in degree $(0, 0)$).

 \item[(f)]  Let $r\in \bbZ$ be a fixed integer. A differential  bigraded $\bfk$-vector space  of degree $(r, -r+1)$ is a bigraded $\bfk$-vector space  $E^{*,*} = \{E^{p,q}\}_{p,q \in \bbZ}$ endowed with a bigraded map $\rmd=\{\rmd^{p,q}: E^{p,q}\to E^{p+r, q-r+1}\}_{p,q\in \bbZ}$, called the differential, subject to $\rmd^2=0$, that is, $\rmd^{p+r, q-r+1}\circ \rmd^{p,q}=0$ for any $ p,q\in \bbZ$.
 
 Let ${}_r\mathrm{DBGVect}_{\bfk}$ denote the category of  differential  bigraded $\bfk$-vector spaces of degree $(r, -r+1)$, whose morphisms are bigraded maps of degree (0,0) commuting with the differential. Its monoidal structure $E\ot F$ for $E=(E^{*, *}, \rmd_E)$ and $F=(F^{*,*}, \rmd_F)$ is  given by the tensor product of bigraded vector spaces 
\[
(E^{*,*} \otimes F^{*,*})^{p,q} = \bigoplus_{\substack{s+u=p \\ t+v=q}} E^{s,t} \otimes_{\bfk} F^{u,v},
\]
which  carries the differential
\[
\rmd_{E \otimes F}(v \otimes w) = \rmd_{E}(v) \otimes w + (-1)^{p+q} v \otimes \rmd_{F}(w), 
\]
for $v\in E^{p,q}, w\in F$.  The unit is  $\underline{\bfk} = \bfk[0,0]$ with zero differential. 

 \item[(g)] A  bicomplex is a  bigraded $\bfk$-vector spaces  $ E^{*,*} = \{E^{p,q}\}_{p,q \in \bbZ}$ endowed with a bigraded map $\rmd_h=\{\rmd_h^{p,q}: E^{p,q}\to E^{p+1,q}\}_{p,q\in \bbZ}$ of degree $(1,0)$, called the horizontal differential,  and a bigraded map $\rmd_v=\{\rmd_v^{p,q}: E^{p,q}\to E^{p,q+1}\}_{p,q\in \bbZ}$ of degree $(0,1)$, called the vertical differential, such that $\rmd_v^2=0, \rmd_h^2=0$ and 
 $\rmd_h\rmd_v+\rmd_v\rmd_h=0$. 
Let $\mathrm{BiCom}_{\bfk}$ denote the category of  bicomplexes of  $\bfk$-vector spaces, whose morphisms are bigraded maps of degree (0,0) commuting with the two differentials.
Its monoidal structure $E\ot F$ for $E=(E^{*, *}, \rmd_h^E, \rmd_v^E)$ and $F=(F^{*,*}, \rmd_h^F, \rmd_v^F)$ is  given by the tensor product of bigraded vector spaces 
\[
(E^{*,*} \otimes F^{*,*})^{p,q} = \bigoplus_{\substack{s+u=p \\ t+v=q}} E^{s,t} \otimes_{\bfk} F^{u,v},
\]
which  carries the horizontal differential
\[
\rmd_h^{E \otimes F}(v \otimes w) = \rmd_h^{E}(v) \otimes w + (-1)^{p+q} v \otimes \rmd_h^{F}(w), 
\]
for $v\in E^{p,q}, w\in F$ and a similar vertical differential $\rmd_v^{E\ot F}$.  The unit is  $\underline{\bfk} = \bfk[0,0]$ with zero horizontal and vertical differentials. 
\end{itemize}

From now on, in this section, $\calM$ denotes a $\bfk$-linear monoidal abelian category. 
\begin{definition} \label{def:graded-lax}
A \textbf{graded lax monoidal functor} from $\calM$ to $\Vect$ is a lax monoidal functor
\[
\calH : (\calM, \ot_{\calM}, 1) \longrightarrow (\mathrm{GrVect}_{\bfk}, \otimes, \underline{\bfk}).
\]
Concretely, it consists of:
\begin{enumerate}
\item[(a)] a family of additive functors $\calH^{n} : \calM \to \Vect$ for $n \in \bbZ$;
\item[(b)] a natural transformation (the monoidal constraint)
\[
-\times- : \calH \otimes \calH \longrightarrow \calH,
\]
with  components
\begin{equation}\label{eq:graded-constraint}
-\times_{M,N}^{m,n}- : \calH^{m}(M) \ot \calH^{n}(N) \longrightarrow \calH^{m+n}(M \ot_{\calM} N)
\end{equation}
which satisfies associativity: for all $M,N,P \in \calM$, the diagram
\[
\xymatrix@C=20mm{
\calH(M) \otimes \calH(N) \otimes \calH(P) \ar[r]^-{\id \otimes (-\times_{N, P}-)} \ar[d]_{(-\times_{M, N}-) \otimes \id} & \calH(M) \otimes \calH(N \ot_{\calM} P) \ar[d]^{-\times_{M, N \ot_{\calM} P}-} \\
\calH(M \ot_{\calM} N) \otimes \calH(P) \ar[r]^-{-\times_{M \ot_{\calM} N, P}-} & \calH(M \ot_{\calM} N \ot_{\mathcal{M}} P)
}
\]
commutes;
\item[(c)] a unit morphism $\eta : \underline{\bfk} \to \calH(1)$ satisfying the usual unitality constraints.
\end{enumerate}
\end{definition}

\begin{remark} The monoidal constraint  is sometimes called the  \textbf{functorial exterior product}. 
Condition (b) may also be written in components: for all $M,N,P \in \calM$ and $m,n,k \in \bbZ$,
\[
(-\times_{M \ot N, P}^{m+n,k}-) \circ ((-\times_{M,N}^{m,n}-) \otimes \id) = (-\times_{M, N \ot P}^{m,n+k}- )\circ (\id \otimes (-\times_{N,P}^{n,k}-)).
\]
\end{remark}

\begin{definition} \label{def:dg-lax}
A \textbf{differential graded lax monoidal functor} from $\calM$ to $\Vect$ is a lax monoidal functor
\[
\calC : (\calM, \ot_{\calM}, 1) \longrightarrow (\mathrm{DGVect}_{\bfk}, \otimes, \underline{\bfk}).
\]
Concretely, it is a graded lax monoidal functor $(\calC^{*}, \times)$ equipped with a family of  natural transformations $$\rmd = \{\rmd_{M}^{n} : \calC^{n}(M) \to \calC^{n+1}(M)\}_{n\in \bbZ, M\in \calM}$$ such that $\rmd_{M}^{n+1}\circ \rmd_{M}^{n}  = 0, $ for any $  n\in \bbZ, M\in \calM$ and the   functorial exterior product is a morphism of complexes:
\begin{equation}\label{eq:leibniz}
\rmd_{M \ot N} \circ (-\times_{M,N}-) = (-\times_{M,N}-) \circ (\rmd_{M} \otimes \id + \id \otimes \rmd_{N}).
\end{equation}
In components, this reads:
\[
\rmd_{M \ot N}^{m+n}\circ (-\times_{M,N}^{m,n}-) = (-\times_{M,N}^{m+1,n}-)\circ (\rmd_{M}^{m} \otimes \id) + (-\times_{M,N}^{m,n+1}-) \circ (\id \otimes \rmd_{N}^{n}).
\]
\end{definition}

\begin{proposition}\label{prop:cohomology-lax}
Let $(\calC, \times, \rmd)$ be a differential graded lax monoidal functor from $\calM$ to $\Vect$. Then the cohomology
\[
\rmH(\calC) : \calM \longrightarrow \mathrm{GrVect}_{\bfk}, \qquad M \mapsto \rmH^{*}(\calC(M), \rmd_{M})
\]
inherits a canonical graded lax monoidal structure $-\widetilde{\times}-$ induced from $-\times-$ by passing to cohomology.
\end{proposition}

\begin{proof}
By Eq. \eqref{eq:leibniz}, $-\times_{M,N}-$ is a morphism of complexes, and hence induces a well-defined map on cohomology. The associativity of $-\widetilde{\times}-$ follows from that of $-\times-$ by the functoriality of $\rmH^{*}$.
\end{proof}

We give a method to construct differential graded lax monoidal functors from differential graded coalgebras. 
\begin{definition} 
A \textbf{differential $\bbN$-graded coalgebra} in $\calM$ is a non-negatively graded object $B_{*} = \{B_{n}\}_{n \in \bbN}$ in $\calM$ equipped with:
\begin{enumerate}
\item[(a)] a differential $b=\{b_{n} : B_{n} \to B_{n-1}\}_{n\in \bbN}$ (by convention, $B_{-1}=0$);
\item[(b)] a comultiplication $\Delta=\{\Delta_{p,q} : B_{p+q} \to B_{p} \ot_{\calM} B_{q}\}_{p,q\in \bbN}$;
\item[(c)] a counit $\epsilon=\{ \epsilon_n:B_n\to 1 \}$, where $\epsilon_n$ is zero morphism for any $n\neq 0$.
\end{enumerate}
These data satisfy the coassociativity: 
$$(\Delta\ot_\calM \id)\circ \Delta=( \id\ot_\calM \Delta)\circ \Delta,$$
 the counit identities
$$ (\id \otimes_\calM \epsilon)\circ\Delta=\id =(\epsilon\otimes_\calM \id)\circ \Delta, $$
and  the coderivation condition:
\[
\Delta \circ b = (\id \ot_{\calM} b + b \ot_{\calM} \id) \circ \Delta.
\]
In terms of components, those equalities could be expressed as 
$$(\Delta_{p,q}\ot_{\calM} \id_{B_r})\circ \Delta_{p+q, r}=(\id_{B_p}\ot_{\calM} \Delta_{q,r})\circ \Delta_{p, q+r},$$

$$
(\id_{B_n} \ot_\calM \epsilon_0)\circ \Delta_{n,0}=\id_{B_n} =(\epsilon_0 \ot_\calM \id_{B_n})\circ \Delta_{0,n},
$$
and 
$$\Delta_{p,q}\circ b_{p+q+1}=(\id_{B_p}\ot_{\calM} b_{q+1})\circ \Delta_{p, q+1}+(b_{p+1}\ot_{\calM} \id_{B_q})\circ \Delta_{p+1, q}  $$
for any $p,q,r,n\in \bbN.$
\end{definition}

\begin{proposition}\label{prop:hom-lax}
Let $$B=(B_{*}, b, \Delta, \epsilon)=(\{B_n\}_{n\in \bbN}, \{b_n\}_{n\in \bbN}, \{\Delta_{p,q}\}_{p,q\in \bbN},\{\epsilon_n\}_{n\in \bbN} )$$ be a differential $\bbN$-graded coalgebra in $\calM$. Then the functor
\[
\calC(-) = \Hom_{\calM}(B_{*}, -) : \calM \longrightarrow \mathrm{DGVect}_{\bfk}
\]
admits a canonical differential graded lax monoidal structure $$-\times-=\{-\times_{M,N}^{m,n}-\}_{m,n\in \bbN, M, N\in \calM}$$ given by
\[
f\times g:=f\times_{M,N}^{m,n} g  = (-1)^{mn} (f \ot_{\calM} g) \circ \Delta_{m,n}\in \calC^{m+n}(M\ot_\calM N)
\]
for $f \in \calC^{m}(M)$, $g \in \calC^{n}(N)$, 
differential $\rmd=\{\rmd_M^n\}_{n\in \bbN, M\in \calM}$ given by $$\rmd_{M}^{m}(f) = (-1)^{m+1} f \circ b_{m+1}\in  \calC^{m+1}(M),$$
and the unit morphism given by 
\[
\eta (\lambda)=\lambda \epsilon: \underline{\bfk}\to \calC(1)
\]
for any $\lambda\in \underline{\bfk}$.
\end{proposition}

\begin{proof}
The coassociativity of $\Delta$ ensures the associativity of $-\times-$, the counit identities give the unitality.
And the coderivation condition on $\Delta$ is equivalent to the Leibniz rule \eqref{eq:leibniz} for $\rmd$ with respect to $-\times-$.
\end{proof}

\begin{corollary}
If $A = (A, \mu_{A},\eta_A)$ is an algebra object in $\calM$, then $\calC(A)$ is a differential graded algebra with cup product given by 
$$-\cup_A -=\calC^{m}(A)\ot \calC^{n}(A)\xrightarrow{-\times_{A, A}^{m,n}-} \calC^{m+n}(A\ot_\calM A)\xrightarrow{\calC^{m+n}(\mu_A)} \calC^{m+n}(A).$$

\end{corollary}

\begin{definition} \label{def:fg-lax}
A \textbf{filtered   graded lax monoidal functor} from $\calM$ to $\Vect$ is a lax monoidal functor
\[
\calH : (\calM, \ot_{\calM}, 1) \longrightarrow (\mathrm{FGVect}_{\bfk}, \otimes, \underline{\bfk}).
\]
Concretely, it is a graded lax monoidal functor 
$$(\calH^{*}, -\times-)$$
equipped with   a decreasing exhaustive Hausdorff   filtration by subfunctors
$$\calF^{*}(\calH^n)=(\cdots \subseteq \calF^p(\calH^{n})\subseteq \calF^{p-1}(\calH^{n})\subseteq\cdots ), $$ 
for any $n\in \bbZ$, such that the functorial exterior product is filtered:
\begin{equation}\label{eq:filtered-constraint}
\times_{M,N}^{m,n}\big(\calF^{p}(\calH^m)(M) \otimes \calF^{q}(\calH^n(N)\big) \subseteq \calF^{p+q}(\calH^{m+n})(M \ot_{\calM} N),
\end{equation} 
for any $p,q,m,n\in \bbZ, M, N\in \calM.$
\end{definition}

\begin{definition} \label{def:fdg-lax}
A \textbf{filtered differential graded lax monoidal functor} from $\calM$ to $\Vect$ is a   lax monoidal functor
\[
\calC : (\calM, \ot_{\calM}, 1) \longrightarrow (\mathrm{FDGVect}_{\bfk}, \otimes, \underline{\bfk}).
\]
Equivalently, it is a filtered  graded lax monoidal functor
$$(\calC, -\times-, \calF)$$ such that $\calF^p(\calC^\bullet)$ is a functorial subcomplex of $\calC^\bullet$ for any $p\in \bbZ$,
together with a family of  natural transformations $$\rmd = \{\rmd_{M}^{n} : \calC^{n}(M) \to \calC^{n+1}(M)\}_{n\in \bbZ, M\in \calM}$$ satisfying $\rmd_{M}^{n+1}\circ \rmd_{M}^{n}  = 0,  n\in \bbZ, M\in \calM$ and the functorial exterior product is a morphism of complexes:
\[
\rmd_{M \ot N} \circ (-\times_{M,N}-) = (-\times_{M,N}-) \circ (\rmd_{M} \otimes \id + \id \otimes \rmd_{N}).
\]
\end{definition}

\begin{definition}
A \textbf{bigraded lax monoidal functor} from $\calM$ to $\Vect$ is a lax monoidal functor
\[
\calE : (\calM, \ot_{\calM}, 1) \longrightarrow (\mathrm{BiGrVect}_{\bfk}, \otimes, \underline{\bfk}).
\]
Concretely, it consists of:
\begin{enumerate}
\item[(a)] a family of additive functors $\calE= \{\calE^{p,q} : \calM \to \Vect \}_{p,q\in \bbZ}$;
\item[(b)] a functorial exterior product which satisfies associativity, with components
\[
\times_{M,N}^{p,q;u,v} : \calE^{p,q}(M) \ot_{\bfk} \calE^{u,v}(N) \longrightarrow \calE^{p+u,\, q+v}(M \ot_{\calM} N),
\]
for $p,q,u,v\in \bbZ$ and $M,N\in \calM$.
\end{enumerate}
\end{definition}

\begin{definition} \label{def: bicomplex lax monoidal functor}
A  \textbf{bicomplex lax monoidal functor}  from $\calM$ to $\Vect$ is a lax monoidal functor 
$\calC: \calM\to \mathrm{BiCom}_{\bfk}$. 
Concretely, it is 
a  bigraded lax monoidal functor   
\[
\calC : (\calM, \ot_{\calM}, 1) \longrightarrow (\mathrm{BiCom}_{\bfk}, \otimes, \underline{\bfk}).
\] endowed with a bigraded natural transformation  $\rmd_h=\{\rmd_h^{p,q}: \calC^{p,q}\to \calC^{p+1,q}\}_{p,q\in \bbZ}$ of degree $(1,0)$, called the horizontal differential,  and a bigraded natural transformation  $\rmd_v=\{\rmd_v^{p,q}: \calC^{p,q}\to \calC^{p,q+1}\}_{p,q\in \bbZ}$ of degree $(0,1)$, called the vertical differential, such that $\rmd_v^2=0, \rmd_h^2=0$ and 
$\rmd_h\rmd_v+\rmd_v\rmd_h=0$, 
such that the functorial exterior product is a morphism of bicomplexes
\[
(\rmd_h)_{M\ot N}\circ(-\times_{M,N}-)=(-\times_{M,N}-)\circ ((\rmd_h)_M\ot \id+\id\ot (\rmd_h)_N) \\
(\rmd_v)_{M\ot N}\circ(-\times_{M,N}-)=(-\times_{M,N}-)\circ ((\rmd_v)_M\ot \id+\id\ot (\rmd_v)_N) 
\]
\end{definition}

\begin{definition}
A \textbf{differential bigraded lax monoidal functor of degree $(r,-r+1)$} is a lax monoidal functor
\[
\calE : (\calM, \ot_{\calM}, 1) \longrightarrow ({}_r\mathrm{DBGVect}_{\bfk}, \otimes, \underline{\bfk})
\]

Concretely, it consists of:
\begin{enumerate}
\item[(a)] a family of additive functors $\calE= \{\calE^{p,q} : \calM \to \Vect \}_{p,q\in \bbZ}$;
\item[(b)] a functorial exterior product which satisfies associativity ,with components
\[
\times_{M,N}^{p,q;u,v} : \calE^{p,q}(M) \ot_{\bfk} \calE^{u,v}(N) \longrightarrow \calE^{p+u,\, q+v}(M \ot_{\calM} N),
\]
for $p,q,u,v\in \bbZ$ and $M,N\in \calM$.
\item[(c)] a family of functorial differential 
$\rmd=\{\rmd_M^{p,q}: \calE^{p,q}(M) \to \calE^{p+r,\, q-r+1}(M)\}_{p,q\in \bbZ, M\in \calM}$
of degree $(r, -r+1)$ subject to the Leibniz equality:
$$ 
\rmd_{M \ot N}^{p+u,\, q+v} \circ (-\times_{M,N}^{p,q;u,v}-) = (-\times_{M,N}^{p+r,\, q-r+1;\, u,v}-) \circ (\rmd_{M}^{p,q} \otimes \id)  
+ (-\times_{M,N}^{p,q;\, u+r,\, v-r+1}-) \circ (\id \otimes \rmd_{N}^{u,v}).
$$ 
\end{enumerate}

\end{definition}

\begin{definition} \label{def:Functorial  multiplicative  spectral sequence}
A \textbf{functorial multiplicative  spectral sequence} (starting from $a \in \bbZ$) is a family $\{ (\calE_{r}, \times_{r}, \rmd_{r}) \}_{r \geq a}$ where:
\begin{enumerate}
\item[(a)] for each $r \geq a$, $\calE_{r} : (\calM, \ot_{\calM}, 1) \to ({}_r\mathrm{DBGVect}_{\bfk}, \otimes, \underline{\bfk})$ is a differential bigraded lax monoidal functor of degree $(r,-r+1)$;
\item[(b)] for each $M \in \calM$, the family of pairs $\{(\calE_{r}(M), \rmd_{r}(M))\}_{r\ge a}$ is a   spectral sequence;
\item[(c)] for each $r \geq a$, there are functorial bigraded isomorphisms $\rmH^{*,*}(\calE_{r}, \rmd_{r}) \cong \calE_{r+1}$ which identify the lax monoidal structure on $\calE_{r+1}$ with the one induced from $\calE_{r}$ via Proposition~\ref{prop:cohomology-lax}.
\end{enumerate}
\end{definition}

\begin{definition} \label{def:Multiplicative convergence}
A  functorial multiplicative spectral sequence $\{ (\calE_{r}, -\times_{r}-, \rmd_{r}) \}_{r \geq a}$ \textbf{converges multiplicatively} to a filtered graded lax monoidal functor $\calH : \calM \to \mathrm{FGVect}_{\bfk}$ if:
\begin{enumerate}
\item[(a)] for each $M \in \calM$, the spectral sequence $$\{(\calE_{r}(M), \rmd_{r}(M))\}_{r\ge a}$$ converges to $(\calH(M), \calF^{*}(\calH)(M))$;
\item[(b)] the bigraded functorial isomorphisms $\calE_{\infty}^{p,q} \cong \calF^{p}\calH^{p+q} / \calF^{p+1}\calH^{p+q}$ form a natural isomorphism of lax monoidal functors, i.e., the diagram
\[
\xymatrix{
\calE_{\infty}(M) \otimes \calE_{\infty}(N) \ar[r]^-{\times_{\infty}} \ar[d]_{\cong} & \calE_{\infty}(M \ot_{\calM} N) \ar[d]^{\cong} \\
\mathrm{Gr}_{\calF}(\calH(M)) \otimes \mathrm{Gr}_{\calF}(\calH(N)) \ar[r]^-{\mathrm{Gr}(\times)} & \mathrm{Gr}_{\calF}(\calH(M \ot_{\calM} N))
}
\]
commutes for all $M,N \in \calM$.
\end{enumerate}
\end{definition}

The following easy result establishes the connection between the formalism of functorial multiplicative spectral sequences  and the classical setup. 
\begin{proposition}\label{prop: classical vs functorial}
Let a  functorial multiplicative spectral sequence $\{ (\calE_{r}, -\times_{r}-, \rmd_{r}) \}_{r \geq a}$  converge  multiplicatively  to a filtered graded lax monoidal functor $\calH : \calM \to \mathrm{FGVect}_{\bfk}$. 
    Let  $A=(A, \mu_A, \eta_A)$ be an algebra object in $\calM$. Then for each $r\geq a$,  the composition 
    $$\calE_r(A)\ot \calE_r(A)\xrightarrow{-\times_r-} \calE_r(A\ot_\calM A)\xrightarrow{\calE_r(\mu_A)} \calE_r(A)$$
    together with $\rmd_r(A)$ 
    endows $\calE_r(A)$ with a structure of differential bigraded algebra. Similarly, $\calH(A)$ has a natural filtered graded algebra structure.  
    
   More  importantly,  the spectral sequence 
    $\calE_r(A)$ converges multiplicatively  to $\calH(A)$ in the sense of Definitions~\ref{def: multiplicative spectral sequence} and \ref{definition: Classical multiplicative convergence}. 
\end{proposition}

\begin{theorem}\label{thm:construction}
Let $\calC : (\calM, \ot_{\calM}, 1) \to (\mathrm{FDGVect}_{\bfk}, \otimes, \underline{\bfk})$ be a filtered differential graded lax monoidal functor. Assume that for each $n \in \bbZ, M\in \calM$, the filtration on $\calC^{n}(M)$ is bounded. 
Then there exists a functorial  multiplicative  spectral sequence $\{ (\calE_{r}, \times_{r}, \rmd_{r}) \}_{r \geq 0}$ with
\[
\calE_{1}^{p,q}(M) = \rmH^{p+q}(\calF^{p}\calC(M) / \calF^{p+1}\calC(M)),
\]
which converges multiplicatively to the filtered graded lax monoidal functor $\rmH^{*}(\calC)$.
\end{theorem}

\begin{proof}[Sketch of proof]
For each $M \in \calM$, the bounded filtration on $\calC(M)$ yields a spectral sequence in the usual way. The lax monoidal structure $-\times-$ on $\calC$ induces, at each page $r$, an exterior product    $-\times_{r}-$ on $\calE_{r}$ because the exterior product is filtered (condition \eqref{eq:filtered-constraint}) and compatible with differentials (condition \eqref{eq:leibniz}). The passage from $\calE_{r}$ to $\calE_{r+1}$ via cohomology preserves the lax monoidal structure by Proposition~\ref{prop:cohomology-lax}. At $\calE_{\infty}$, the induced structure agrees with that on $\mathrm{Gr}_{\calF}(\rmH^{*}(\calC))$ by the compatibility of $\times$ with the filtration.
\end{proof}

A  bicomplex lax monoidal functor  from $\calM$ to $\Vect$ gives rise to two natural filtered differential graded lax monoidal functors  from $\calM$ to $\Vect$ via the column and row filtrations.  

Since taking total complexes is a lax monoidal functor from the category of bicomplexes to that of differential graded vector spaces, for a  bicomplex lax monoidal functor $\calC : (\calM, \ot_{\calM}, 1) \to (\mathrm{BiCom}_{\bfk}, \otimes, \underline{\bfk})$, 
 $\Tot(\calC)$ is a differential graded lax monoidal functor.
\begin{corollary}
Given a  bicomplex lax monoidal functor   $\calC : (\calM, \ot_{\calM}, 1) \to (\mathrm{BiCom}_{\bfk}, \otimes, \underline{\bfk})$, the two spectral sequences induced by the standard column and row filtrations are multiplicative functorial spectral sequences in the sense of Definition~\ref{def:Functorial  multiplicative  spectral sequence}, and converge multiplicatively to $\rmH(\Tot(\calC))$ in the sense of Definition~\ref{def:Multiplicative convergence}.
\end{corollary}

From now on, in order to avoid the heavy machinery, a spectral sequence satisfying Definitions \ref{def:Functorial  multiplicative  spectral sequence} and \ref{def:Multiplicative convergence}
will be called a \textbf{functorial multiplicative convergent spectral sequence}.

\begin{remark}[Historical remark]  The first examples of  spectral sequences   are all multiplicative spectral sequences, hence also called  spectral rings in the literature. 
Structural explanations of the multiplicative phenomena have been developed.     Cartan-Eilenberg systems \cite[Chapter XV Section 7]{CE56} provide  a first approach. In fact, Douady \cite{Douady59} showed that a pairing of Cartan-Eilenberg systems   gives rise to a pairing of the associated spectral
sequences. 

Another approach to spectral sequences is via exact couples \cite{Massey52, Massey53, Massey54}. Massey's  paper  \cite{Massey54} explicitly addressed the multiplicative
structure by using exact couples.

In recent years, the development of higher category theory and $\infty$-operads has provided a
conceptual framework that explains this phenomenon systematically.   The $\infty$-category of towers (filtrations) of spectra  carries a symmetric monoidal structure via Day convolution \cite{Day70} generalised to this setup by Lurie \cite{LurieHA} and  Glasman \cite{Glasman16}.  
As shown by
Hedenlund \cite{Hedenlund20} (see also  de Potter \cite{dePotter23}) in her thesis,   the functor   extracting a spectral sequence from the $\infty$-category of towers (filtrations) of spectra  is a lax monoidal functor.  

It seems that our treatment of functorial multiplicative spectral
sequences is different from  their framework by combining  a family of   spectral
sequences indexed by a monoidal category. Moreover, the source monoidal category we have in mind is the category of bimodules over a ring, which is NOT symmetric.

\end{remark}
\bigskip

\section{Constructing cup products in cohomology theories}\label{section: Constructing Cup products in cohomology theories}

Since we need to consider the multiplicative property of some spectral sequences, in this section, constructions of cup products in Hochschild cohomology of an associative algebra and cohomology of a Hopf algebra are recalled. 

\subsection{Cup products in Hochschild cohomology
}

Let $R$ be an associative $\bfk$-algebra. The two-sided bar resolution ${B_\bullet}:={B_\bullet}(R;R;R)$ is given as follows:
$B_n=R\ot R^{\ot n}\ot R$ for $n\geq 0$, and the differential
$b_n: {B}_n\to {B}_{n-1}$ for $n\geq 0$ sends $a_0\ot a_1\ot\cdots \ot a_n\ot a_{n+1}$ to
$\sum_{i=0}^n (-1)^i a_0\ot \cdots\ot a_ia_{i+1}\ot \cdots \ot a_{n+1}$, where by convention, ${B}_{-1}=R$. The map $\epsilon^{B_\bullet}:=b_0$ establishes a free bimodule resolution of $R$, which is exact, because of the existence of a contracting homotopy:
$s_n: {B}_{n}\to {B}_{n+1}$ with $n\geq -1$ sending $a_0\ot a_1\ot\cdots \ot a_n\ot a_{n+1}$ to $(-1)^{n+1}a_0\ot a_1\ot\cdots \ot a_n\ot a_{n+1}\ot 1$.
Note that all $s_n$ are morphisms of left $R$-modules.
Hence, for a right $R$-module $M$, the one-sided bar resolution ${B_\bullet}(M;R;R):=M\ot_R {B_\bullet}$ is still exact and is thus a free resolution of $M$. Similarly, there exists a contracting homotopy consisting of morphisms of right $R$-modules and for a left $R$-module $N$, the one-sided bar resolution ${B_\bullet}(R;R;N):= {B_\bullet}\ot_R N$ is still exact and is also a free resolution of $N$.

Recall that the Hochschild cohomology of an algebra $R$ with coefficients in a bimodule $U$ is defined to be
$\rmHH^\bullet(R, U)=\Ext^\bullet_{R^e}(R, U)$. One can use the two-sided bar resolution to compute the Hochschild cohomology groups.
Write $$\rmC^\bullet(R, U):=\Hom_{R^e}({B_\bullet}, U)\cong (\Hom(R^{\ot \bullet},  U), \rmd_U^\bullet)$$ which is called the Hochschild cochain complex of $R$ with coefficients in $U$, where the differential $\rmd_U^\bullet$ sends
$f: R^{\ot p}\to U$ to
\begin{align*}
\rmd_U^p(f): R^{\ot p+1}\to U, \quad  a_1\ot\cdots \ot a_{p+1}&\mapsto (-1)^{p+1} a_1f(a_2\ot \cdots \ot a_{p+1}) \\
&{}+\sum_{i=1}^p(-1)^{p+1-i} f(a_1\ot \cdots \ot a_ia_{i+1}\ot \cdots \ot a_{p+1})+f(a_1\ot \cdots \ot a_{p})a_{p+1}.
\end{align*}

Let $V$ be another bimodule. The exterior product on Hochschild cohomology groups
$$-\times^{p,q}-:=-\times_{U, V}^{p,q}-: \rmHH^p(R, U)\ot\rmHH^q(R, V)\to \rmHH^{p+q}(R, U\ot_R V)$$
with $p,  q \geq 0$ is induced from a chain level map: for any $p, q\geq 0$
$$-\times_{B_\bullet}^{p,q}-:=-\times_{B_\bullet; U, V}^{p,q}-: \rmC^p(R, U)\ot\rmC^q(R, V)\to \rmC^{p+q}(R, U\ot_R V)$$
which assigns to $f: R^{\ot p}\to U$ and $g: R^{\ot q}\to V$ the map
$$f\times^{p,q}_{B_\bullet}  g:   R^{\ot p+q}\to U\ot_R V, \qquad
a_1\ot\cdots\ot a_{p+q}\mapsto  (-1)^{pq} f(a_1\ot\cdots\ot a_{p})\ot_R g(a_{p+1}\ot\cdots\ot a_{p+q}).$$
The exterior product is compatible with the differentials, that is,
$$\rmd_U^p(f)\times_{B_\bullet}^{p+1,q} g+(-1)^p f\times_{B_\bullet}^{p,q+1} \rmd_V^q(g)=\rmd_{U\ot_R V}^{p+q}(f\times^{p,q}_{B_\bullet} g)\in \rmC^{p+q+1}(R, U\ot_R V).$$
Moreover, the exterior product is associative in the following sense: Let $W$ be a third bimodule. Then for $f: R^{\ot p}\to U$,  $g: R^{\ot q}\to V$
and $h: R^{\ot r}\to W$, we have  $$(f\times^{p,q}_{B_\bullet} g)\times^{p+q,r}_{B_\bullet} h=f\times^{p,q+r}_{B_\bullet} (g\times^{q,r}_{B_\bullet} h)\in  \rmC^{p+q+r}(R, U\ot_R V\ot_R W).$$

Let $\varphi: R\to S$ be a morphism of $\bfk$-algebras. The multiplication of $S$ induces a map $\tilde{\mu}_S: S\ot_R S\to S$. When $U=S=V$, the cup product:
$$-\cup^{p,q}_{B_\bullet}-:=(\tilde{\mu}_S)_*\circ (-\times^{p,q}_{B_\bullet}-): \rmC^p(R, S)\ot \rmC^q(R, S)\to \rmC^{p+q}(R, S\ot_R S)\to \rmC^{p+q}(R, S)$$ endows  $\rmC^\bullet(R, S)$ with a differential graded algebra structure and thus  $ \rmHH^*(R, S)=\rmH^*(\rmC^\bullet(R, S))$  is a graded  algebra.
Moreover, let $U$ be an $(S,S)$-bimodule and let $V=W=S$ (resp.\ let $W$ be an $(S,S)$-bimodule and let $U=V=S$). The exterior product endows $\rmC^\bullet (R, U)$ with a structure of   differential graded right (resp.\ left) module over  the differential graded algebra $\rmC^\bullet(R, S)$. Hence  $\rmC^\bullet (R, U)$ is made into a differential graded bimodule over $\rmC^\bullet(R, S)$. As a consequence, $\rmHH^*(R, U)$ is a  graded bimodule over  the   graded algebra $\rmHH^*(R, S)$.
In particular, when $R=S$,  $\rmHH^*(R):=\rmHH^*(R, R)$ is a    graded commutative algebra (\cite{Ger63}).

From the viewpoint of Subsection~\ref{subsec: Multiplicative functorial spectral sequences}, 
$\rmC^\bullet(R,-)$ is a differential graded lax monoidal functor from $R^e\Mod$ to $\Vect$ and hence its cohomology 
$\rmHH^*(R,-)$ is a   graded lax monoidal functor from $R^e\Mod$ to $\Vect$ by Proposition~\ref{prop:cohomology-lax}.

The exterior product and consequently the cup product  can be deduced from a diagonal map  on the  two-sided bar resolution:
$$\varDelta^{B_\bullet}: {B_\bullet}\to \Tot({B_\bullet}\ot_R {B_\bullet}) $$
which sends
$ a_0\ot \cdots \ot a_{n+1}$
to
$$\sum_{p=0}^{n} (a_0\ot \cdots \ot a_p\ot 1)\otimes_R (1\ot a_{p+1}\ot \cdots \ot a_{n+1}).$$
It is easy to verify that $({B_\bullet}, b_\bullet, \varDelta^{B_\bullet})$ is a differential graded coalgebra over $R$.

The diagonal map $\varDelta^{B_\bullet}: {B_\bullet}\to \Tot({B_\bullet}\ot_R {B_\bullet}) $ can be described componentwise. In fact, for $p, q\geq 0$, write
$$\varDelta^{B_\bullet}_{p, q}: {B_{p+q}}\to  {B_p}\ot_R {B_q},\quad   a_0\ot \cdots \ot a_{p+q+1}\mapsto  (a_0\ot \cdots \ot a_p\ot 1)\otimes_R (1\ot a_{p+1}\ot \cdots \ot a_{p+q+1}).$$
Then $$\varDelta^{B_\bullet}_n=\sum_{p+q=n}\varDelta^{B_\bullet}_{p, q}: {B_n}\to \Tot({B_\bullet}\ot_R {B_\bullet})_n
=\bigoplus_{p+q=n} B_p\ot_R B_q. $$
The coassociativity of $\varDelta^{B_\bullet}$ can be expressed as: for any $p, q, r\geq 0$,
$$(\varDelta^{B_\bullet}_{p, q}\ot_R \id_{B_r})\circ \varDelta^{B_\bullet}_{p+q, r}=(\id_{B_p}\ot_R\varDelta^{B_\bullet}_{q,r} )\circ \varDelta^{B_\bullet}_{p, q+r}: B_{p+q+r}\to  B_p\ot_R B_q\ot_R B_r.$$

Moreover, we can see  that  for  $f\in \rmC^p(R, U)$ and $g\in  \rmC^q(R, V)$,
$$f\times^{p,q}_{B_\bullet} g=(-1)^{pq} (f\ot_R g)\circ \varDelta^{B_\bullet}_{p, q}\in \rmC^{p+q}(R, U\ot_R V).$$
Hence,  the properties mentioned above of the exterior product  and consequently of the cup product follow.
In particular, the associativity of the exterior product  and consequently of the cup product follows from the coassociativity of $\varDelta^{B_\bullet}$.     

Now let $P_\bullet\stackrel{\epsilon^{P_\bullet}}{\to} R$  be another projective bimodule resolution of $R$. Then by the Comparison Lemma, there exists a homotopy equivalence $\varphi_\bullet: P_\bullet\to {B_\bullet}$ lifting the identity of $R$, with homotopy inverse $\psi_\bullet: {B_\bullet}\to {P_\bullet}$. Since ${P_\bullet}$ splits as a complex of left $R$-modules and as a complex of right $R$-modules, $\Tot({P_\bullet}\ot_R {P_\bullet})$ is also a projective bimodule resolution of $R$. There exists a diagonal map $\varDelta^{P_\bullet}: {P_\bullet}\to \Tot({P_\bullet}\ot_R {P_\bullet})$ lifting the identity of $R$.
Remark that by the Comparison Lemma, the diagonal map $\varDelta^{P_\bullet}: {P_\bullet}\to \Tot({P_\bullet}\ot_R {P_\bullet})$ is only coassociative up to homotopy.

Similarly, the  diagonal map $\varDelta^{P_\bullet}: {P_\bullet}\to \Tot({P_\bullet}\ot_R {P_\bullet})$  can also be described componentwise. In fact, for $p, q\geq 0$, write
$$\varDelta^{P_\bullet}_{p, q}: {P_{p+q}}\to  {P_p}\ot_R {P_q}$$
as the composition of  $\varDelta^{P_\bullet}_{p+q}: {P_{p+q}}\to \Tot({P_\bullet}\ot_R {P_\bullet})_{p+q}$ with the natural projection $\Tot({P_\bullet}\ot_R {P_\bullet})_{p+q}\to P_p\ot_R P_q$. Then for any $n\geq 0$, $$\varDelta^{P_\bullet}_n=\sum_{p+q=n}\varDelta^{P_\bullet}_{p, q}: {P_n}\to \bigoplus_{p+q=n} P_p\ot_R P_q = \Tot({P_\bullet}\ot_R {P_\bullet})_n $$

Furthermore, still by the Comparison Lemma, the diagram
$$\begin{tikzcd}
    {P_\bullet} \arrow[r, "\varDelta^{P_\bullet}"] \arrow[d, "\varphi_\bullet"'] & \Tot({P_\bullet} \otimes_R {P_\bullet}) \arrow[d, "\Tot(\varphi_\bullet \otimes_R \varphi_\bullet)"] \\
   {B_\bullet} \arrow[r, "\varDelta^{B_\bullet}"]   & \Tot({B_\bullet} \otimes_R {B_\bullet}).
\end{tikzcd}$$
is commutative up to homotopy. This shows that if for $f\in \Hom_{R^e}({P}_p, U)$ and $g\in \Hom_{R^e}({P}_q, V)$,  we define a new exterior product:
$$f\times^{p,q}_{P_\bullet} g=(-1)^{pq} (f\ot_R g)\circ \varDelta^{P_\bullet}_{p, q}: P_{p+q}\to U\ot_R V,$$ it coincides with the usual exterior product $(f\circ \psi_p)\times^{p,q}_{B_\bullet} (g\circ \psi_q)$ on the  level of cohomology groups, although not at the complex level.

\subsection{Cup products in the cohomology theory of a Hopf algebra}\label{subsection: Cup products in the cohomology theory of a Hopf algebra}

Next let us recall the construction of cup product on cohomology groups of a Hopf algebra.
Let $H=(H, \mu, \eta, \Delta, \epsilon, S)$ be a Hopf algebra, where $\mu: H\ot H\to H$, $\eta: \bfk\to H$, $\Delta: H\to H\ot H$, $\epsilon: H\to \bfk$, $S: H\to H$ are the multiplication, the unit, the comultiplication, the counit and the antipode, respectively.
Then $\bfk$ is a right $H$-module via $\epsilon$. Let $X$, $Y$ be right $H$-modules. Then $X\ot Y$ is a right $H$-module via $(x\ot y)\cdot h=\sum xh_{(1)}\ot y h_{(2)}$ for $x\in X, y\in Y, h\in H$,  where $\Delta(h)=\sum h_{(1)}\ot h_{(2)}$.

The usual cohomology of $H$ with coefficients in $X$ is defined to be
$$\rmH^*(H, X)=\Ext^*_H(\bfk, X).$$
Since $B_\bullet'=({B_\bullet}(\bfk; H; H), b')\xrightarrow{\pi^{B_\bullet'}}\bfk$ is a free resolution of  the trivial $H$-module $\bfk$, the usual cohomology of $H$ with coefficients in $X$:
  $\rmH^*(H, X)=\Ext^*_H(\bfk, X) $
  can be computed by the cochain complex:
  $$\rmC^\bullet(H, X)=\Hom_H({B_\bullet'}, X)\cong (\Hom(H^{\bullet}, X), \partial^\bullet_X),$$
  where $\partial^\bullet_X$ sends $f\in \Hom(H^{\ot p}, X)$ to
  \begin{align*}
  \partial^p_X(f): H^{\ot p+1}\to X,\quad h_1\ot \cdots \ot h_{p+1} &\mapsto  (-1)^{p+1} \epsilon(h_1)f(h_2\ot \cdots \ot h_{p+1}) \\
  &{}+\sum_{i=1}^p (-1)^{p+1-i} f(h_1\ot \cdots\ot h_ih_{i+1} \ot \cdots \ot h_{p+1})+f(h_1\ot \cdots\ot h_{p})h_{p+1}.
  \end{align*}

The exterior product on  the usual cohomology of $H$:
$$\star^{p,q}:=\star^{p,q}_{X, Y}: \rmH^p(H, X)\ot\rmH^q(H, Y)\to \rmH^{p+q}(H, X\ot  Y)$$
with $p,  q\geq 0$ is induced from a chain  map: for any $p, q\geq 0$
$$\star^{p,q}_{B_\bullet'}:=\star^{p,q}_{B_\bullet'; X, Y}:  \rmC^p(H, X)\ot\rmC^q(H, Y)\to \rmC^{p+q}(H, X\ot  Y),$$
which assigns  $f: H^{\ot p}\to X$ and $g: H^{\ot q}\to Y$ to the map
\begin{align*}
f\star^{p,q}_{B_\bullet'} g:   H^{\ot p+q}&\to X\ot  Y, \\
h_1\ot\cdots\ot h_{p+q}&\mapsto  (-1)^{pq} f(h_1\ot\cdots\ot h_{p})(h_{p+1})_{(1)}\cdots(h_{p+q})_{(1)}\ot g((h_{p+1})_{(2)}\ot\cdots\ot (h_{p+q})_{(2)}).
\end{align*}

The exterior product is compatible with the differentials, that is,
$$\partial_X^p(f)\star^{p+1,q}_{B_\bullet'} g+(-1)^p f\star^{p,q+1}_{B_\bullet'} \partial_Y^q(g)=\partial_{X\ot Y}^{p+q}(f\star^{p,q}_{B_\bullet'} g):   H^{\ot p+q+1}\to X\ot Y.$$
Moreover, the exterior product is associative in the following sense: Let $Z$ be a third  $H$-module. Then for $f: H^{\ot p}\to X$,  $g: H^{\ot q}\to Y$
and $h: H^{\ot r}\to Z$, we have  $$(f\star^{p,q}_{B_\bullet'} g)\star^{p+q,r}_{B_\bullet'} h=f\star^{p,q+r}_{B_\bullet'} (g\star^{q,r}_{B_\bullet'} h)\in  \rmC^{p+q+r}(H, X\ot  Y\ot  Z).$$

Let $A=(A, \mu_A, \eta_A, \gamma_A)$ be a right  $H$-module algebra, where $\mu_A: A\ot A\to A$, $\eta_A: \bfk\to A$, $\gamma_A:  A\ot H\to A$ are the multiplication, the unit and the right $H$-module structure on $A$, respectively.
When $X=A=Y$, the family:
$$-\smallsmile^{p,q}_{B_\bullet'}-:=(\mu_A)_*\circ (-\star^{p,q}_{B_\bullet'}-): \rmC^p(H, A)\ot \rmC^q(H, A)\to \rmC^{p+q}(H, A\ot  A)\to \rmC^{p+q}(H, A), \quad p,q\geq 0$$ forms an associative product on $\rmC^\bullet(H, A)$,  called \textbf{cup product},  such that   $\rmC^\bullet(H, A)$ becomes a differential graded algebra  and thus  $ \rmH^*(H, A)=H^*(\rmC^\bullet(H, A))$  is a graded  algebra.
Moreover, let $Y=A$ and let $X=Z$ be an $(A,A)$-bimodule whose structure morphisms $A\otimes X\to X$ and $X\otimes A\to X$ are $H$-linear. The exterior product endows $\rmC^\bullet (H, X)$ with a structure of   differential graded bimodule over  the differential graded algebra $\rmC^\bullet(H, A)$.  As a consequence, $\rmH^*(H, X)$ is a    graded bimodule over  the   graded algebra $\rmH^*(H, A)$. In particular, when $A=\bfk$,  $\rmH^*(H):=\rmH^*(H, \bfk)$ is a    graded commutative algebra.

From the viewpoint of Subsection~\ref{subsec: Multiplicative functorial spectral sequences}, 
$\rmC^\bullet(H,-)$ is a differential graded lax monoidal functor from $H\Mod$ to $\Vect$ and hence its cohomology 
$\rmH^*(H,-)$ is a   graded lax monoidal functor from $H\Mod$ to $\Vect$ by Proposition~\ref{prop:cohomology-lax}.

In fact, the exterior product and the cup product can be deduced from a diagonal map on the  one-sided bar resolution ${B_\bullet'}$ defined as follows:
$$\varDelta^{B_\bullet'}: {B_\bullet'}\to \Tot({B_\bullet'}\ot {B_\bullet'}) $$
sending
$ h_1\ot \cdots \ot h_{n+1}$
to
$$\sum_{p=0}^{n} \big(h_1\ot \cdots \ot h_p\ot (h_{p+1})_{(1)}\cdots (h_{n+1})_{(1)} \big)\otimes  \big( (h_{p+1})_{(2)}\ot \cdots\ot  (h_{n+1})_{(2)}\big).$$
It is easy to verify that $({B_\bullet'}, b', \varDelta^{B_\bullet'})$ is a differential graded coalgebra.

The diagonal map $\varDelta^{B_\bullet'}: {B_\bullet'}\to \Tot({B_\bullet'}\ot {B_\bullet'}) $ can be described componentwise as well. In fact, for $p, q\geq 0$, write
$$\varDelta^{B_\bullet'}_{p, q}: {B'_{p+q}}\to  {B'_p}\ot {B'_q},  h_1\ot \cdots \ot h_{p+q+1}\mapsto  \big(h_1\ot \cdots \ot h_p\ot (h_{p+1})_{(1)}\cdots (h_{p+q+1})_{(1)} \big)\otimes  \big( (h_{p+1})_{(2)}\ot \cdots\ot  (h_{p+q+1})_{(2)}\big).$$
Then $$\varDelta^{B_\bullet'}_n=\sum_{p+q=n}\varDelta^{B_\bullet'}_{p, q}: {B_n'}\to \Tot({B_\bullet'}\ot {B_\bullet'})_n. $$
The coassociativity of $\varDelta^{B_\bullet'}$ can be expressed as: for any $p, q, r\geq 0$,
$$(\varDelta^{B_\bullet'}_{p, q}\ot  \id_{B_r'})\circ \varDelta^{B_\bullet'}_{p+q, r}=(\id_{B_p'}\ot\varDelta^{B_\bullet'}_{q, r} )\circ \varDelta^{B_\bullet'}_{p, q+r}.$$

Moreover, we can see  that  for  $f\in \rmC^p(H, X)$ and $g\in  \rmC^q(H, Y)$,
$$f\star^{p,q}_{B_\bullet'} g=(-1)^{pq} (f\ot  g)\circ \varDelta^{B_\bullet'}_{p, q}:  H^{p+q}\to  X\ot Y.$$
Hence,  the properties mentioned above of $\star$  and consequently of $\smallsmile$ follow.

Now let ${Q_\bullet} \xrightarrow{\pi^{Q_\bullet}} \bfk$  be another projective   resolution of the trivial $H$-module $\bfk$. Then by the Comparison Lemma, there exists a homotopy equivalence $\phi_\bullet: {Q_\bullet}\to {B_\bullet'}$ lifting the identity of $\bfk$, with homotopy inverse $\tau_\bullet: {B_\bullet'}\to {Q_\bullet}$. By the K\"unneth formula, $\Tot({Q_\bullet}\ot  {Q_\bullet})$ is also a projective   resolution of $\bfk$. There exists a diagonal map $\varDelta^{Q_\bullet}: {Q_\bullet}\to \Tot({Q_\bullet}\ot  {Q_\bullet})$ lifting the identity of $\bfk$.

Remark that the diagonal map $\varDelta^{Q_\bullet}: {Q_\bullet}\to \Tot({Q_\bullet}\ot {Q_\bullet})$ is NOT coassociative in general, but only coassociative up to homotopy.  In fact, this follows from the Comparison Lemma,  as both  $(\id_{Q_\bullet}\ot \varDelta^{Q_\bullet})\circ\varDelta^{Q_\bullet}$ and $(\varDelta^{Q_\bullet}\ot \id_{Q_\bullet} )\circ\varDelta^{Q_\bullet} $   lift the canonical isomorphism
$\bfk \to \bfk\ot \bfk \ot  \bfk$.

Similarly, the  diagonal map $\varDelta^{Q_\bullet}: {Q_\bullet}\to \Tot({Q_\bullet}\ot {Q_\bullet})$  can also be described componentwise. In fact, for $p, q\geq 0$, write
$$\varDelta^{Q_\bullet}_{p, q}: {Q_{p+q}}\to  {Q_p}\ot {Q_q}$$
as the composition of  $\varDelta^{Q_\bullet}_{p+q}: {Q_{p+q}}\to \Tot({Q_\bullet}\ot  {Q_\bullet})_{p+q}$ with the natural projection $\Tot({Q_\bullet}\ot  {Q_\bullet})_{p+q}\to Q_p\ot  Q_q$. Then for any $n\geq 0$, $$\varDelta^{Q_\bullet}_n=\sum_{p+q=n}\varDelta^{Q_\bullet}_{p, q}: {Q_n}\to \bigoplus_{p+q=n} Q_p\ot Q_q =\Tot({Q_\bullet}\ot  {Q_\bullet})_n. $$

Furthermore, by the Comparison Lemma, the diagram
$$\begin{tikzcd}
    {Q_\bullet} \arrow[r, "\varDelta^{Q_\bullet}"] \arrow[d, "\phi_\bullet"'] & \Tot({Q_\bullet} \otimes  {Q_\bullet})\arrow[d, "\Tot(\phi_\bullet \otimes  \phi_\bullet)"] \\
   {B_\bullet'} \arrow[r, "\varDelta^{{B_\bullet'}}"]   & \Tot({B_\bullet'} \otimes  {B_\bullet'}).
\end{tikzcd}$$
is commutative up to homotopy. This shows that if for $f\in \Hom_{H}({Q}_p, X)$ and $g\in \Hom_{H}({Q}_q, Y)$,  we define a new exterior product by the family:
$$ f\star^{p,q}_{Q_\bullet} g=(-1)^{pq} (f\ot  g)\circ \varDelta^{{Q_\bullet}}_{p, q}: Q_{p+q}  \to X\ot Y,$$ it coincides with the usual exterior product $ (f\circ \tau_p)\star^{p,q}_{B_\bullet'} (g\circ \tau_q) $ on the  level of cohomology groups, although not at the complex level.

\bigskip

\section{Hopf Galois extensions}\label{section: Hopf Galois extensions}

From now on, let $(H,\mu,\eta,\Delta,\epsilon,S)$ be a Hopf $\bfk$-algebra with a bijective antipode $S$.

\subsection{Basics about Hopf Galois extensions}

\begin{definition}
We call $B/A$ a \textbf{right $H$-extension} if $B$ is a right $H$-comodule algebra with the structure map $$\rho:B\rightarrow B\otimes H,~b\mapsto \sum b_{(0)}\otimes b_{(1)}$$ and $A$ is the coinvariant subalgebra of $B$, namely $$A=B^{\coH}=\{b\in B ~|~ \rho(b)=b\otimes 1\}.$$

Associated with a right $H$-extension $B/A$, there is a canonical \textbf{Galois map}
$$\beta : B \otimes_A B \longrightarrow B \otimes H,~~~
b \otimes_A b' \longmapsto \sum b b'_{(0)} \otimes b'_{(1)}$$
The right $H$-extension $B/A$ is said to be
\begin{myenumerate}
    \item \textbf{flat} if $B$ is flat over $A$ as a left and a right $A$-module;
    \item \textbf{Galois} if the   map $\beta$ is bijective.
\end{myenumerate}

Given a right $H$-Galois extension $B/A$, for   any $h\in H$, write $$\beta^{-1}(1\otimes h)=\sum_i l_i(h)\otimes_A r_i(h).$$

Left $H$-(Galois) extensions can be defined analogously.
\end{definition}

Unless otherwise specified, Hopf Galois extensions  refer to right Hopf Galois extensions.

\begin{lemma}[See, for instance, {\cite[1.1 Remark]{Sch90}}] \label{lemma:betaisamodulemap}
Let $B/A$ be an $H$-extension.
 The space   $B\otimes_A B$ has the left $B$-module structure  given by
   $$\ol{b}\cdot (b\otimes_A b')=\ol{b}b\otimes_A b',$$
   the right $B$-module structure   given by
   $$(b\otimes_A b')\cdot \ol{b}=b\otimes_A b'\ol{b},$$
  and $B \otimes H$ has  the left $B$-module structure   given by
   $$\ol{b}\cdot (b\ot h)=\ol{b}b\ot h,$$
   the right $B$-module structure   given by
   $$(b\ot h)\cdot \ol{b}=\sum b \ol{b}_{(0)}\ot h \ol{b}_{(1)},$$
   for any $\ol{b}, b, b'\in B$ and $h\in H$.
  Then the Galois map  $\beta : B \otimes_A B \longrightarrow B \otimes H$ is a morphism of left and right $B$-modules.
   It is also a morphism of right $H$-comodules where
the right $H$-comodule structure on $B\ot_A B$ is given by
$\id_B\ot_A \rho$,  and  the right $H$-comodule structure on $B\ot H$ is given by $\id_B\ot \Delta$.
\end{lemma}

\begin{remark}
    If $B$ is a right $H$-comodule algebra, then  $B^\op$ is a left $H^{\op,\cop}$-comodule algebra. The map $S^{-1}$ gives an isomorphism as Hopf algebras from $H^{\op,\cop}$ to $H$.  Hence,
    the opposite algebra $B^{\op}$ is a left $H$-comodule algebra  with the structure map $$\rho^{\op}:B^{\op}\rightarrow H\otimes B^{\op},~~ ~b^{\op}\mapsto  \sum S^{-1}(b_{(1)})\otimes (b_{(0)})^{\op},$$  where for $b\in B$, write $b^\op$ for the corresponding element in $B^\op$. Moreover, if $B/A$ is a right $H$-extension, we can show easily that $$A^{\op}={}^{\coH}(B^{\op}):=\{ b^{\op}\in B^{\op}\  |\  \rho^{\op}(b^{\op})=1\ot b^{\op}\}. $$  This implies that $B^{\op}/A^{\op}$ is a left $H$-extension.

    When $B/A$ is an $H$-Galois extension with Galois map $\beta: B\ot_A B \to B\ot H$, then
    define $$\beta^{\op} : B^\op \otimes_{A^\op} B^\op\rightarrow H\otimes B^\op,\ b^\op\otimes_{A^\op} {b'}^\op\mapsto \sum S^{-1}(b_{(1)})\otimes (b' b_{(0)})^\op,$$
   which is also bijective. Hence,  $B^\op/A^\op$ is a left $H$-Galois extension with Galois map $\beta^\op$.
\end{remark}

\begin{lemma}
     Let $B/A$ be an $H$-extension.
  Introduce
  $$\beta' : B^\op\otimes_{A} B\rightarrow H\otimes B,\ b^\op \otimes_A {b'}\mapsto \sum S^{-1}(b_{(1)})\otimes b_{(0)} b'.$$
  where $B^{\op}$ is regarded as a right $A$-module via $b^\op\cdot a=(ba)^\op$ for any $b^\op \in B^\op$ and $a\in A$.
  \begin{itemize}
\item[(a)]For $h\in H$, we have $${\beta'}^{-1}(h\ot 1)=\sum_i {l_i(h)}^\op\otimes_A r_i(h).$$
\item[(b)] The space $B^\op\ot_A B$  has the  left $B$-module structure given by
   $$\ol{b}\cdot (b^\op\ot_A b')=(\ol{b}b)^\op\ot_A b',$$
   the right $B$-module structure   given by
   $$(b^\op\ot_A b')\cdot \ol{b}=b^\op\ot_A b'\ol{b},$$ and  $H \otimes B$ has
   the left $B$-module structure   given by
   $$\ol{b}\cdot (h\ot b)=\sum hS^{-1}(\ol{b}_{(1)})\ot \ol{b}_0b,$$
   and the right $B$-module structure   given by
   $$(h\ot b)\cdot \ol{b}=h \ot b\ol{b},$$
   for $\ol{b}, b, b'\in B$ and $h\in H$.
 Then the map $\beta' : B^\op \otimes_A B \longrightarrow H \otimes B$ is a morphism of left and right $B$-modules.
\item[(c)] The map $\beta'$ is also a left $H$-comodule morphism where
the left $H$-comodule structure on $B^\op\ot_A B$ is given by
$\rho^\op\otimes_A \id_B$ and the left $H$-comodule structure on $H\ot B $ is given by $\Delta \ot\id_B$.
\item[(d)]  Define $$\tau: B\ot H\to H\ot B,\  b\ot h\mapsto \sum hS^{-1}(b_{(1)})\ot b_{(0)},$$
and  $$\tau': B^\op\otimes_A B \rightarrow B\otimes_A B,\ b^\op\otimes_A b'\mapsto b\otimes_A b'.$$
Then $\tau$ is bijective with inverse
$$\tau^{-1}: H\ot B\to B\ot H,\  h\ot b\mapsto \sum b_{(0)} \ot hb_{(1)},$$
and $\tau'$ is also bijective.
Moreover, $\tau\circ \beta\circ \tau'=\beta'$, hence $\beta$ is bijective if and only if $\beta'$ is bijective.
 \end{itemize}
\end{lemma}

\begin{proof}
     The statements are easy to prove and are left to the reader.
\end{proof}

\begin{proposition}
    Let $B/A$ be an $H$-Galois extension. One has the following well-known equalities:
    \begin{align}
        &\sum_i l_i(h)r_i(h)=\epsilon(h)1_B \label{beta1},\\
        &\sum_i (\id_B\otimes_A \rho)(l_i(h)\otimes_A r_i(h))=\sum_i l_i(h_{(1)})\otimes_A r_i(h_{(1)})\otimes h_{(2)}, \label{beta2}\\
        & \sum_i (\rho^\op\otimes \id_B)({l_i(h)}^\op\otimes_A r_i(h))=\sum_i h_{(1)}\otimes {l_i(h_{(2)})}^\op\otimes_A r_i(h_{(2)}), \label{beta3}\\
        &\sum_i l_i(hh')\otimes_A r_i(hh')=\sum_{i,j} l_i(h')l_j(h)\otimes_A r_j(h)r_i(h')\label{beta4} ,\\
        &\sum_i al_i(h)\otimes_A r_i(h)=\sum_i l_i(h)\otimes_A r_i(h)a, \label{beta5} \\
        &\sum_i (al_i(h))^\op \otimes_A r_i(h)=\sum_i l_i(h)^\op \otimes_A r_i(h)a, \label{beta5'}\\
        & \sum_i b_{(0)}l_i(b_{(1)})\otimes_A r_i(b_{(1)})=1\otimes_A b, \label{beta6}\\
        & \sum_i {l_i(S^{-1}(b_{(1)}))}^\op\otimes_A r_i(S^{-1}(b_{(1)}))b_{(0)}=b^\op \otimes_A 1, \label{beta7}
    \end{align}
    for all $h,h' \in H, b\in B, b^\op \in B^\op$ and $a\in A$.
\end{proposition}

\begin{proof}
    All equalities above except \eqref{beta3}, \eqref{beta7} and \eqref{beta5'} can be found in \cite[3.4 Remark]{Sch90}, and the proof of \eqref{beta3} (resp.  \eqref{beta7},\eqref{beta5'}) is similar to  that of \eqref{beta2} (resp. \eqref{beta6}, \eqref{beta5}).
\end{proof}

\subsection{The Ehresmann--Schauenburg bialgebroid}

In this subsection, we introduce the Ehresmann--Schauenburg bialgebroid \cite{Sch98} which plays an important role in this paper.

Bialgebroids were introduced by Takeuchi \cite{Tak77} (under the name of $\times_A$-bialgebras) and independently by Lu \cite{Lu96}, generalising bialgebras to the setting of a possibly noncommutative base algebra $A$.
From a categorical standpoint, a $\times_A$-bialgebra is nothing but a bialgebra object in the symmetric monoidal category of $(A^e,A^e)$-bimodules, equipped with the product $\times_{A}$.
Schauenburg \cite{Sch98} showed that an algebra over $A^e$ is a $\times_A$-bialgebra if and only if its module category is a monoidal category in such a way that the underlying functor to the category of $(A, A)$-bimodules is monoidal.
Important examples include the Ehresmann--Schauenburg bialgebroid of a Hopf--Galois extension, and universal enveloping algebras of Lie algebroids.

\begin{definition}[{\cite[DEFINITION 3.2]{Sch98}}]
For an $H$-Galois extension $B/A$, define the Ehresmann--Schauenburg bialgebroid to be
  $$\Gamma= \biggl\{\sum_p b_p \otimes {b'_p}^\op \in B^e\,\biggm
|\, \sum (b_p)_{(0)}\otimes {(b'_p)_{(0)}}^{\op} \otimes (b_p)_{(1)} (b'_p)_{(1)}=\sum_p b_p \otimes {b'_p}^{\op}\otimes 1_H  \in B\ot B^{\op}\ot H\biggr\}.$$
\end{definition}
In the literature \cite{Sch98, BW03, Zhu26}, the Ehresmann--Schauenburg bialgebroid  is usually denoted by $L(B, B, H)$.

It's straightforward to confirm that $\Gamma$ is a subalgebra of $B^e$ which contains $A^e$.
\begin{lemma}[{\cite[3.1 Lemma (1)]{Sch90Princialhomospace}}]
    There exists an isomorphism
    $$\Gamma\cong B\square_H B^{\op} :=\biggl\{\sum_p x_p\otimes y_p^{\op} \in B^e \,\biggm|\,\sum_p (x_p)_{(0)}\otimes (x_p)_{(1)}\otimes y_p^{\op}=\sum_px_p\otimes S^{-1}((y_p)_{(1)})\otimes {(y_p)_{(0)}}^{\op} \biggr\}.$$
\end{lemma}

Let $M$ be a left $\Gamma$-module. For any $\sum_p x_p \otimes {y_p}^\op \in \Gamma$ and $m \in M$, we sometimes write $\sum_p x_p m y_p$ instead of $\left(\sum_p x_p \otimes {y_p}^\op\right) \cdot m$ for convenience, although the $\Gamma$-module structure on $M$ is not necessarily obtained by a $(B, B)$-bimodule structure.

\begin{lemma} [{\cite[Theorem 3.1.(1)]{Sch98},  \cite[Lemma 2.2]{Wang2018}}]
   The algebra $A$ is a left $\Gamma$-module given by $(\sum_p x_p\otimes {y_p}^\op)\cdot a=\sum_p x_pay_p$ for any $\sum_p x_p\otimes {y_p}^\op\in \Gamma$ and $a\in A$.
\end{lemma}

\begin{proposition}[{\cite[Theorem 5.1 and 6.3]{Sch98}, \cite[Lemma 4.1]{LRW25}}]\label{prop: Gamma is a A-coring}
    $(\Gamma,\Delta_{\Gamma},\epsilon)$ is a coassociative $A$-coring, where for $\gamma=\sum_p x_p\ot y_p^\op \in \Gamma$,
    \[
    \Delta_{\Gamma}(\gamma)= \sum (x_p)_{(0)}\ot l_i((x_p)_{(1)})^\op \ot_A r_i((x_p)_{(1)}) \ot y_p^\op=\sum x_p\ot l_j(S^{-1}(y_p)_{(1)})^\op \ot_A r_j(S^{-1}(y_p))_{(1)}\ot (y_p)_{(0)}^\op \in \Gamma\ot_A \Gamma
    \] and 
    \[
    \epsilon_\Gamma (\gamma)=\sum_p x_py_p \in A.
    \]
    satisfying
    $\id=(\id\ot_A \epsilon_\Gamma)\circ \Delta_{\Gamma}=(\epsilon_\Gamma\ot_A \id) \Delta_{\Gamma}$ and $\Delta_\Gamma$ is an algebra morphism.

    Moreover, let $K$, $K'$ be two left $\Gamma$-modules, then $K\otimes_A K'$ is also a left $\Gamma$-module. Explicitly, for $  k\otimes_A k'\in K\otimes_A K'$ and $\sum_p x_p\otimes y_p^\op\in \Gamma$,  the action is given by
    \begin{align*}
   \biggl( \sum_p x_p\otimes y_p^\op\biggr)\cdot(k\otimes_A k') &=\sum_{i,p} (x_{p})_{(0)} k l_i \bigl((x_p)_{(1)}\bigr)\otimes_A r_i\bigl((x_p)_{(1)}\bigr)k' y_p \\
   &=\sum_{j,p} x_p k l_j\Bigl(S^{-1}\bigl((y_p)_{(1)}\bigl)\Bigr)\otimes_A r_j \Big(S^{-1}\bigl((y_p)_{(1)}\bigr)\Bigr)k' (y_p)_{(0)}.
    \end{align*}
    
    In fact, $(\Gamma\Mod ,-\otimes_A -,A)$ is a monoidal category. Thus, for three left $\Gamma$-modules $K, K', K''$, we may unambiguously denote either of left $\Gamma$-modules $(K\ot_A K')\ot_A K''\cong K\ot_A (K'\ot_A K'')$ by $ K\ot_A K'\ot_A K''$.
\end{proposition}
We also use Sweedler's notation for the $A$-coring. That is, we denote $\Delta_\Gamma(\gamma)=\sum \gamma_{(1)}\ot_A \gamma_{(2)}$.

Next, we  describe  $\Gamma$  in some concrete examples of Hopf Galois extensions.

\begin{example}[\textbf{Strongly graded algebras}]
 Let $G$ be a group with identity element $e$, $H$ be the Hopf algebra $\bfk G$ and $B=\oplus_{g\in G}B_g$ be a $G$-graded right $H$-comodule  algebra  with the comodule map $$\rho:B\rightarrow B\otimes \bfk G,~b\mapsto \sum_g b_g\otimes g,$$
 where $b=\sum_g b_g\in B=\oplus_{g\in G}B_g.$
Then the $H$-coinvariant subalgebra is $A=B^{\coH}= B_e$.
It's well known that $B/A$ is an $H$-extension, and it's an $H$-Galois extension if and only if $B$ is strongly graded, that is, $B_gB_h=B_{gh},~ \forall g,h\in G$. (\cite{Ul81})

In this case, $\Gamma=\bigoplus_{g\in G} B_g\otimes B_{g^{-1}}^\op$.
\end{example}

\begin{example}[\textbf{Smash product algebras}] \label{example: smash product is H-Galois extension}
    Let $H$ be a Hopf algebra, $A$ be an $H$-module algebra, and $B=A\# H$ be the smash product algebra, whose
    multiplication is given by
    $$(a\# h)(b\#k)=\sum a(h_{(1)} b) \# h_{(2)} k,$$
    for $a, b\in A, h, k\in H$.
    And $B$ is a comodule algebra with the  right $H$-comodule map $\rho$, where
    $$
    \rho:A\# H\rightarrow (A\# H)\otimes H,~a\# h \mapsto \sum (a\# h_{(1)})\otimes h_{(2)}
    $$
    for $a\in A, h\in H$.
    It's well known that $B/A$ is an $H$-Galois extension.

    By using the fact: $$(H \otimes H)^{\coH}=\Span_\bfk\{ \sum h_{(1)}\otimes S(h_{(2)}) | h\in H\},$$ one has
    \begin{equation}\label{smashgamma1}
        \Gamma=\Span_\bfk \{ \sum (a\# h_{(1)})\otimes \bigl(a'\#  S(h_{(2)})\bigr)^{\op}\ |\  a,a'\in A,h\in H\}.
    \end{equation}
    Another useful equivalent form of $\Gamma$ is
    \begin{equation}\label{smashgamma2}
        \Gamma=\Span_\bfk \{ \sum (a\# h_{(1)})\otimes \bigl(S(h_{(3)})\rightharpoonup a'\# S(h_{(2)})\bigr)^{\op}\ |\  a,a'\in A,h\in H\}.
    \end{equation}

    A third description of $\Gamma$ is given by Burciu and   Witherspoon \cite{BW07}.
    Let $\mathcal{D}=A^e\otimes H$ be a $\bfk$-algebra whose  multiplication  is given by:
    \[
    \bigl((a\otimes a'^\op)\otimes h\bigr)\cdot \bigl((\tilde{a}\otimes (\tilde{a'})^\op)\otimes h'\bigr)=a(h_{(1)}\rightharpoonup \tilde{a})\otimes \bigl((h_{(3)}\rightharpoonup \tilde{a'})a'\bigr)^\op\otimes h_{(2)} h',
    \]
    for $a,a',\tilde{a},\tilde{a'}\in A$ and $h,h'\in H$.
    Then $\mathcal{D}$ is isomorphic to $ \Gamma$ via
    \[
    \varphi: \mathcal{D}=A\otimes A^\op \otimes H\rightarrow \Gamma,~ a\otimes a'^\op \otimes h \mapsto \sum (a\# h_{(1)})\otimes \Bigl(\bigl(S(h_{(3)})\rightharpoonup a'\bigr)\# S(h_{(2)})\Big)^\op,
    \]
   with inverse
    \[
    \psi: \Gamma \rightarrow A\otimes A^\op \otimes H=\mathcal{D},~\sum (a\# h_{(1)})\otimes \bigl(a'\# S(h_{(2)})\bigr)^\op \mapsto \sum a\otimes (h_{(2)}\rightharpoonup a')^\op \otimes h_{(1)},
    \]
    where we use the formula \eqref{smashgamma2} as the definition of $\Gamma$.
\end{example}

\begin{example}[\textbf{Crossed product algebras}]
Let  $A$ be a  $\bfk$-algebra  such that $H$ measures $A$.
Let  $\sigma \in \Hom_\bfk(H\ot H,A)$ be an invertible map with respect to the convolution product.  Assume that
$A$ is a twisted $H$-module and that  $\sigma$ is a Hopf 2-cocycle; for precise definitions, see \cite[7.1.2. Lemma]{Mont93}.
The crossed product $B:=A\#_\sigma H$ is defined to be  $A\ot H$ as a vector space whose multiplication is given by
    $$(a\#_\sigma h)(b\#_\sigma k)=\sum a(h_{(1)} b)\sigma(h_{(2)},k_{(1)}) \#_\sigma h_{(3)} k_{(2)},$$
    for $a, b\in A, h, k\in H$.
  Then   $B$ is a comodule algebra with    right $H$-comodule map
    $$
    \rho:A\#_\sigma H\rightarrow (A\#_\sigma H)\otimes H,~a\#_\sigma h \mapsto \sum (a\#_\sigma h_{(1)})\otimes h_{(2)},
    $$
    for $a\in A, h\in H$.
    It's well known that $B/A$ is an $H$-Galois extension, whose Galois map is given by
    $$\beta:  B\ot_A B\to B\ot H, ~(a\#_\sigma h) \otimes_A (b\#_\sigma k) \mapsto \sum a(h_{(1)} b)\sigma(h_{(2)},k_{(1)}) \#_\sigma h_{(3)} k_{(2)} \otimes k_{(3)},$$
    see, for instance, \cite[7.2.2. Theorem and 8.2.5 Corollary]{Mont93}.

   When $A$ is an $H$-module algebra and $\sigma\in \Hom_\bfk(H\ot H,A)$ is the unity of the convolution product, say,
   $\sigma=\eta_A\circ (\epsilon\ot \epsilon)$, where $\eta_A$ is the unit of $A$, crossed products reduce to smash products.
   However,  when $A$ is a twisted $H$-module algebra,  due to the right normal basis property of  the crossed product $A\#_\sigma H$ (\cite[(2.6) Definition] {KT81},\cite[Theorem 11]{DT86},\cite[Definition 4.1]{BCM86}), that is,   $A\#_\sigma H$ has the same underlying space and the same right $H$-comodule structure as the smash product $A\# H$, with the same proof,  one also has
    \begin{align*}
        \Gamma&=\Span_\bfk \{ \sum (a\#_\sigma h_{(1)})\otimes (a'\#_\sigma S(h_{(2)}))^{\op}\ |\  a,a'\in A,h\in H\}\\
        &=\Span_\bfk \{ \sum (a\#_\sigma h_{(1)})\otimes \bigl(S(h_{(3)})\rightharpoonup a'\#_\sigma S(h_{(2)})\bigr)^{\op}\ |\  a,a'\in A,h\in H\}.
    \end{align*}
\end{example}

\subsection{The $H$-module structures}

In this subsection, let $B/A$ be a flat right $H$-Galois extension.
We state some results from \cite{LRW25} that give   $H$-module structures on some Hom spaces and tensor products, respectively.

\begin{lemma}[{\cite[Lemma 3.3]{LRW25}}]  \label{gammalem}
    For all $h\in H$, $$\sum_{i,j} l_i(h_{(1)})\otimes_A r_i(h_{(1)})\otimes l_j(h_{(2)})^\op \otimes_A r_j(h_{(2)})\in B\otimes_A \Gamma\otimes_A B.$$
\end{lemma}

\begin{proposition}[{\cite[Proposition 3.4]{LRW25}}]  \label{prop: H module structure on tensor and Hom}
Let $M$ be a left $\Gamma$-module and $N$ be a $(B,B)$-bimodule. Then there
is a left $H$-module structure on $N \otimes_{A^e} M$ defined by
\[
h \rightharpoonup (n \otimes_{A^e} m)
=
\sum_{i,j}
r_j(h_{(2)}) n l_i(h_{(1)})
\otimes_{A^e}
r_i(h_{(1)}) m l_j(h_{(2)}),
\]
as well as a right $H$-module structure on $\Hom_{A^e}(M,N)$
defined by
\[
(f \leftharpoonup h)(m)
=
\sum_{i,j}
l_i(h_{(1)})
f\bigl(r_i(h_{(1)}) m l_j(h_{(2)})\bigr)
r_j(h_{(2)}),
\]
for any $h \in H$, $m \in M$, $n \in N$, and
$f \in \Hom_{A^e}(M,N)$.
\end{proposition}

\begin{remark}\label{Remark: free H mod}
    Let $N=B^e,M=\Gamma$, then $B^e\otimes_{A^e} \Gamma$,  which is isomorphic to $B\ot H\ot B$,  is a free $H$-module. In fact,  there exists an isomorphism of  left $H$-modules and left $B^e$-modules
    \[
    \varphi: B^e \otimes_{A^e} \Gamma\rightarrow B\ot H\ot B, ~ \sum (b\otimes b'^\op) \otimes_{A^e} (x_p\otimes {y_p}^\op)
           \mapsto
           \sum b(x_p)_{(0)}  \otimes (x_p)_{(1)}\otimes y_p b',
    \]
  with inverse
    \[
    \psi: B  \otimes H\otimes B \rightarrow B^e \otimes_{A^e} \Gamma, ~ b \otimes h \otimes b'\mapsto \sum  \big(bl_i(h_{(1)})\otimes  (r_j(h_{(2)})b')^\op \big)\otimes_{A^e} \big(r_i(h_{(1)})\otimes {l_j({h_{(2)}})}^\op \big).
    \]
\end{remark}

\begin{proposition}[{\cite[Proposition 3.5]{LRW25}}]  \label{Prop: H inviant isom}
Let $M$ be a left $\Gamma$-module and $N$ be a $(B,B)$-bimodule. The following statements hold.
\begin{enumerate}
    \item[(a)]   There exists a natural isomorphism of vector spaces:
    \[
    (N \otimes_{A^e} M)_H
    :=
    \frac{N \otimes_{A^e} M}{[H, N \otimes_{A^e} M]}
    \cong
    N \otimes_{\Gamma} M,
    \]
where $[H, N \otimes_{A^e} M]$ denotes the subspace spanned by all elements
$ h \rightharpoonup (n \otimes_{A^e} m) - \varepsilon(h)n \otimes_{A^e} m,$
for all $h \in H$, $m \in M$, and $n \in N$.

    \item[(b)] There exists a natural isomorphism of vector spaces:
    \[
    \Hom_{A^e}(M,N)^H
    :=
    \left\{
    f \in \Hom_{A^e}(M,N)
    \,\middle|\,
    f \leftharpoonup h = \varepsilon(h)f,\ \forall h \in H
    \right\}
   \cong
    \Hom_{\Gamma}(M,N).
    \]
\end{enumerate}
\end{proposition}

\subsection{Faithfully flat Hopf Galois extensions}

In this subsection, we will show that  $\Gamma$ is projective as an $(A,A)$-bimodule if $B/A$ is a faithfully flat extension.

\begin{theorem-def}\label{Theorem: Equivalent definition of faithfully flat Hopf Galois extension}
({\cite[Theorem 1]{Sch90Princialhomospace}, \cite[Theorem 5.6]{SS05}})
Let $H$ be a Hopf algebra with a bijective antipode, $B$ be a right
$H$-comodule algebra, and $A := B^{\operatorname{co}H}$. Then the following statements
are equivalent.
\begin{myenumerate}
    \item $B$ is faithfully flat as a right $A$-module, and
    $B/A$ is an $H$-Galois extension;

    \item $B$ is faithfully flat as a left $A$-module, and
    $B/A$ is an $H$-Galois extension;

    \item $B/A$ is a Hopf-Galois extension, and the comodule algebra $B$ is right $H$-equivariantly projective as a left $A$-module, i.e.,
    there exists a right $H$-colinear and left $A$-linear splitting of the multiplication $A \otimes B \rightarrow B$.
\end{myenumerate}

In this case,  we say that a Hopf-Galois extension $B/A$ is faithfully flat if $B$ is faithfully flat as a left or right $A$-module.
\end{theorem-def}

A faithfully flat $H$-Galois extension is also called a principal homogeneous space in \cite{Sch98} or a principal $H$-extension in \cite{BH04}.

Applying (c) in Theorem \ref{Theorem: Equivalent definition of faithfully flat Hopf Galois extension} to the $H$-Galois extension $B/A$ and the $H^\op$-Galois extension $B^\op/A^\op$, one has the following Lemma.
\begin{lemma}[{\cite[Lemma 3.3]{Zhu26}}]  \label{Lemma: Gamma is projective as left and right A-module}
    If $B/A$ is a faithfully flat $H$-Galois extension, then $\Gamma$ is  projective as a left $A$-module and right $A$-module.
\end{lemma}

We obtain  the  following result
  by Lemma \ref{Lemma: Gamma is projective as left and right A-module}.
\begin{proposition}\label{proposition: Gamma is projective as A^e-module}
    If $B/A$ is a faithfully flat $H$-Galois extension, then $\Gamma$ is a projective left $A^e$-module.
\end{proposition}

\begin{proof}
By Theorem \ref{Theorem: Equivalent definition of faithfully flat Hopf Galois extension} and Lemma \ref{Lemma: Gamma is projective as left and right A-module}, there is a left $A$-linear and right $H$-colinear splitting
\[
s_l:B\rightarrow A\otimes B
\]
of the multiplication map
\[
m_l:A\otimes B\rightarrow B, a\otimes b\mapsto ab.
\]

Similarly, there exists a right $A$-linear and right $H^{\op}$-colinear splitting
\[
s_r:B^\op\rightarrow B^\op\otimes A
\]
of the multiplication map
\[
m_r:B^\op\otimes A\rightarrow B^\op,  b^\op\otimes a \mapsto (ba)^\op.
\]
Therefore, $s_l\otimes s_r$ is an $(A,A)$-bilinear and right $H\ot H^{\op}$-colinear splitting of $m_l\otimes m_r$.
Hence,  $B\otimes B^\op$ is a direct summand of the free $(A, A)$-bimodule and the right $H\ot H^{\op}$-comodule  $A\otimes B \otimes B^{\op} \otimes A$.

   Since  the map $H\ot H^{\op}\to H, h\ot (h')^{\op}\mapsto hh'$ is a morphism of coalgebras,  consider  the right $H\ot H^{\op}$-comodule  structure map $$A\otimes B\otimes B^\op\otimes A\rightarrow A\otimes B\otimes B^\op\otimes A\otimes H \ot H^{\op}, ~a\otimes b\otimes b'^\op\otimes a'\mapsto a \otimes b_0\otimes (b'_0)^\op\otimes a' \otimes   b_1\ot (b'_1)^{\op} $$
composed with the induced map
$$  A\otimes B\otimes B^\op\otimes A\otimes H \ot H^{\op}\to  A\otimes B\otimes B^\op\otimes A\otimes H,~a\otimes b\otimes b'^\op\otimes a'\ot h\ot (h')^{\op}\mapsto a\otimes b\otimes b'^\op\otimes a'\ot h h',$$  which defines a
      right $H$-comodule structure on $A\otimes B\otimes B^\op\otimes A$.

    Since $s_l,s_r$ are $H$-colinear maps and $H^\op$-colinear maps, respectively, then $s_l\otimes s_r$ is an $H$-colinear map by the above definition.
    Then, $(A\otimes B\otimes B^\op\otimes A)^{\coH}=A\otimes \Gamma\otimes A$ by $A=B^{\coH}$ and $\Gamma=(B\otimes B^\op)^{\coH}$. Since $B\otimes B^\op$ is a direct summand of the free $(A, A)$-bimodule $A\otimes B \otimes B^{\op} \otimes A$, taking $(-)^{\coH}$, one obtains that $\Gamma$ is a direct summand of the free $(A, A)$-bimodule $A\otimes \Gamma \otimes A$.

    We are done.
\end{proof}

Next, we give some examples of faithfully flat Hopf Galois extensions.

\begin{example}[\textbf{Strongly  graded algebras}]
\label{Ex: Strongly G-graded algebras are faithfully flat}
Let $G$ be a   group with identity element $e$ and  $B=\oplus_{g\in G}B_g$ be a strongly $G$-graded algebra. Then for $A=B_e$ and $H=\bfk G$,  $B/A$ is an $H$-Galois extension.

For any $g\in G$, $B_g$ is a finitely generated   projective left $A$-module. In fact,  since $B_{g^{-1}}B_g=B_e=A$,  write  $1_A=\sum_{i}c_id_i$ for $c_i\in B_{g^{-1}}, d_i\in B_g$. One checks that $(d_i,\varphi_i)$ is a pair of dual bases of $B_g$ as a left $A$-module, where $\varphi_i: B_g\rightarrow A, ~x\mapsto xc_i$. Hence, $B=\oplus_{g\in G} B_g$ is a  projective, thus  flat,   $A$-module.

Since $B=A\oplus (\oplus_{g\neq e}B_g)$ as a left  $A$-module, $B$ is an $A$-generator.
Therefore,  $B$ is faithfully flat as a left $A$-module. Together with the $H$-Galois property, this shows that
B/A is a faithfully flat $H$-Galois extension.
\end{example}

\begin{example}[\textbf{Smash product algebras}]\label{ex: smash product is faithfully flat}
    Let $A$ be an $H$-module algebra and  $B=A\# H$ be the smash product.
   The inclusion $A\rightarrow A\# H$ induces a natural left $A$-module structure on $B=A\# H$.
    Then, as a left $A$-module, $A\# H$ is isomorphic to  $A\otimes H$ which is a free left $A$-module, thus $A\# H$ is a faithfully flat left $A$-module. By Example \ref{example: smash product is H-Galois extension} and Theorem-Definition \ref{Theorem: Equivalent definition of faithfully flat Hopf Galois extension}, $B/A$ is a faithfully flat $H$-Galois extension.
\end{example}

\begin{example}[\textbf{Crossed product algebras}]\label{Ex: Crossed product algebras are faithfully flat}
    Let $A$ be  a $\bfk$-algebra  such that $H$ measures $A$ and $B=A\#_\sigma H$ be the crossed product algebra.
   For crossed product algebras, $A\#_\sigma H$ is also isomorphic to the free left $A$-module $A\ot H$, so that $B/A$ is a faithfully flat  $H$-Galois extension.
\end{example}

\bigskip

\section{An isomorphism of   differential bigraded algebras}\label{Section: An isomorphism of   differential bigraded algebras}

In this section, assume that $B/A$ is a faithfully flat $H$-Galois extension, where $H$ is a Hopf algebra with a bijective antipode.

\subsection{An isomorphism of bicomplexes}\ 

Let $L_\bullet$ be a projective resolution of the trivial right $H$-module $\bfk$ and $K_\bullet$ be a  projective resolution of the left $\Gamma$-module $A$. Then by Proposition~\ref{proposition: Gamma is projective as A^e-module},  $K_\bullet$ is also a projective resolution of the left $A^e$-module $A$.
Recall that by Proposition~\ref{prop: H module structure on tensor and Hom},
$B^e\otimes_{A^e}K_\bullet$ has a left $H$-module structure and
$\Hom_{A^e}(K_\bullet,N)$ has a right $H$-module structure, where $N$ is a $(B,B)$-bimodule.
We can form two bicomplexes in the first quadrant which are isomorphic.

\begin{proposition}
\label{prop: isom of bicomplexes}
Let $N$ be a $(B,B)$-bimodule.   There is a bicomplex isomorphism:
    \[
    \Phi^{\bullet,\bullet}_N: \Hom_H(L_\bullet,\Hom_{A^e}(K_\bullet,N))\cong \Hom_{B^e}(L_\bullet\otimes_H(B^e\otimes_{A^e}K_\bullet),N).
    \]
\end{proposition}
\begin{proof}
We first show the isomorphism of vector spaces:
  For any $i,j\ge0$,  define:
    \begin{align*}
        \Phi^{i,j}_N:\Hom_H(L_i,\Hom_{A^e}(K_j,N))&\longrightarrow \Hom_{B^e}(L_i\otimes_H(B^e\otimes_{A^e}K_j),N)\\
        f    &\mapsto \bigl(\Phi^{i,j}_N(f):l_i\otimes_H(b^e\otimes_{A^e}k_j)\mapsto b^e f(l_i)(k_j)\bigr),
    \end{align*}
    where  $b^e\in B^e, l_i\in L_i$ and $k_j\in K_j$.

Since $N$ is a left $B^e$-module, one has  $N\cong \Hom_{B^e}(B^e, N)$ as a left $B^e$-module, so
   $$
\Hom_{A^e}(K_j,N) \cong \Hom_{A^e}(K_j,\Hom_{B^e}(B^e,N)) \cong \Hom_{B^e}(B^e\ot_{A^e} K_j, N).
    $$
    One can show, by direct inspection, that this isomorphism of vector spaces is an isomorphism of  right $H$-modules, where the right $H$-module structure on the right-hand side is induced from the left $H$-module structure on $B^e\ot_{A^e} K_j$.

 Hence, by adjunction, we have
  \[
    \Hom_H(L_i,\Hom_{A^e}(K_j,N))\cong \Hom_H(L_i,\Hom_{B^e}(B^e\ot_{A^e} K_j, N)) \cong \Hom_{B^e}(L_i\otimes_H(B^e\otimes_{A^e}K_j),N).
    \]

    Next, we show that $\Phi^{\bullet,\bullet}_N$ is a morphism of bicomplexes.
    Denote the differentials of the complexes $L_\bullet$ and $K_\bullet$ by $\rmd^L$ and $\rmd^K$, respectively. So, the horizontal differential of the bicomplex $\Hom_H(L_\bullet,\Hom_{A^e}(K_\bullet,N))$ is $\rmd_h:=\Hom_H(-,\Hom_{A^e}(K_\bullet,N))(\rmd^L)$ and the vertical differential is $\rmd_v:=\Hom_H(L_\bullet,\Hom_{A^e}(-,N))(\rmd^K)$, while the horizontal differential of the bicomplex $\Hom_{B^e}(L_\bullet\otimes_H(B^e\otimes_{A^e}K_\bullet),N)$ is $\delta_h:=\Hom_{B^e}(-\otimes_H(B^e\otimes_{A^e}K_\bullet),N)(\rmd^L)$ and the vertical differential is $\delta_v:=\Hom_{B^e}(L_\bullet\otimes_H(B^e\otimes_{A^e}-),N)(\rmd^K)$.

    For any $i,j \ge 0, l_i \in L_i,k_j \in K_j, b^e \in B^e$ and $f\in \Hom_H(L_i,\Hom_{A^e}(K_j,N))$, we have
    \begin{align*}
        \Phi^{i+1,j}_N\circ(\rmd_h(f))(l_{i+1}\otimes_H(b^e\otimes_{A^e} k_j))&=  b^e \bigl(\rmd_h(f)(l_{i+1})\bigr) (k_j)\\
        &= (-1)^{i+j}b^e f\bigl({\rmd^L(l_{i+1})}\bigr) (k_j)\\
        &=(-1)^{i+j}\Phi^{i,j}_N(f)\bigl(\rmd^L(l_{i+1})\otimes_H(b^e\otimes_{A^e} k_j)\bigr)\\
        &=\delta_h\circ \bigl(\Phi^{i,j}_N(f)\bigr)\bigl(l_{i+1}\otimes_H(b^e\otimes_{A^e} k_j)\bigr),
    \end{align*}
    and
    \begin{align*}
        \Phi^{i,j+1}_N\circ \bigl(\rmd_v(f)\bigr)\bigl(l_i\otimes_H(b^e\otimes_{A^e} k_{j+1})\bigr)&= b^e \bigl(\rmd_v(f)(l_i)\bigr) (k_{j+1})\\
        &=(-1)^j b^e  f(l_i) \bigl(\rmd^K(k_{j+1})\bigr)\\
        &=(-1)^j \Phi^{i,j}_N(f)\Bigl(l_i \otimes_H \bigl(b^e\otimes_{A^e} \rmd^K(k_{j+1}) \bigr)\Bigr)\\
        &=\delta_v\circ \bigl(\Phi^{i,j}_N(f)\bigr)\bigl(l_i \otimes_H(b^e\otimes_{A^e} k_{j+1})\bigr).
    \end{align*}
    So $\Phi^{\bullet,\bullet}_N$ commutes with  the horizontal and vertical differentials of the two bicomplexes, which implies that  $\Phi^{\bullet,\bullet}_N$ is a bicomplex isomorphism.
\end{proof}

\begin{proposition}\label{Prop: the total complex is a projective resolution of (B,B) bimodule}
    Let $K_\bullet'$ be the complex $B^e\ot_{A^e} K_\bullet$.
    The total complex $\Tot(L_\bullet\otimes_H{K_\bullet'})$ of the bicomplex $L_\bullet\otimes_H{K_\bullet'}$
    is a  projective resolution of $B$ as a $(B, B)$-bimodule.
\end{proposition}
\begin{proof}
    By Proposition~\ref{prop: isom of bicomplexes}, for any $i,j\ge0$, we have the natural isomorphism of the functors:
    {\small \[
    \Hom_{B^e}(L_i\otimes_H(B^e\otimes_{A^e}K_j),-)\cong \Hom_H(L_i,\Hom_{A^e}(K_j,\Hom_{B^e}(B^e,-)))=\Hom_H(L_i,-)\circ \Hom_{A^e}(K_j,-)\circ \Hom_{B^e}(B^e,-).
    \]}
    Hence, $\Hom_{B^e}(L_i\otimes_H(B^e\otimes_{A^e}K_j),-)$ is an exact functor, since each functor of the composition  is exact, which implies that $L_i\otimes_H(B^e\otimes_{A^e}K_j)$ is a projective left $B^e$-module.

    Next, we show that the complex  $\Tot(L_\bullet\otimes_H{K_\bullet'})$ is quasi-isomorphic to the stalk complex $B$ by a spectral sequence argument.
    Denote $D_{p,q}=L_p\otimes_H(B^e\otimes_{A^e} K_q)$.
    We know that $$\rmE_{p,q}^1=\left\{  \begin{array}{ll}  L_p\ot_H (B^e\ot_{A^e} A),& q=0,\\
       0, & q\neq 0,\end{array}\right.$$
   and $$\rmE_{p,q}^2=\left\{  \begin{array}{ll} \Tor^H_p(\bfk, B^e\ot_{A^e} A),& q=0,\\
       0, & q\neq 0. \end{array}\right.$$
       Hence, for any $n\geq 0$,
       $  \rmH_{n}(\Tot(D_{\bullet,\bullet}))= \Tor^H_n(\bfk, B^e\ot_{A^e} A)$.
   By \cite[Proposition 4.4]{Ste95}, for any $n\geq 1$, we have
   $$\rmH_n(\Tot(L_\bullet\otimes_H {K_\bullet'}))=\Tor^H_n(\bfk,B^e\otimes_{A^e} A)=\Tor^H_n(\bfk,\rmH_0(A,B^e))=0,$$
   and
    \begin{align*}
     \rmH_0(\Tot(L_\bullet\otimes_H{K_\bullet'}))= \bfk\otimes_H(B^e\otimes_{A^e}A)&\cong (B^e\otimes_{A^e}A)_H \\
     &\cong B^e\otimes_\Gamma A ~~~\text{by \ Proposition \ref{Prop: H inviant isom}}\\
     &\cong B ~~~\text{by \ \cite[Lemma 3.6, (2)]{LRW25}}.
    \end{align*}
    This shows the statement.
\end{proof}

Next, we define exterior products on both sides of the isomorphism $\Phi$ of bicomplexes.

Let $$\{\varDelta^{L_\bullet}_{p,q}:L_{p+q} \to L_p\ot L_q,\, p, q\geq 0\}$$ be a diagonal map  of right $H$-modules, lifting the  canonical isomorphism $\bfk\cong \bfk\ot\bfk$ of right  $H$-modules,  and
   $$\{\varDelta^{K_\bullet}_{r,s}:
    K_{r+s}\to K_r \ot_A K_s, r,s\geq 0\}$$ be a diagonal map  of left $\Gamma$-modules,  lifting  the  canonical isomorphism  $A\cong A\ot_A A$ of the left $\Gamma$-modules, where the $\Gamma$-module structure on $K_r\ot_A K_s$ is given in Proposition \ref{prop: Gamma is a A-coring}.

For the $(B, B)$-bimodules $M, N$ and $p,q,r,s\ge 0$, there is a map
$$(\times_\ell)^{p,r,q,s}_{L_\bullet,K_\bullet}: \Hom_H(L_p,\Hom_{A^e}(K_r,M))\ot \Hom_H(L_q,\Hom_{A^e}(K_s,N)) \to \Hom_H(L_{p+q}, \Hom_{A^e}(K_{r+s},M \ot_B N))$$
defined by the  composition below: 
$$\begin{tikzcd}
      \Hom_H(L_p,\Hom_{A^e}(K_r,M))\ot \Hom_H(L_q,\Hom_{A^e}(K_s,N))  \arrow[d, "\star^{p,q}_{L_\bullet}"]  \\
     \Hom_H(L_{p+q}, \Hom_{A^e}(K_r,M)\ot \Hom_{A^e}(K_s,N))
      \arrow[d, "(\times^{r,s}_{K\bullet})_*"]  \\
      \Hom_H(L_{p+q}, \Hom_{A^e}(K_{r+s},M \ot_A N)) \arrow[d] \\
      \Hom_H(L_{p+q}, \Hom_{A^e}(K_{r+s},M \ot_B N)),
\end{tikzcd}$$
where the last map is induced by the surjection $M \ot_A N\to M \ot_B N$. 
Here the subscript $\ell$ indicates that this map is associated with the left-hand side bicomplex in Proposition~\ref{prop: isom of bicomplexes}. Since $\{\star^{p,q}_{L_\bullet}, p, q\geq 0\}$ and $\{\times^{r,s}_{K_\bullet}, r,s\geq 0\}$ are both compatible with differentials, $\times_\ell:=\{(\times_\ell)^{p,r,q,s}_{L_\bullet,K_\bullet}, p, q, r, s\geq 0\}$ is also compatible with differentials.

To define  exterior products on $\Hom_{B^e}(L_\bullet\otimes_H(B^e\otimes_{A^e}K_\bullet) ,-)$,
one needs to construct a diagonal map
\[
\{\varDelta^{P_\bullet}_{u,v}: P_{u+v }\to P_{u}\ot_B P_{v}, u, v\geq 0\}
\]
for $K_\bullet':=B^e\otimes_{A^e}K_\bullet$ and
${P_\bullet}:=\Tot(L_\bullet\otimes_H  K_\bullet')$.
To this end, we  define a family of maps:
$$
\{ \varDelta_{r,s}^{{K_\bullet'}}: B^e\otimes_{A^e}K_{r+s}  \rightarrow (B^e\otimes_{A^e}K_{r})\otimes_B (B^e\otimes_{A^e}K_{s}), r,s\ge 0\}
$$ as follows:
for $ k_{r+s}\in K_{r+s}$, write $\varDelta^{K_\bullet}_{r,s}(k_{r+s})=\sum  k_r\ot_A k_s$ with $k_r\in K_r, k_s\in K_s$.
Then $$\varDelta_{r,s}^{{K_\bullet'}}\bigl((b\ot {b'}^{\op})\ot_{A^e} k_{r+s}\bigr)=\sum \bigl((b\otimes 1^\op)\otimes_{A^e} k_r\bigr)\otimes_B \bigl((1\otimes b'^\op)\otimes_{A^e} k_s \bigr).$$
\begin{lemma}\label{lem: well definedness of delta^K'}
  The map $\varDelta_{r,s}^{{K_\bullet'}}$ is  well-defined and is both a morphism of left $H$-modules and a morphism of $(B,B)$-bimodules.
\end{lemma}
\begin{proof}
    Let $a\otimes {a'}^\op\in A^e, b\ot b'^\op \in B^e$ and $k_{r+s}\in K_{r+s}$.  Since $\varDelta_{r,s}^{K_\bullet}$ is a $\Gamma$-module map, by Proposition \ref{prop: Gamma is a A-coring},
     \begin{align*} \varDelta_{r,s}^{K_\bullet}\bigl((a\otimes {a'}^\op)\cdot k_{r+s})\bigr)
     &=(a\otimes {a'}^\op)\cdot  \varDelta_{r,s}^{K_\bullet}(k_{r+s})\\
     &= (a\otimes {a'}^\op)\cdot  (\sum k_r \ot_A k_s)\\
     &= \sum a  k_r  l_j\bigl(S^{-1}(a'_{(1)})\bigr)\otimes_A r_j\bigl(S^{-1}(a'_{(1)})\bigr) k_s a'_{(0)}\\
     &=\sum ak_r \ot_A k_sa', \end{align*}
     where the last equality follows from $\rho (a')=a'\ot 1$.
    Hence,
   $$  \varDelta_{r,s}^{{K_\bullet'}}\bigl((b\otimes {b'}^\op)\otimes_{A^e} (a\otimes {a'}^\op)\cdot k_{r+s}\bigr)   =\sum \bigl((b\otimes 1^\op)\otimes_{A^e} ak_r \bigr)  \otimes_B \bigl((1\otimes b'^\op)\otimes_{A^e} k_sa'\bigr). $$
    On the other hand, \begin{align*}
        \varDelta_{r,s}^{{K_\bullet'}}\Bigl(  \bigl( (b\otimes {b'}^\op)(a\otimes {a'}^\op) \bigr) \otimes_{A^e} k_{r+s}\Bigr)&=\varDelta_{r,s}^{{K_\bullet'}}\Bigl(\bigl(ba\otimes (a'b')^\op\bigr)\otimes_{A^e}k_{r+s}\Bigr)\\
        &=\sum \bigl((ba\otimes 1^\op)\otimes_{A^e} k_r\bigr)\otimes_B \bigl((1\otimes (a'b')^\op ) \otimes_{A^e} k_s \bigr)\\
        &=\sum \bigl((b\otimes 1^\op)\otimes_{A^e} ak_r \bigr)\otimes_B \bigl( (1\otimes b'^\op)\otimes_{A^e} k_sa' \bigr).
    \end{align*}
    So $\varDelta_{r,s}^{{K_\bullet'}}$ is  well defined.

    Next, we check that $\varDelta^{{K_\bullet'}}_{r,s}$ is a morphism of $H$-modules.
    By Proposition \ref{prop: H module structure on tensor and Hom},  for  $h\in H$, the left $H$-module structure of $B^e\otimes_{A^e} K_{r+s}$ is given by
     \begin{align}\label{eq: H modules structure}
    h\rightharpoonup \bigl( (b\otimes b'^\op)\otimes_{A^e} k_{r+s} \bigr)=\sum_{i,j} \bigl( bl_i(h_{(1)})\otimes r_j(h_{(2)}) b' \bigr)\otimes_{A^e} \big(r_i(h_{(1)})k_{r+s} l_j(h_{(2)}) \big).
    \end{align}

    Since $\varDelta^{K_\bullet}_{r,s}$ is a morphism of $\Gamma$-modules,  for $\sum r_i(h_{(1)})\otimes {l_j(h_{(2)})}^\op\in \Gamma$,  by Proposition \ref{prop: Gamma is a A-coring}
    \begin{align}
        \varDelta^{K_\bullet}_{r,s} \bigl(\sum_{i,j} r_i(h_{(1)}) k_{r+s} l_j(h_{(2)})\bigr)&=\sum\sum_{i,j} \big( r_i(h_{(1)})\otimes {l_j(h_{(2)})}^\op \big) \cdot ( k_r\otimes_A k_s)\nonumber\\
        &=\sum\sum_{i,j,u}  r_i(h_{(1)})_{(0)}k_{r} l_u \bigl( r_i(h_{(1)})_{(1)} \bigr) \otimes_A r_u \bigl(r_i(h_{(1)})_{(1)} \bigr) k_sl_j(h_{(2)}). \label{eq: varDelta formulae}
    \end{align}
    Thus,
    \begin{align*}
        &\varDelta_{r,s}^{{K_\bullet'}} \Bigl( h\rightharpoonup \bigl((b\otimes b'^\op)\otimes_{A^e} k_{r+s} \bigr) \Bigr)\\
       \stackrel{\eqref{eq: H modules structure}}{=}& \varDelta_{r,s}^{{K_\bullet'}} \Bigl( \sum_{i,j} \bigl( bl_i(h_{(1)})\otimes r_j(h_{(2)}) b' \bigr)\otimes_{A^e}\bigl(r_i(h_{(1)})k_{r+s} l_j(h_{(2)}) \bigr) \Bigr)\\
        \stackrel{\eqref{eq: varDelta formulae}}{=}&\sum \sum_{i,j,u} \Bigl( (bl_i(h_{(1)})\otimes 1^\op)\otimes_{A^e} r_i(h_{(1)})_{(0)}k_r l_u \bigl( r_i(h_{(1)})_{(1)} \bigr) \Bigr)\otimes_B \Bigl( \bigl(1\otimes  ( r_j(h_{(2)})b' )^\op \bigr)\otimes_{A^e} r_u (r_i(h_{(1)})_{(1)}) k_sl_j(h_{(2)}) \Bigr)\\
        =&\sum \sum_{i,j,u} \Bigl( \bigl(bl_i(h_{(1)})\otimes 1^\op \bigr)\otimes_{A^e} r_i(h_{(1)})k_r l_u(h_{(2)}) \Big) \otimes_B \Bigl( \bigl(1\otimes  (r_j(h_{(3)})b' )^\op \bigr)\otimes_{A^e} r_u(h_{(2)}) k_sl_j(h_{(3)}) \Bigr),
    \end{align*}
    where the last equality holds since by Eq. \eqref{beta2}, one has
    $$\sum_i l_i(h_{(1)})\otimes_A r_i(h_{(1)})_{(0)}\otimes r_i(h_{(1)})_{(1)}=\sum l_i(h_{(1)})\otimes_A r_i(h_{(1)})\otimes h_{(2)}.$$

    On the other hand, the $H$-module structure of $(B^e\otimes_{A^e}K_{r})\otimes_B (B^e\otimes_{A^e}K_{s})$ is induced by the diagonal action. Hence,
    \begin{align}
        &  h\rightharpoonup \Bigl(\varDelta_{r,s}^{{K_\bullet'}} \bigl( (b\otimes b'^\op)\otimes_{A^e} k_{r+s} \bigr) \Bigr)\nonumber\\
        ={}& h\rightharpoonup \Bigl( \sum \bigl((b\otimes 1^\op)\otimes_{A^e} k_r \bigr)\otimes_B \bigl( (1\otimes b'^\op)\otimes_{A^e} k_s \bigr) \Bigr)\nonumber\\
        ={}&\sum \Bigl(h_{(1)}\rightharpoonup \bigl( (b\otimes 1^\op)\otimes_{A^e} k_{r} \bigr) \Bigr) \otimes_B \Bigl(h_{(2)}\rightharpoonup \bigl( (1\otimes b'^\op)\otimes_{A^e} k_{s} \bigr) \Bigr)\nonumber\\
      ={} &\sum \sum_{i,j,u,v} \Bigl( \bigl(b l_i(h_{(1)}) \otimes {r_j(h_{(2)})}^\op \bigr) \otimes_{A^e} \bigl( r_i(h_{(1)}) k_r l_j(h_{(2)}) \bigr) \Bigr) \otimes_B \Bigl( \bigl(l_{u}(h_{(3)})\otimes (r_{v}(h_{(4)})b' )^\op \bigr)\otimes_{A^e} \bigl(r_{u}(h_{(3)})k_s l_v(h_{(4)}) \bigr) \Bigr)\nonumber\\
        ={}&\sum \sum_{i,j,u,v} \Bigl( \bigl(b l_i(h_{(1)})\otimes 1^\op \bigr)\otimes_{A^e} \bigl(r_i(h_{(1)}) k_r l_j(h_{(2)}) \bigr) \Bigr)\otimes_B \Bigl( \bigl(r_j(h_{(2)})l_{u}(h_{(3)})\otimes (r_{v}(h_{(4)})b' )^\op \bigr)\otimes_{A^e} \bigl(r_{u}(h_{(3)})k_s l_v(h_{(4)}) \bigr) \Bigr) \label{eq: equation in proof of varDelta is a morphism of left H modules}, 
    \end{align}
      where the $H$-module structures in the third equality are given by   Proposition \ref{prop: H module structure on tensor and Hom}.
By replacing $b$ with $r_j(h_{(2)})$ in Eq. \eqref{beta6},  we have $$\sum_{j,u} r_j(h_{(2)})_{(0)}l_{u}(r_j(h_{(2)})_{(1)})\otimes_A r_{u}(r_j(h_{(2)})_{(1)})=1\otimes_A r_j(h_{(2)}),$$
so by Eq. \eqref{beta2},
    \begin{align*}
        \sum_{j,u} l_j(h_{(2)})\otimes r_j(h_{(2)})l_{u}(h_{(3)})\otimes_A r_{u}(h_{(3)})&= \sum_{j,u} l_j(h_{(2)})\otimes r_j(h_{(2)})_{(0)} l_{u}\bigl( r_j(h_{(2)})_{(1)} \bigr)\otimes_A r_{u}\bigl(r_j(h_{(2)})_{(1)} \bigr)\\
        &=\sum_j l_j(h_{(2)})\otimes 1 \otimes_A r_j(h_{(2)}).
    \end{align*}
    Hence, Eq. $\eqref{eq: equation in proof of varDelta is a morphism of left H modules}$ is equal to $$\sum \sum_{i,j,v} \Bigl( \bigl(b l_i(h_{(1)})\otimes 1^\op \bigr)\otimes_{A^e} \bigl(r_i(h_{(1)}) k_r l_j(h_{(2)}) \bigr) \Bigr)\otimes_B \Bigl( \bigl(1\otimes (r_{v}(h_{(3)})b' )^\op \bigr)\otimes_{A^e} \bigl(r_j(h_{(2)})k_sl_v(h_{(3)}) \bigr) \Bigr).$$

    Therefore, we have shown that  $$\varDelta_{r,s}^{{K_\bullet'}} \Bigl(h\rightharpoonup \bigl( (b\otimes b'^\op)\otimes_{A^e} k_{r+s} \bigr) \Bigr)=h\rightharpoonup \Bigl(\varDelta^{{K_\bullet'}}\bigl( (b\otimes b'^\op)\otimes_{A^e} k_{r+s} \bigr) \Bigr),$$ which implies that $\varDelta^{K_\bullet'}_{r,s}$ is a morphism of left $H$-modules.

    It remains to show that $\varDelta^{K'}_{r,s}$ is a morphism of $(B,B)$-bimodules. Note that for left $\Gamma$-module $K$, the $(B,B)$-bimodule structure of $B^e\ot_{A^e} K$ is induced by the left $B^e$-module of $B^e$, that is, for any $c,c'\in B$ and $k\in K$,
    \[
    c\cdot\bigl((b\otimes b'^\op)\ot_{A^e} k \bigr)\cdot c'= \bigl(cb \otimes (b'c')^\op \bigr) \otimes_{A^e} k.
    \]
    Hence, the $(B,B)$-bimodule structure of $(B^e\otimes_{A^e}K_{r})\otimes_B (B^e\otimes_{A^e}K_{s})$ can be defined as follows: for $b\otimes b'^\op,b''\otimes b'''^\op\in B^e, k_r\in K_r$ and $k_s\in K_s$,
    \[
    c\cdot \big( (b\otimes b'^\op)\otimes_{A^e} k_r \big) \otimes_B \big( (b''\otimes b'''^\op)\otimes_{A^e} k_s \big)  \cdot c'=\big( (cb\otimes b'^\op)\otimes_{A^e} k_r \big)\otimes_B \big( (b''\otimes (b'''c')^\op)\otimes_{A^e} k_s \big).
    \]
    Then,
    \begin{align*}
        \varDelta^{K_\bullet'}_{r,s}\Big(c\cdot \big( (b\otimes b'^\op)\ot_{A^e} k_r \big)\cdot c' \Big)=& \varDelta^{K_\bullet'}_{r,s} \Big( \big(cb\otimes (b'c')^\op \big)\ot_{A^e} k_r \Big)\\
        =&\sum \big( (cb\otimes 1^\op)\otimes_{A^e} k_r \big)\otimes_B \big( (1\otimes (b'c')^\op )\otimes_{A^e} k_s \big),
    \end{align*}
    while
    \begin{align*}
        c\cdot\varDelta^{K_\bullet'}_{r,s} \Big( \big( (b\otimes b'^\op)\ot_{A^e} k_r \big) \Big)\cdot c' =&\sum c\cdot \big( (b\otimes 1^\op)\otimes_{A^e} k_r \big)\otimes_B \big( (1\otimes b'^\op)\otimes_{A^e} k_s \big)\cdot c'\\
        =&\sum \big( (cb\otimes 1^\op)\otimes_{A^e} k_r \big)\otimes_B \big( (1\otimes (b'c')^\op)\otimes_{A^e} k_s \big).
    \end{align*}
    So, $\varDelta^{K'}_{r,s}$ is a morphism of $(B,B)$-bimodules.
\end{proof}

Now for $p, q, r, s\geq 0$,   we can define a  map
$\varDelta_{p,r;q,s}^{P_\bullet}: P_{p+r+q+s}\to {P_{p+r}}\ot_B {P_{q+s}}$ as the composition of  $$\varDelta_{p,q}^{L_\bullet} \otimes_H \varDelta_{r,s}^{{K_\bullet'}}:L_{p+q} \otimes_H \left(B^e \otimes_{A^e} K_{r+s}\right) \to  (L_p\otimes L_q) \otimes_H
\left(
\left(B^e \otimes_{A^e} K_r \right)
\otimes_B
\left(B^e \otimes_{A^e} K_s\right)
\right)$$ with the obvious map
\begin{align*}
    \alpha_{p,q,r,s}: (L_p \otimes L_q) \otimes_H
\left(
\left(B^e \otimes_{A^e} K_r \right)
\otimes_B
\left(B^e \otimes_{A^e} K_s\right)
\right) &\rightarrow
\left(L_p \otimes_H \left(B^e \otimes_{A^e} K_r \right)\right)
\otimes_B
\left(L_q \otimes_H \left(B^e \otimes_{A^e} K_s \right)\right)\\
(l_p\otimes l_q)\ot_H \big( ( (b\otimes b'^\op)\otimes_{A^e} k_r ) \ot_B ( (b''\otimes b'''^\op)\otimes_{A^e} k_s ) \big)&\mapsto (-1)^{qr} \big(l_p\ot_H ((b\ot b'^\op)\ot_{A^e} k_r) \big) \ot_B \big(l_q\ot_H ((b''\otimes b'''^\op)\ot_{A^e} k_s) \big),
\end{align*}
for any $l_p\in L_p,l_q\in L_q,k_r\in K_r, k_s \in K_s$ and $ b\otimes b'^\op, b''\otimes b'''^\op \in B^e$.

These maps $\{\varDelta_{u, v}^{P_\bullet}=\sum_{p+r=u, q+s=v} \varDelta_{p,r;q,s}^{P_\bullet}: P_{u+v}\to {P_{u}}\ot_B {P_{v}}\}$ form a diagonal map of $(B,B)$-bimodules $\varDelta^{P_\bullet}: P_\bullet \to \Tot(P_\bullet\ot_B P_\bullet)$, which lifts the  canonical isomorphism $B\cong B\ot_B B$ of  $(B,B)$-bimodules.

 This diagonal map induces an exterior product of the right-hand side of the isomorphism $\Phi^{\bullet, \bullet}$, whose components are of the form:
 $$(\times_r)^{p,r,q,s}_{L_\bullet,K_\bullet}: \Hom_{B^e} (L_p \otimes_H (B^e \otimes_{A^e} K_r), M) \ot \Hom_{B^e}(L_q \otimes_H (B^e \otimes_{A^e} K_s), N) \to \Hom_{B^e}(L_{p+q} \otimes_H (B^e \otimes_{A^e} K_{r+s}), M\ot_B N).$$
 Here the subscript $r$ indicates that this map is associated with the right-hand side bicomplex in Proposition~\ref{prop: isom of bicomplexes}.
Since $\varDelta^{L_\bullet},\varDelta^{K'_\bullet}$ and $\alpha$ are  compatible with differentials, $\times_r$ is also compatible with differentials.

We claim that the isomorphism $\Phi$ commutes with exterior products.
\begin{proposition}\label{proposition: Phi commutes with cup product}
    The chain map $\Phi^{\bullet,\bullet}$ is a morphism commuting with exterior products, that is, the diagram below commutes
    for two $(B, B)$-bimodules $M, N$ and $p,q,r,s\ge 0$
 \[
 \begin{tikzcd}[cells={nodes={font=\scriptsize}}]
      \Hom_H(L_p,\Hom_{A^e}(K_r,M))\ot  \Hom_H(L_q,\Hom_{A^e}(K_s,N))  \arrow[d, "(\times_\ell)^{p,r,q,s}_{L_\bullet,K_\bullet}"]\arrow[r, "\Phi^{p,r}_M\ot \Phi^{q,s}_N"] & \Hom_{B^e}(L_p \otimes_H (B^e \otimes_{A^e} K_r), M) \ot \Hom_{B^e}(L_q \otimes_H (B^e \otimes_{A^e} K_s), N)\arrow[d, "(\times_r)^{p,r,q,s}_{L_\bullet,K_\bullet}"] \\
      \Hom_H(L_{p+q}, \Hom_{A^e}(K_{r+s},M \ot_B N)) \arrow[r, "\Phi^{p+q,r+s}_{M\ot_B N}"]& \Hom_{B^e}(L_{p+q} \otimes_H (B^e \otimes_{A^e} K_{r+s}), M\ot_B N)
\end{tikzcd}
\]

\end{proposition}
\begin{proof}
    Let
$f\in \Hom_H\bigl(L_p,\Hom_{A^e}(K_r,M)\bigr),~
g\in \Hom_H\bigl(L_q,\Hom_{A^e}(K_s,N)\bigr).$
For $l_{p+q}\in L_{p+q}$ and $k_{r+s}\in K_{r+s}$, write
$\varDelta_{p,q}^{L_\bullet}(l_{p+q})=\sum l_p\otimes l_q,~ \varDelta_{r,s}^{K_\bullet}(k_{r+s})=\sum k_r\otimes_A k_s$,
where $l_p\in L_p, l_q\in L_q, k_r\in K_r$, and $k_s\in K_s$.
 Since the degrees of $M,N$ are both 0, the degrees of $f$ and $g$ are $-p-r$  and $-q-s$, respectively.

    By definition,  one has
    \[
    (\times_\ell)^{p,r,q,s}_{L_\bullet,K_\bullet}(f\otimes g): l_{p+q}\mapsto F_{l_{p+q}}: k_{r+s}\mapsto \sum (-1)^{rs+p(q+s)} \big(f(l_p) \big)(k_r)\otimes_B \big(g(l_q) \big)(k_s),
    \]
    hence, by Proposition \ref{prop: isom of bicomplexes}, for $b^e= b\otimes {b'}^\op\in B^e$,
    \begin{align*}
        \Phi^{p+q,r+s}_{M\ot_B N}\circ (\times_\ell)^{p,r,q,s}_{L_\bullet,K_\bullet}(f\otimes g): l_{p+q}\otimes_H (b^e\otimes_{A^e} k_{r+s})&\mapsto  b^e\cdot F_{l_{p+q}}(k_{r+s}) \\
        &=(-1)^{rs+p(q+s)}b^e\cdot \Big( \sum \big(f(l_p) \big)(k_{r})\otimes_B \big(g(l_q) \big) (k_{s}) \Big).
    \end{align*}
    On the other hand, as discussed above,
    \[
    \varDelta^{P_\bullet}_{p+r,q+s}(l_{p+q}\otimes_H(b^e\otimes_{A^e} k_{r+s}))=\sum (-1)^{qr} \Big(l_p\ot_H \big( (b\ot 1^\op)\ot_{A^e} k_r \big) \Big)\otimes_B \Big(l_q\ot_H \big( (1\otimes b'^\op)\ot_{A^e} k_s \big) \Big),
    \]
    hence,
    \[
    (\times_r)^{p,r,q,s}_{L_\bullet,K_\bullet}\circ(\Phi^{p,r}_M \otimes \Phi^{q,s}_N)(f\otimes g): l_{p+q}\otimes_H(b^e\otimes_{A^e} k_{r+s})\mapsto \sum (-1)^{pq+ps+rs} \Big( (b\otimes 1^\op)\cdot  \big(f(l_p) \big) (k_{r}) \Big) \otimes_B \Big( (1\otimes b'^\op)\cdot \big(g(l_q) \big) (k_{s}) \Big).
    \]
    And by the $(B,B)$-bimodule structure on $M\otimes_B N$, one has
    $$b^e\cdot \big(\sum f(l_p)(k_{r})\otimes_B g(l_q)(k_{s}) \big)=\sum (b\otimes 1^\op)\cdot \big(f(l_p)(k_{r}) \big) \otimes_B (1\otimes b'^\op)\cdot \big(g(l_q)(k_{s}) \big),$$
    which implies $\Phi^{p+q,r+s}_{M\ot_B N}\circ(\times_\ell)^{p,r,q,s}_{L_\bullet,K_\bullet}(f\ot g)=(\times_r)^{p,r,q,s}_{L_\bullet,K_\bullet}\circ(\Phi^{p,r}_M \otimes \Phi^{q,s}_N)(f\ot g)$. By the arbitrariness of $f,g,l_{p+q},k_{r+s}$ and $b^e$, the diagram above commutes.
\end{proof}

\subsection{Making exterior products strictly associative}\ 

Note that the two exterior products $\times_\ell$ and $\times_r$ are only associative up to homotopy. However, we will show   that  special choices of $K_\bullet$, $L_\bullet$ (and also $P_\bullet$) as well as the diagonal maps $\varDelta^{K_\bullet}, \varDelta^{L_\bullet}$ (and also $ \varDelta^{P_\bullet}$) can be made  so that
they  are strictly associative.

  Let $L_\bullet$ be the  one-sided bar resolution $B_\bullet'={B_\bullet}(\bfk; H; H)$ of the trivial right $H$-module $\bfk$ and $\varDelta^{L_\bullet}$ be the diagonal map $\varDelta^{B_\bullet'}$, displayed in  Subsection~\ref{subsection: Cup products in the cohomology theory of a Hopf algebra}, which is coassociative.

  We now consider the relative one-sided bar resolution $B^A_\bullet(\Gamma;\Gamma;A)$ of the left $\Gamma$-module $A$. For any $n\ge 0$, set
     \[
    B^A_n(\Gamma;\Gamma;A)=\Gamma\ot_{A^\op} \Gamma^{\ot_{A^\op} n},
    \]
    where the left $\Gamma$-module of $B^A_n(\Gamma;\Gamma;A)$ is given by the left $\Gamma$-module structure of the leftmost $\Gamma$ and the modules structures to define  $\Gamma\ot_{A^\op} \Gamma$ is  defined as follows: for $\sum_p x_p\ot y_p^\op \in \Gamma$, the left $A^\op$-module structure and right $A^\op$-module structure are defined by  $$a^\op \cdot \sum_p (x_p\ot y_p^\op) =\sum_p (x_p\ot (y_pa)^\op) ~~ \text{and} ~~ \sum_p (x_p\ot y_p^\op)\cdot a^\op=\sum_p (x_p\ot (ay_p)^\op).$$

    The key reason for using the above module structures over $A^\op$  is that the multiplication $\Gamma\ot_{A^\op}  \Gamma\rightarrow  \Gamma$ is well-defined:  for $\sum_p x_p\ot y_p^\op,\sum_p x'_{q}\ot y_{q}'^\op\in \Gamma$, we have 
    $$\begin{array}{rcl}
    (\sum_{p} (x_p\ot {y_p}^\op)\cdot a^{\op})\cdot(\sum_{q}x'_{q}\ot {y'_{q}}^{\op})&=&\sum_{p,q} (x_p\ot {(ay_p)}^\op)\cdot(x'_{q}\ot {y'_{q}}^{\op})\\
    &=&\sum_{p,q} (x_px'_{q} \ot {(y'_{q}ay_p)}^\op)\\
    &=&\sum_{p,q} (x_p\ot {(y_p)}^\op)\cdot (x'_{q}\ot (y'_{q}a)^\op)\\
    &=& (\sum_{p} (x_p\ot {y_p}^\op) )\cdot(a^{\op} \cdot(\sum_{q}x'_{q}\ot {y'_{q}}^{\op}).
    \end{array}
    $$

    For $\gamma_0 \ot_{A^\op} \cdots\ot_{A^\op} \gamma_n\in B_n^A(\Gamma;\Gamma;A)$ and $n\ge1$,  the differential $\rmd_n$ is defined by 
    \[
     \rmd_n(\gamma_0 \ot_{A^\op} \cdots\ot_{A^\op} \gamma_n)=\sum_{i=0}^{n-1} (-1)^i \gamma_0\ot_{A^\op} \cdots \ot_{A^\op} \gamma_i\gamma_{i+1}\ot_{A^\op} \cdots \ot_{A^\op} \gamma_n+ (-1)^n\gamma_0 \ot_{A^\op} \cdots\ot_{A^\op} \gamma_{n-1}\cdot \epsilon_\Gamma(\gamma_n)^\op,
    \]
    and for  $n=0$, $\rmd_0: \Gamma\rightarrow A$ sends $\gamma_0=\sum_p x_p \ot y_p^\op$  to $\epsilon_\Gamma(\gamma_0)= \sum x_p\cdot y_p$.
     By direct computation, $\rmd_n$ is well-defined and $\rmd_n\circ \rmd_{n-1}=0, \forall n\geq 1$. 
     
     And we can construct a contracting homotopy $s_n$ sending $\gamma_0 \ot_{A^\op} \cdots\ot_{A^\op} \gamma_n$ to $1_\Gamma\ot_{A^\op}\gamma_0 \ot_{A^\op} \cdots\ot_{A^\op} \gamma_n$ for $n\ge0$. and $s_{-1}$ maps $a$ to $1_\Gamma\cdot a^\op$. Thus, the complex is exact.

    Next,  by Lemma \ref{Lemma: Gamma is projective as left and right A-module}, $\Gamma$ is a  projective right    $A$-module by outer action, so $\Gamma$ is a  projective left $A^\op$-module and thus $B_n^A(\Gamma;\Gamma;A)=\Gamma\ot_{A^\op} \Gamma^{\ot_{A^\op} n}$ is a   projective left $\Gamma$-module. 

    This shows that the relative one-sided bar resolution is indeed a projective resolution of ${}_\Gamma A$. 

The following result is the most technical part of this paper. 
\begin{proposition}\label{prop: the diagonal map of B^A}
   Define a diagonal map  $$\varDelta^{B^A_\bullet}=\{\varDelta^{B^A_\bullet}_{r,s}: B_{r+s}^A(\Gamma;\Gamma;A)\rightarrow B_r^A(\Gamma;\Gamma;A)\ot_A B_s^A(\Gamma;\Gamma;A)\}_{r,s\geq 0}$$     by
    \[
    \varDelta^{B^A_\bullet}_{r,s}(\gamma_0 \ot_{A^\op} \cdots\ot_{A^\op} \gamma_{r+s})= \big( (\gamma_0)_{(1)} \ot_{A^\op} \cdots\ot_{A^\op} (\gamma_{r})_{(1)} \big)\ot_A \big( (\gamma_0)_{(2)}\cdots (\gamma_r)_{(2)}\ot_{A^\op} \gamma_{r+1} \ot_{A^\op} \cdots \ot_{A^\op}  \gamma_{r+s} \big),
    \]
    where, by convention, when  $s=0$,  the right handed side is understood as 
    \[
      \big( (\gamma_0)_{(1)} \ot_{A^\op} \cdots\ot_{A^\op} (\gamma_{r})_{(1)} \big)\ot_A \big((\gamma_0)_{(2)}\cdots (\gamma_r)_{(2)}\big).
    \]
    Then it is  a coassociative $\Gamma$-linear map and commutes with differentials.
\end{proposition}

\begin{proof}
   \textbf{Step 1}:  We  check that $\varDelta^{B^A_\bullet}_{r,s}$ is well defined, that is,   for any $ 0\le t \le r+s-1$ and $\gamma_0, \dots, \gamma_{r+s}\in \Gamma$, the images under $\varDelta^{B^A_\bullet}_{r,s}$ of $$\gamma_0 \ot_{A^\op} \cdots\ot_{A^\op} \gamma_t\cdot a^\op  \ot_{A^\op} \gamma_{t+1} \ot_{A^\op} \cdots\ot_{A^\op} \gamma_{r+s}$$ and 
    $$\gamma_0 \ot_{A^\op} \cdots\ot_{A^\op} \gamma_t \ot_{A^\op} a^\op\cdot \gamma_{t+1} \ot_{A^\op} \cdots\ot_{A^\op} \gamma_{r+s}  $$ coincide. 
     
    Let $\gamma=\sum_p x_p\ot y_p^\op\in\Gamma$. Then by Proposition \ref{prop: Gamma is a A-coring},  $$\Delta_{\Gamma}(\gamma\cdot a^\op)= \sum (x_p)_{(0)}\ot l_i\big( (x_p)_{(1)} \big)^\op \ot_A r_i\big( (x_p)_{(1)} \big) \ot (ay_p)^\op,$$  
    so \[
    \sum (\gamma\cdot a^\op)_{(1)}\ot_A (\gamma\cdot a^\op)_{(2)}=\sum \gamma_{(1)}\ot_A \gamma_{(2)}\cdot a^\op.  
    \] Similarly, we have 
    \[ \sum (a^\op\cdot \gamma)_{(1)}\ot_A (a^\op\cdot \gamma)_{(2)}=\sum \gamma_{(1)}\ot_A a^\op \cdot\gamma_{(2)}.\]
    The cases  $0\le t\le r-1$ follow  from these equalities and the fact that  the multiplication $\Gamma \ot_{A^\op} \Gamma\to \Gamma$ is well-defined.

    For $t=r$,
    \begin{align*}
        &\varDelta^{B^A_\bullet}_{r,s}(\gamma_0 \ot_{A^\op} \cdots\ot_{A^\op} \gamma_r\cdot a^\op  \ot_{A^\op} \gamma_{r+1} \ot_{A^\op} \cdots\ot_{A^\op} \gamma_{r+s}) \\
        ={}& \big( (\gamma_0)_{(1)} \ot_{A^\op} \cdots\ot_{A^\op} (\gamma_{r})_{(1)} \big) \ot_A \big( (\gamma_0)_{(2)}\cdots (\gamma_r)_{(2)}\cdot a^\op \ot_{A^\op} \gamma_{r+1} \ot_{A^\op} \cdots \ot_{A^\op}  \gamma_{r+s} \big)\\
        ={}& \big( (\gamma_0)_{(1)} \ot_{A^\op} \cdots\ot_{A^\op} (\gamma_{r})_{(1)} \big)\ot_A \big( (\gamma_0)_{(2)}\cdots (\gamma_r)_{(2)} \ot_{A^\op} a^\op\cdot \gamma_{r+1} \ot_{A^\op} \cdots \ot_{A^\op}  \gamma_{r+s}\big)\\
        ={}& \varDelta^{B^A_\bullet}_{r,s}(\gamma_0 \ot_{A^\op} \cdots\ot_{A^\op} \gamma_r \ot_{A^\op} a^\op\cdot \gamma_{r+1} \ot_{A^\op} \cdots\ot_{A^\op} \gamma_{r+s}).
    \end{align*}
    
    And the cases  $t\ge r+1$ are  trivially  true.

    This shows that  $\varDelta^{B^A_\bullet}_{r,s}$ is well defined.

\medskip

    \textbf{Step 2}: We show that  $\varDelta^{B^A_\bullet}_{r,s}$ is a $\Gamma$-linear map.
    
    For $\gamma,\gamma_0\cdots,\gamma_{r+s}\in \Gamma$, 
    $$\begin{array}{rl}
    &\varDelta^{B^A_\bullet}_{r,s}(\gamma\gamma_0 \ot_{A^\op} \cdots\ot_{A^\op} \gamma_{r+s})\\
    =&\big((\gamma \gamma_0)_{(1)} \ot_{A^\op} \cdots\ot_{A^\op} (\gamma_{r})_{(1)} \big)\ot_A \big((\gamma \gamma_0)_{(2)}\cdots (\gamma_r)_{(2)}\ot_{A^\op} \gamma_{r+1} \ot_{A^\op} \cdots \ot_{A^\op}  \gamma_{r+s}\big)\\
    =&\big(\gamma_{(1)}(\gamma_0)_{(1)} \ot_{A^\op} \cdots\ot_{A^\op} (\gamma_{r})_{(1)} \big)\ot_A \big(\gamma_{(2)}(\gamma_0)_{(2)}\cdots (\gamma_r)_{(2)}\ot_{A^\op} \gamma_{r+1} \ot_{A^\op} \cdots \ot_{A^\op}  \gamma_{r+s}\big),
    \end{array}$$
    and 
    $$\begin{array}{rl}
    &\gamma\cdot \varDelta^{B^A_\bullet}_{r,s}(\gamma_0 \ot_{A^\op} \cdots\ot_{A^\op} \gamma_{r+s})\\
    =&\gamma\cdot \Big(\big(  (\gamma_0)_{(1)} \ot_{A^\op} \cdots\ot_{A^\op} (\gamma_{r})_{(1)} \big)\ot_A \big( (\gamma_0)_{(2)}\cdots (\gamma_r)_{(2)}\ot_{A^\op} \gamma_{r+1} \ot_{A^\op} \cdots \ot_{A^\op}  \gamma_{r+s}\big)\Big) \\
    =&\big( \gamma_{(1)}(\gamma_0)_{(1)} \ot_{A^\op} \cdots\ot_{A^\op} (\gamma_{r})_{(1)} \big)\ot_A \big(\gamma_{(2)}(\gamma_0)_{(2)}\cdots (\gamma_r)_{(2)}\ot_{A^\op} \gamma_{r+1} \ot_{A^\op} \cdots \ot_{A^\op}  \gamma_{r+s}\big).
    \end{array}$$
    where the last equality follows from Proposition \ref{prop: Gamma is a A-coring}.
    Thus, $\varDelta^{B^A_\bullet}_{r,s}$ is a $\Gamma$-linear map.

\medskip

    \textbf{Step 3}: We show that $\varDelta^{B^A_\bullet}$ is coassociative.

    For $\gamma_0,\cdots,\gamma_{r+u+s}\in \Gamma$ and $r,u, s\ge0$,
    \begin{align*}
        &(\varDelta^{B^A_\bullet}_{r,u}\ot_A \id_{B^A_s})\circ \varDelta^{B^A_\bullet}_{r+u,s}(\gamma_0 \ot_{A^\op} \cdots\ot_{A^\op} \gamma_{r+u+s})\\
        ={}&(\varDelta^{B^A_\bullet}_{r,u}\ot_A \id_{B^A_s}) \Big( \big((\gamma_0)_{(1)} \ot_{A^\op} \cdots\ot_{A^\op} (\gamma_{r+u})_{(1)} \big)\ot_A \big((\gamma_0)_{(2)}\cdots (\gamma_{r+u})_{(2)}\ot_{A^\op} \gamma_{r+u+1} \ot_{A^\op} \cdots \ot_{A^\op}  \gamma_{r+u+s} \big) \Big)\\
        ={}&\big((\gamma_0)_{(1)} \ot_{A^\op} \cdots\ot_{A^\op} (\gamma_r)_{(1)}\big)\ot_A \big( (\gamma_0)_{(2)}\cdots (\gamma_{r})_{(2)} \ot_{A^\op} (\gamma_{r+1})_{(1)} \ot_{A^\op} \cdots\ot_{A^\op} 
        (\gamma_{r+u})_{(1)} \big)\\
      &  \ot_A\big((\gamma_0)_{(3)}\cdots \ (\gamma_r)_{(3)}(\gamma_{r+1})_{(2)}\cdots  (\gamma_{r+u})_{(2)}\ot_{A^\op} \gamma_{r+u+1} \ot_{A^\op} \cdots \ot_{A^\op}  \gamma_{r+u+s} \big),
    \end{align*}
    while 
    \begin{align*}
        &(\id_{B^A_r} \ot_A \varDelta^{B^A_\bullet}_{u,s})\circ \varDelta^{B^A_\bullet}_{r,u+s}(\gamma_0 \ot_{A^\op} \cdots\ot_{A^\op} \gamma_{r+u+s})\\
        ={}&(\id_{B^A_r} \ot_A \varDelta^{B^A_\bullet}_{u,s}) \Big( \big( (\gamma_0)_{(1)} \ot_{A^\op} \cdots\ot_{A^\op} (\gamma_{r})_{(1)} \big)\ot_A \big((\gamma_0)_{(2)}\cdots (\gamma_{r})_{(2)}\ot_{A^\op} \gamma_{r+1} \ot_{A^\op} \cdots \ot_{A^\op}  \gamma_{r+u+s}\big) \Big)\\
        ={}&\big((\gamma_0)_{(1)} \ot_{A^\op} \cdots\ot_{A^\op} (\gamma_r)_{(1)}\big)\ot_A \big( (\gamma_0)_{(2)}\cdots (\gamma_{r})_{(2)} \ot_{A^\op} (\gamma_{r+1})_{(1)} \ot_{A^\op}  \cdots\ot_{A^\op} 
        (\gamma_{r+u})_{(1)} \big)\\
      &  \ot_A \big((\gamma_0)_{(3)}\cdots \ (\gamma_r)_{(3)}(\gamma_{r+1})_{(2)}\cdots  (\gamma_{r+u})_{(2)}\ot_{A^\op} \gamma_{r+u+1} \ot_{A^\op} \cdots \ot_{A^\op}  \gamma_{r+u+s}\big).
    \end{align*}
    Thus, $$(\varDelta^{B^A_\bullet}_{r,u}\ot_A \id_{B^A_s})\circ \varDelta^{B^A_\bullet}_{r+u,s}=(\id_{B^A_r} \ot_A \varDelta^{B^A_\bullet}_{u,s})\circ \varDelta^{B^A_\bullet}_{r,u+s},$$ and the diagonal map $\varDelta^{B^A_\bullet}$ is coassociative.

\medskip

    \textbf{Step 4}: We show that the diagonal map $\varDelta^{B^A_\bullet}$ is a chain map, that is, for any $r,s\ge0$, $$\varDelta^{B^A_\bullet}_{r,s}\rmd_{r+s+1}=(\id_{B^A_r} \ot_A  \rmd_{s+1})\varDelta^{B^A_\bullet}_{r,s+1}+(\rmd_{r+1}\ot_A \id_{B^A_s})\varDelta^{B^A_\bullet}_{r+1,s}.$$

    We begin with a claim: 
    \begin{claim}
        For any $r,s \ge0$, one has
        \begin{align*}
            &\big((\gamma_0)_{(1)} \ot_{A^\op} \cdots\ot_{A^\op} (\gamma_{r})_{(1)} \big) \ot_A \big((\gamma_0)_{(2)}\cdots (\gamma_r)_{(2)}\gamma_{r+1}  \ot_{A^\op} \gamma_{r+2} \ot_{A^\op} \cdots \ot_{A^\op}  \gamma_{r+s+1}\big)\\
    ={}& \big((\gamma_0)_{(1)} \ot_{A^\op} \cdots\ot_{A^\op} (\gamma_{r})_{(1)}\cdot \epsilon_\Gamma((\gamma_{r+1})_{(1)})^\op \big)\ot_A \big((\gamma_0)_{(2)}\cdots (\gamma_{r+1})_{(2)} \ot_{A^\op} \gamma_{r+2} \ot_{A^\op} \cdots \ot_{A^\op}  \gamma_{r+s+1}\big).
        \end{align*}
    \end{claim}
    \begin{proof}[Proof of the Claim]
        First, note that for any $a\in A,\gamma=\sum_p x_p\ot y_p^\op \in \Gamma$,  \begin{align*}
    (\gamma_r)_{(1)}\cdot a^\op\ot_A (\gamma_r)_{(2)}&=\sum_{p,i} (x_p)_{(0)}\ot \big(al_i((x_p)_{(1)})\big)^\op\ot_A r_i((x_p)_{(1)})\ot y_p^\op\\
    &\stackrel{Eq.\eqref{beta5'}}{=}\sum_{p,i} (x_p)_{(0)}\ot l_i((x_p)_{(1)})^\op\ot_A r_i((x_p)_{(1)})a\ot y_p^\op\\
    &= (\gamma_r)_{(1)}\ot_A (\gamma_r)_{(2)}(a\ot 1).
\end{align*}
        When $a=\epsilon_\Gamma((\gamma_{r+1})_{(1)})$, we have
        \[
        (\gamma_{r})_{(1)}\cdot \epsilon_\Gamma((\gamma_{r+1})_{(1)})^\op \ot_A (\gamma_r)_{(2)} (\gamma_{r+1})_{(2)} =(\gamma_{r})_{(1)} \ot_A (\gamma_r)_{(2)}\cdot (\epsilon_\Gamma((\gamma_{r+1})_{(1)})\ot 1_{B^\op})\cdot (\gamma_{r+1})_{(2)}= (\gamma_{r})_{(1)}\ot_A (\gamma_r)_{(2)}\gamma_{r+1}.
        \]
        We can view $$(\gamma_{r})_{(1)}\cdot \epsilon_\Gamma((\gamma_{r+1})_{(1)})^\op \ot_A (\gamma_r)_{(2)} (\gamma_{r+1})_{(2)}=(\gamma_{r})_{(1)}\ot_A (\gamma_r)_{(2)}\gamma_{r+1}\in \Gamma\ot_A \Gamma$$ as an element in $\Gamma\ot_A \Gamma\ot_A \Gamma$ by the inclusion 
        \begin{align*}
            \Gamma \ot_A \Gamma &\rightarrow (\Gamma \ot_A \Gamma) \ot_A \Gamma\\
            \gamma\ot_A \gamma' &\mapsto (1_\Gamma\ot_A \gamma)\ot_A \gamma'.
        \end{align*}
    So, we have 
    \begin{align*}
        &((\gamma_{r-1})_{(1)}\ot_A (\gamma_r)_{(1)}  \cdot \epsilon_\Gamma((\gamma_{r+1})_{(1)})^\op ) \ot_A (\gamma_{r-1})_{(2)}(\gamma_r)_{(2)} \\
       = &\gamma_{r-1}\cdot\big( (1_\Gamma\ot_A(\gamma_{r})_{(1)}\cdot \epsilon_\Gamma((\gamma_{r+1})_{(1)})^\op) \ot_A (\gamma_r)_{(2)} (\gamma_{r+1})_{(2)} \big) \\
       = & \gamma_{r-1}\cdot \big ((1_\Gamma\ot_A(\gamma_{r})_{(1)} )\ot_A (\gamma_r)_{(2)}\gamma_{r+1} \big)\\
       = & ( (\gamma_{r-1})_{(1)} \ot_A(\gamma_{r})_{(1)}) \ot_A (\gamma_{r-1})_{(2)} (\gamma_r)_{(2)}\gamma_{r+1}.
    \end{align*}
        By induction, we obtain 
        \begin{align*}
            &\big((\gamma_0)_{(1)} \ot_{A^\op} \cdots\ot_{A^\op} (\gamma_{r})_{(1)} \big) \ot_A \big((\gamma_0)_{(2)}\cdots (\gamma_r)_{(2)}\gamma_{r+1}  \big)\\
    ={}& \big((\gamma_0)_{(1)} \ot_{A^\op} \cdots\ot_{A^\op} (\gamma_{r})_{(1)}\cdot \epsilon_\Gamma((\gamma_{r+1})_{(1)})^\op \big)\ot_A \big((\gamma_0)_{(2)}\cdots (\gamma_{r+1})_{(2)} \big).
        \end{align*}
       and hence,  \begin{align*}
    &\big((\gamma_0)_{(1)} \ot_{A^\op} \cdots\ot_{A^\op} (\gamma_{r})_{(1)} \big) \ot_A \big((\gamma_0)_{(2)}\cdots (\gamma_r)_{(2)}\gamma_{r+1}  \ot_{A^\op} \gamma_{r+2} \ot_{A^\op} \cdots \ot_{A^\op}  \gamma_{r+s+1}\big)\\
    ={}& \big((\gamma_0)_{(1)} \ot_{A^\op} \cdots\ot_{A^\op} (\gamma_{r})_{(1)}\cdot \epsilon_\Gamma((\gamma_{r+1})_{(1)})^\op \big)\ot_A \big((\gamma_0)_{(2)}\cdots (\gamma_{r+1})_{(2)} \ot_{A^\op} \gamma_{r+2} \ot_{A^\op} \cdots \ot_{A^\op}  \gamma_{r+s+1}\big).
\end{align*}
     This completes the proof of the claim.\end{proof}

     Now we show that the diagonal map $\varDelta^{B^A_\bullet}$ is a chain map. 
        
   We only consider the case  $r,s\ge1$, the cases $r=0$ or $s=0$ being similar. 
   
    For $\gamma_0\ot_{A^\op} \cdots \ot_{A^\op} \gamma_{r+s+1}\in B^A_{r+s+1}(\Gamma;\Gamma;A)$, one has
    \begin{align*}
          & \varDelta^{B^A_\bullet}_{r,s}\circ \rmd_{r+s+1}(\gamma_0\ot_{A^\op} \cdots \ot_{A^\op} \gamma_{r+s+1})\\
          ={}&\varDelta^{B^A_\bullet}_{r,s}(\sum_{i=0}^{r+s} (-1)^i\gamma_0\ot_{A^\op} \cdots \ot_{A^\op} \gamma_i\gamma_{i+1}\ot_{A^\op} \cdots \ot_{A^\op} \gamma_{r+s+1}+ (-1)^{r+s+1}\gamma_0 \ot_{A^\op} \cdots\ot_{A^\op} \gamma_{r+s}\cdot \epsilon_\Gamma(\gamma_{r+s+1})^\op)\\
          ={}& \sum_{i=0}^r (-1)^i\big( (\gamma_0)_{(1)} \ot_{A^\op}\cdots \ot_{A^\op} (\gamma_i)_{(1)}(\gamma_{i+1})_{(1)}\ot_{A^\op} \cdots\ot_{A^\op}(\gamma_{r+1})_{(1)} \big)\ot_A \big((\gamma_0)_{(2)}\cdots (\gamma_{r+1})_{(2)} \ot_{A^\op} \cdots \ot_{A^\op} \gamma_{r+s+1} \big)\\
          +{}&\sum_{i=1}^{s} (-1)^{r+i} \big((\gamma_0)_{(1)} \ot_{A^\op}\cdots \ot_{A^\op}(\gamma_{r})_{(1)} \big)\ot_A \big( (\gamma_0)_{(2)}\cdots (\gamma_{r})_{(2)}\ot_{A^\op} \gamma_{r+1} \ot_{A^\op} \cdots \ot_{A^\op} \gamma_{r+i}\gamma_{r+i+1}\ot_{A^\op} \cdots  \ot_{A^\op} \gamma_{r+s+1} \big)\\
          +{}& (-1)^{r+s+1}\big((\gamma_0)_{(1)} \ot_{A^\op}\cdots \ot_{A^\op} (\gamma_{r})_{(1)} \big) \ot_A \big((\gamma_0)_{(2)}\cdots (\gamma_r)_{(2)}\ot_{A^\op} \gamma_{r+1} \ot_{A^\op} \cdots \ot_{A^\op}  \gamma_{r+s}\cdot \epsilon_\Gamma(\gamma_{r+s+1})^\op \big),
    \end{align*}
    while
    \begin{align*}
        &(\id_{B^A_r} \ot_A  \rmd_{s+1})\circ \varDelta^{B^A_\bullet}_{r,s+1}(\gamma_0\ot_{A^\op} \cdots \ot_{A^\op} \gamma_{r+s+1})\\
        ={}&(\id_{B^A_r} \ot_A  \rmd_{s+1}) \Big(\big( (\gamma_0)_{(1)} \ot_{A^\op} \cdots\ot_{A^\op} (\gamma_{r})_{(1)} \big)\ot_A \big((\gamma_0)_{(2)}\cdots (\gamma_r)_{(2)}\ot_{A^\op} \gamma_{r+1} \ot_{A^\op} \cdots \ot_{A^\op}  \gamma_{r+s+1} \big) \Big)\\
        ={}&(-1)^r \big((\gamma_0)_{(1)} \ot_{A^\op} \cdots\ot_{A^\op} (\gamma_{r})_{(1)} \big) \ot_A \big((\gamma_0)_{(2)}\cdots (\gamma_r)_{(2)}\gamma_{r+1}  \ot_{A^\op} \gamma_{r+2} \ot_{A^\op} \cdots \ot_{A^\op}  \gamma_{r+s+1}\big)\\
        +{}& \sum_{i=1}^s ((-1)^{r+i} (\gamma_0)_{(1)} \ot_{A^\op} \cdots\ot_{A^\op} (\gamma_{r})_{(1)} )\ot_A \big((\gamma_0)_{(2)}\cdots (\gamma_r)_{(2)}\ot_{A^\op} \gamma_{r+1} \ot_{A^\op} \cdots \ot_{A^\op} \gamma_{r+i}\gamma_{r+i+1}  \ot_{A^\op} \cdots \ot_{A^\op}  \gamma_{r+s+1}\big)\\
        +{}&(-1)^{r+s+1} \big( (\gamma_0)_{(1)} \ot_{A^\op} \cdots\ot_{A^\op} (\gamma_{r})_{(1)} \big)\ot_A \big((\gamma_0)_{(2)}\cdots (\gamma_r)_{(2)}\ot_{A^\op} \gamma_{r+1} \ot_{A^\op} \cdots \ot_{A^\op}  \gamma_{r+s}\cdot \epsilon_\Gamma(\gamma_{r+s+1})^\op \big).
    \end{align*}
    and
    \begin{align*}
        &(\rmd_{r+1} \ot_A \id_{B^A_s})\circ \varDelta^{B^A_\bullet}_{r+1,s}(\gamma_0\ot_{A^\op} \cdots \ot_{A^\op} \gamma_{r+s+1})\\
        ={}&(\rmd_{r+1} \ot_A \id_{B^A_s}) \Big(\big((\gamma_0)_{(1)} \ot_{A^\op} \cdots\ot_{A^\op} (\gamma_{r+1})_{(1)} \big)\ot_A \big((\gamma_0)_{(2)}\cdots (\gamma_{r+1})_{(2)}\ot_{A^\op} \gamma_{r+2} \ot_{A^\op} \cdots \ot_{A^\op}  \gamma_{r+s+1} \big) \Big)\\
        ={}& \sum_{i=0}^r (-1)^{i} \big((\gamma_0)_{(1)} \ot_{A^\op} \cdots \ot_{A^\op}(\gamma_i)_{(1)} (\gamma_{i+1})_{(1)} \ot_{A^\op} \cdots  \ot_{A^\op} (\gamma_{r+1})_{(1)} \big) \ot_A  \big((\gamma_0)_{(2)}\cdots (\gamma_{r+1})_{(2)} \ot_{A^\op} \cdots \ot_{A^\op}  \gamma_{r+s+1} \big)\\
        +{}&(-1)^{r+1} \big((\gamma_0)_{(1)} \ot_{A^\op} \cdots\ot_{A^\op} (\gamma_{r})_{(1)}\cdot \epsilon_\Gamma((\gamma_{r+1})_{(1)})^\op \big)\ot_A \big((\gamma_0)_{(2)}\cdots (\gamma_{r+1})_{(2)} \ot_{A^\op} \gamma_{r+2} \ot_{A^\op} \cdots \ot_{A^\op}  \gamma_{r+s+1} \big).
    \end{align*}
    
Applying the claim above,  we have 
\[
\varDelta^{B^A_\bullet}_{r,s}d_{r+s+1}(\gamma_0\ot_{A^\op} \cdots \ot_{A^\op} \gamma_{r+s+1})=\big( (\id_{B^A_r} \ot_A  d_{s+1})\varDelta^{B^A_\bullet}_{r,s+1}+(d_{r+1} \ot_A \id_{B^A_s})\varDelta^{B^A_\bullet}_{r+1,s} \big)(\gamma_0\ot_{A^\op} \cdots \ot_{A^\op} \gamma_{r+s+1}),
\]
which implies $$\varDelta^{B^A_\bullet}_{r,s}\rmd_{r+s+1}=(\id_{B^A_r} \ot_A  \rmd_{s+1})\varDelta^{B^A_\bullet}_{r,s+1}+(\rmd_{r+1}\ot_A \id_{B^A_s})\varDelta^{B^A_\bullet}_{r+1,s}.$$

We are done!
\end{proof}
    
\begin{corollary}\label{prop: the products are coassociative}
     Let $K_\bullet$ be the  relative one-sided bar resolution $B_\bullet^A(\Gamma;\Gamma;A)$ of the  left $\Gamma$-module $A$. The following statements hold.
    \begin{itemize}
        \item[(a)] The map  $\varDelta^{K_\bullet'}$ induced by $\varDelta^{K_\bullet}$ in Proposition \ref{prop: the diagonal map of B^A} is a coassociative  chain map.
        \item[(b)] The map $\varDelta^{P_\bullet}$ induced by $\varDelta^{K'_\bullet}$ in (a) and $\varDelta^{L_\bullet}$ in Section \ref{section: Constructing Cup products in cohomology theories} is a coassociative chain map.
        \item[(c)] The exterior products $\times_\ell, \times_r$ are associative and commute with differentials, with the above choices  of $K_\bullet$, $L_\bullet$ (and also $P_\bullet$) and their  diagonal maps $\varDelta^{K_\bullet}, \varDelta^{L_\bullet}$ (and also $ \varDelta^{P_\bullet}$).
    \end{itemize}
\end{corollary}
\begin{proof}
    \begin{itemize}
        \item[(a)]
     First, the differential $\rmd^{K'_\bullet}$ is given by $\id_{B^e}\ot_{A^e} \rmd^{K_\bullet}$. The fact that $\varDelta^{K'_\bullet}$ is a chain map  is immediately followed  by Proposition \ref{prop: the diagonal map of B^A}. 
  
    Next, for $r,s,v\ge 0, b^e=b\ot b'^\op\in B^e$ and $k_{r+s+v}\in K_{r+s+v}$, by the coassociativity of $\varDelta^{K_\bullet}$, one has 
    
    \[
    (\varDelta^{K_\bullet}_{r,v}\ot \id_{K_s})\circ\varDelta^{K_\bullet}_{r+v,s}(k_{r+v+s})=(\id_{K_r}\ot \varDelta^{K_\bullet}_{v,s})\circ \varDelta^{K_\bullet}_{r,v+s}(k_{r+v+s})=\sum k_r\ot k_v \ot k_s
    \]
    Then, 
    \begin{align*}
       &( \varDelta^{K'_\bullet}_{r,v}\ot_B \id_{K'_s})\circ \varDelta^{K'_\bullet}_{r+v,s}(b^e\ot_{A^e} k_{r+v+s})\\
       =& ( \varDelta^{K'_\bullet}_{r,v}\ot_B \id_{K'_s}) \Big(\sum \big((b\ot 1_{B^\op}) \ot_{A^e} k_{r+v} \big)\ot_B \big((1_B\ot b'^\op)\ot_{A^e} k_{s} \big) \Big)\\
       =& \sum \big((b \ot 1_{B^\op})\ot_{A^e} k_r \big)\ot_B \big((1_B\ot 1_{B^\op}) \ot_{A^e} k_v \big) \ot_B \big((1_B\ot b'^\op)\ot_{A^e} k_s \big),
    \end{align*}
    and 
    \begin{align*}
       &( \id_{K'_r} \ot_B \varDelta^{K'_\bullet}_{v,s})\circ \varDelta^{K'_\bullet}_{r,v+s}(b^e\ot_{A^e} k_{r+v+s})\\
       =& ( \id_{K'_r} \ot_B \varDelta^{K'_\bullet}_{v,s}) \Big(\sum \big((b\ot 1_{B^\op}) \ot_{A^e} k_{r} \big)\ot_B \big((1_B\ot b'^\op)\ot_H k_{v+s} \big) \Big)\\
       =& \sum \big((b \ot 1_{B^\op})\ot_{A^e} k_r \big)\ot_B \big((1_B\ot 1_{B^\op}) \ot_{A^e} k_v \big) \ot_B \big( (1_B\ot b'^\op)\ot_{A^e} k_s \big),
    \end{align*}
    Thus, the above two composites are equal. Therefore, $\varDelta^{K'_\bullet}$ is coassociative.
    
    \item[(b)] 
    By definition, $\varDelta^{P_\bullet}=\alpha\circ(\varDelta^{L_\bullet}\ot_H \varDelta^{K'_\bullet})$, it only needs to show that the map $\alpha:(L_\bullet\ot L_\bullet)\ot_H (K'_\bullet\ot_A K'_\bullet)\to (L_\bullet\ot_H K'_\bullet)\ot_B  (L_\bullet\ot_H K'_\bullet)$ is a chain map since $\varDelta^{L_\bullet}$ and $\varDelta^{K'_\bullet}$ are both chain maps by subsection \ref{subsection: Cup products in the cohomology theory of a Hopf algebra} and  (a).

    Let $k'_i\in K'_i$ and  $l_j\in L_j$, for $i,j\ge0$. For any $z=(l_p\ot l_q) \ot_H(k'_r\ot_A k'_s) \in P_{p+r+q+s}$ , we have
    \begin{align*}
        &\alpha\circ\rmd^{(L_\bullet \ot L_\bullet) \ot_H (K'_\bullet \ot_A K'_\bullet)}_{p+r+q+s}(z)\\
        =& \alpha\circ(\rmd^{L_\bullet\ot L_\bullet}_{p+q}\ot \id_{K'_\bullet\ot_A K'_\bullet} + \id_{L_\bullet\ot L_\bullet}\ot \rmd^{K'_\bullet\ot_A K'_\bullet}_{r+s})(z)\\
        =& \alpha \big(\sum (\rmd^{L_\bullet}_p (l_p) \ot l_q) \ot_H(k'_r\ot_A k'_s) +(-1)^p (l_p\ot \rmd_q^{L_\bullet}(l_q)) \ot_H(k'_r\ot_A k'_s) \\
        +& (-1)^{p+q}(l_p\ot l_q) \ot_H(\rmd^{K'_\bullet}_r(k'_r)\ot_A k'_s) + (-1)^{p+q+r} (l_p\ot l_q) \ot_H(k'_r\ot_A \rmd_s(k'_s)) \big)\\
        =&\sum (-1)^{qr} (\rmd^{L_\bullet}_p(l_p) \ot_H k'_r)\ot_B (l_q\ot_H k'_s) +(-1)^{(q-1)r+p} (l_p \ot_H k'_r)\ot_B (\rmd^{L_\bullet}_q(l_q)\ot_H k'_s)\\
        +&(-1)^{p+q+q(r-1)}(l_p \ot_H  \rmd^{K'_\bullet}_r (k'_r))\ot_B (l_q\ot_H k'_s) +(-1)^{p+q+r+qr}(l_p \ot_H k'_r)\ot_B (l_q\ot_H \rmd^{K'_\bullet}_s(k'_s)),
    \end{align*}
    and 
    \begin{align*}
        &\rmd^{P_\bullet\ot_B P_\bullet}_{p+r+q+s}\circ\alpha(z)\\
        =&\sum (-1)^{qr} \rmd^{P_\bullet\ot_B P_\bullet}_{p+r+q+s} \big((l_p\ot_{H}k'_r )\ot_B (l_q\ot_{H} k'_s) \big)\\
        =&\sum (-1)^{qr}(\rmd^{P_\bullet}_{p+r} \ot_B \id_{P_{q+s}} + \id_{P_{p+r}} \ot_B \rmd^{P_\bullet}_{q+s}) \big((l_p\ot_{H}k'_r )\ot_B (l_q\ot_{H} k'_s) \big)\\
        =& \sum (-1)^{qr} (\rmd^{L_\bullet}_p(l_p) \ot_{H} k'_r)\ot_B (l_q\ot_{H} k'_s) + (-1)^{qr+p}(l_p \ot_{H}  \rmd^{K'_\bullet}_r (k'_r))\ot_B (l_q\ot_{H} k'_s)\\
        +& (-1)^{qr+p+r} (l_p \ot_{H} k'_r)\ot_B (\rmd^{L_\bullet}_q(l_q)\ot_{H} k'_s)  +(-1)^{p+q+r+qr}(l_p \ot_{H} k'_r)\ot_B (l_q\ot_{H} \rmd^{K'_\bullet}_s(k'_s))
    \end{align*}
    Thus, the two composites agree on all elementary tensors and hence are equal. Therefore, $\alpha$ is a chain map. Next, for $p,q,r,s,u,v\ge 0$, $l_{p+u+q}\in L_{p+u+q}$ and $k'_{r+s+v}\in K'_{r+s+v}$, by the coassociativity of $\varDelta^{L_\bullet}$ and $\varDelta^{K'_\bullet}$, one has 
    \[
    (\varDelta^{L_\bullet}_{p,u}\ot \id_{L_q})\circ\varDelta^{L_\bullet}_{p+u,q}(l_{p+u+q})=(\id_{L_p}\ot \varDelta^{L_\bullet}_{u,q})\circ \varDelta^{L_\bullet}_{p,u+q}(l_{p+u+q})=\sum l_p\ot l_u \ot l_q
    \]
    and 
    \[
    (\varDelta^{K'_\bullet}_{r,v}\ot \id_{K'_s})\circ\varDelta^{K'_\bullet}_{r+v,s}(k'_{r+v+s})=(\id_{K'_r}\ot \varDelta^{K'_\bullet}_{v,s})\circ \varDelta^{K'_\bullet}_{r,v+s}(k'_{r+v+s})=\sum k'_r\ot k'_v \ot k'_s
    \]
    Then, 
    \begin{align*}
       &( \varDelta^{P_\bullet}_{p,r;u,v}\ot_B \id_{P_{q+s}})\circ \varDelta^{P_\bullet}_{p+u,r+v;q,s}(l_{p+u+q}\ot_H k'_{r+v+s})\\
       =& (-1)^{q(r+v)}( \varDelta^{P_\bullet}_{p,u;r,s}\ot_B \id_{P_{q+s}}) \big(\sum (l_{p+u} \ot_H k'_{r+v} )\ot_B (l_q\ot_H k'_{s}) \big)\\
       =&(-1)^{q(r+v)+ru} \sum (l_p\ot_H k'_r)\ot_B (l_u \ot_H k'_v) \ot_B (l_q\ot_H k'_s),
    \end{align*}
    and 
    \begin{align*}
       &( \id_{P_{p+r}} \ot_B \varDelta^{P_\bullet}_{u,v;q,s})\circ \varDelta^{P_\bullet}_{p,r;u+q,v+s}(l_{p+u+q}\ot_H k'_{r+v+s})\\
       =& (-1)^{r(u+q)}( \id_{P_{p+r}} \ot_B \varDelta^{P_\bullet}_{u,q;v+s}) \big(\sum (l_{p} \ot_H k'_{r} )\ot_B (l_{u+q}\ot_H k'_{v+s}) \big)\\
       =&(-1)^{r(u+q)+qv} \sum (l_p \ot_H  k'_r) \ot_B ( l_u \ot_H k'_v) \ot_B ( l_q\ot_H k'_s ).
    \end{align*}
    Thus, the above two composites are equal. Therefore, $\varDelta^{P_\bullet}$ is coassociative.

    \item[(c)] This is an immediate consequence of the coassociativity of $\varDelta^{L_\bullet}$, part (a) and part (b) by Proposition~\ref{prop:hom-lax}.
    \end{itemize}
\end{proof}

 Combining all the results of this section, we obtain the following result:
\begin{theorem}
    \label{thm: isom of bicomplexes compatible with associative cup products}
  Let $B/A$ be a faithfully flat $H$-Galois extension.
Under appropriate choices of   $K_\bullet$, $L_\bullet$ (and also $P_\bullet$) as well as the diagonal maps $\varDelta^{K_\bullet}, \varDelta^{L_\bullet}$ (and also $ \varDelta^{P_\bullet}$) constructed above, the isomorphism
    \[
    \Phi^{\bullet,\bullet}: \Hom_H(L_\bullet,\Hom_{A^e}(K_\bullet, -))\cong \Hom_{B^e}(L_\bullet\otimes_H(B^e\otimes_{A^e}K_\bullet), -)
    \]
    is  a natural  isomorphism of  first quadrant bicomplex lax monoidal functors in the sense of Definition~\ref{def: bicomplex lax monoidal functor}.   
\end{theorem}

\begin{proof} Note that the two exterior products $\times_\ell$ and $\times_r$ are constructed by Proposition~\ref{prop:hom-lax}.
    This result  is thus a combination of Propositions \ref{prop: isom of bicomplexes}, \ref{proposition: Phi commutes with cup product} and Corollary \ref{prop: the products are coassociative}.
\end{proof}

\subsection{A functorial multiplicative spectral sequence}

Using the isomorphism $\Phi^{\bullet, \bullet}$, we give a functorial multiplicative spectral sequence which is isomorphic to the \c{S}tefan spectral sequence shown in Theorem \ref{theorem: two contructions are isomorphic}


\begin{theorem}
\label{theorem: A new proof of Stefan spectral sequence}
    Let $H$ be a Hopf algebra with a bijective antipode, $B/A$ be a faithfully flat $H$-Galois extension and $N$ be a $(B, B)$-bimodule. Then there exists  a multiplicative convergent spectral sequence:
    \[
    \rmE^{p,q}_{2}=\rmH^p(H,\HH^q(A,N))\Rightarrow \HH^{p+q}(B,N).
    \]
   More precisely,  there exists a functorial multiplicative spectral sequence with 
   $$\calE^{p,q}_{2}=\rmH^p(H,\HH^q(A,-)): B^e\Mod\to \Vect$$
which converges multiplicatively to the filtered graded lax monoidal functor $$\{\HH^{n}(B,-): B^e\Mod \to\Vect\}_{n\in \bbN}.$$


\end{theorem}

\begin{proof} As in Section~\ref{Section: An isomorphism of   differential bigraded algebras}, let $L_\bullet$ be a projective resolution of the trivial right $H$-module $\bfk$ and  $K_\bullet$ be a  projective resolution of the left $\Gamma$-module $A$.
    Consider the bicomplex   $$C^{\bullet,\bullet}=(\{C^{p,q}=\Hom_H(L_p,\Hom_{A^e}(K_q,N))\}_{ p, q\geq 0}, \rmd_h, \rmd_v).$$ Since $L_p$ is projective as a right $H$-module, $\Hom_H(L_p,-)$ is an exact functor, and  by Proposition~\ref{proposition: Gamma is projective as A^e-module},  $K_\bullet$ is also a projective resolution of the left $A^e$-module $A$. Hence, the spectral sequence associated with the column filtration of $C^{\bullet,\bullet}$ has
\[
\rmE^{p,q}_{1} = \Hom_H(L_p,\HH^q(A,N)) \qquad\text{and}\qquad
\rmE^{p,q}_{2} = \Ext^p_H\bigl(\bfk,\HH^q(A,N)).
\]
It converges to $\rmH^{p+q}(\Tot(C^{\bullet,\bullet}))$.

    Next, consider the bicomplex 
    \[
    D^{\bullet,\bullet}=( \{D^{p,q}:=\Hom_{B^e}(L_p\otimes_H(B^e\otimes_{A^e}K_q),N) \}_{p,q\ge 0} ,\delta_h,\delta_v).
    \] 
    By Proposition \ref{prop: isom of bicomplexes}, $C^{\bullet,\bullet}$ is isomorphic to $D^{\bullet,\bullet}$ as bicomplexes, thus $\rmH^{n} (\Tot(C^{\bullet,\bullet})) \cong \rmH^{n} (\Tot(D^{\bullet,\bullet}))$ for any $n\geq 0$. By Proposition \ref{Prop: the total complex is a projective resolution of (B,B) bimodule},  the total complex $P_\bullet$
    is a  projective resolution of $B$ as a $(B, B)$-bimodule, thus $\rmH^{n} (\Tot(C^{\bullet,\bullet})) \cong \rmH^{n} (\Tot(D^{\bullet,\bullet}))=\HH^n(B,N)$.
    Therefore,  we obtain a convergent  spectral sequence:
    \begin{equation}\label{equation: the spectral sequence}
    \rmE^{p,q}_2=\rmH^p(H,\HH^q(A,N))\Rightarrow \HH^{p+q}(B,N).
    \end{equation}

      The multiplicative property follows from    Theorem \ref{thm: isom of bicomplexes compatible with associative cup products} and Corollary \ref{corollary: multiplicative for spectral sequence induced by a bicomplex}.
\end{proof}

\bigskip

\section{The multiplicative property of  the \c{S}tefan spectral sequence}\label{Section: The multiplicative property of Stefan Spectral Sequence}

\subsection{Identifying spectral sequences}

Next, we show that the spectral sequence \eqref{equation: the spectral sequence} is isomorphic to the \c{S}tefan spectral sequence from the second page by using \cite[Theorem 31 and Theorem 34]{Kunzer08}.
We first recall the definition of $(G,F)$-acyclic.
Let $\calA$ be an  abelian category. Denote by $C(\calA)$ (resp. $  C^{\geq 0}(\calA), C_{\geq 0}(\calA)$)  the category of complexes in $\calA$ (resp. the category of  cochain complexes concentrated in nonnegative degrees, the category of   chain complexes concentrated in nonnegative degrees).

\begin{definition}
    Let $\calA$ and $\calB$  be abelian categories with enough injective objects, $\calC$ be  an abelian category. Let $G:\calA\rightarrow \calB, F:\calB\rightarrow \calC$  be left exact functors. A complex $C^\bullet\in C(\calA)$ is called  $(G,F)$-acyclic if the following three conditions hold:
    \begin{myenumerate}
        \item $C^i$ is $G$-acyclic for any $i$, that is, for each $j\geq 1$, $R^j G(C^i)=0$;
        \item  $C^i$ is $F\circ G$-acyclic for any $i$;
        \item $G(C^i)$ is $F$-acyclic for any $i$.
    \end{myenumerate}
\end{definition}

\begin{theorem}[{\cite[Theorem 31]{Kunzer08}}]\label{Theorem: kunzer thm 31}
   Let  $\calA, \calA', \mathcal{B}$ and $\calC$ be  abelian categories such that  $\calA$ has enough projective objects and $ \calA', \mathcal{B}$ have enough injective objects.
    Let $G: \calA^{\op} \times \calA' \rightarrow  \mathcal{B}$ be a biadditive functor and  $F:\mathcal{B} \rightarrow \mathcal{C}$ be an additive functor.
    Suppose $X\in \mathcal{A}$ and $X'\in \calA'$.  Assume that  the following properties hold:
    \begin{itemize}
        \item[(a)] The functor $G(-,X'): \calA^{\op} \rightarrow \mathcal{B} $ is left exact;

        \item[(a')] the functor $ G(X,-) : \calA' \rightarrow  \mathcal{B}$ is left exact;

        \item[(b)] the functor $F$ is left exact;

        \item[(c)] the object $X$ has a $(G(-,X'),F)$-acyclic resolution $K_\bullet\in  \rmC_{\ge 0}(\calA)$, where $G(K_i,-)$ is exact for any $i \ge 0$;

        \item[(c')] the object $X'$ has a $(G(X,-),F)$-acyclic coresolution $I^\bullet\in  \rmC^{\ge 0}(\calA')$, where $G(-,I^i)$ is exact for any $i \ge 0$.
    \end{itemize}
    Then the Grothendieck spectral sequence for the functors $G(X,-)$ and $F$, evaluated at $X'$, is isomorphic, from the second page, to the Grothendieck spectral sequence for the functors $G(-,X')$ and $F$, evaluated at $X$ .
\end{theorem}

In our setup, let $B/A$ be a faithfully flat Hopf Galois extension over a Hopf algebra $H$ with a bijective antipode. Let $N$ be a $(B, B)$-bimodule and let $F:=\Hom_H(\bfk,-): H\Mod \rightarrow \bfk\Mod , G:=\Hom_{A^e}(-,-):(\Gamma\Mod)^\op \times  B^e\Mod \rightarrow H \Mod$. Denote $G_1:=\Hom_{A^e}(A,-): B^e\Mod \to  H \Mod$ and $G_2:=\Hom_{A^e}(-,N): (\Gamma\Mod)^\op \to  H \Mod$ for convenience.

    \begin{lemma}\label{lemma:  K is (Hom(-,N), Hom_H(k,-))-acyclic}
        Let $K_\bullet$ be a projective resolution of the left $\Gamma$-module $A$, then  $K_\bullet$ is $(G_2, F)$-acyclic.
    \end{lemma}
    \begin{proof}
        Since $K_\bullet$ is a projective $\Gamma$-module resolution, by Proposition~\ref{proposition: Gamma is projective as A^e-module},  $K_i$ is a projective $A^e$-module, thus $K_i$ is  $\Hom_{A^e}(-,N)$-acyclic. And by Proposition \ref{Prop: H inviant isom}, one has
        \[
        \Hom_H(\bfk,-)\circ\Hom_{A^e}(-,N)=\Hom_H(\bfk,\Hom_{A^e}(-,N)) \cong \Hom_\Gamma(-,N),
        \]
        which implies $K_i$ is $\Hom_H(\bfk,-)\circ\Hom_{A^e}(-, N)$-acyclic, as $K_i$ is a projective $\Gamma$-module.

        Finally, we need to show that $\Hom_{A^e}(K_i,N)$ is $\Hom_H(\bfk,-)$-acyclic, that is, $\Ext_H^j(\bfk,\Hom_{A^e}(K_i,N))=0$ for $j\ge 1$.
        Let $L_\bullet$  be a projective resolution of right $H$-module $\bfk$. We have
            \begin{align*}
            \Ext_H^j(\bfk,\Hom_{A^e}(K_i,N))&=\rmH^j(\Hom_H(L_\bullet,\Hom_{A^e}(K_i,N)))\\
           & \cong \rmH^j(\Hom_{B^e}(L_\bullet\otimes_H(B^e\otimes_{A^e} K_i ),N)).
        \end{align*}
        Since $B^e\otimes_{A^e} \Gamma\cong B\otimes H\otimes B$ as a left $H$ and $B^e$-module by Remark \ref{Remark: free H mod}, $L_j\otimes_H(B^e\ot_{A^e} K_i)$ is a direct summand of direct sums of  $  L_j\ot_H (B^e\ot_{A^e} \Gamma)\cong  L_j\otimes_H(B\ot H\ot B)\cong   L_j\otimes B^e$. Note that $L_j\otimes B^e$ is a free $B^e$-module and $B^e\ot_{A^e} K_i$ is a projective $H$-module.
        So $L_\bullet\ot_H(B^e\otimes_{A^e} K_i)$ is a projective resolution of the $B^e$-module $B^e\otimes_\Gamma K_i$.
        Then, 
        \[
        \Ext_H^j(\bfk,\Hom_{A^e}(K_i,N))=\Ext^j_{B^e}(B^e\ot_\Gamma K_i,N)=0 , ~~\text{for $j\ge 1$},
        \]
        since $B^e\ot_\Gamma K_i$ is a projective $B^e$-module. Hence, $K_\bullet$ is $(\Hom_{A^e}(-,N), \Hom_H(\bfk,-))$-acyclic.
    
        In summary,  $K_\bullet$ is $(\Hom_{A^e}(-,N), \Hom_H(\bfk,-))$-acyclic.
        \end{proof}

    \begin{lemma}\label{lemma: I is (Hom(A,-), Hom_H(k,-))-acyclic}
        Let $I^\bullet$ be an injective resolution of the left $B^e$-module $N$, then $I^\bullet$ is $(G_1, F)$-acyclic.
    \end{lemma}

    \begin{proof}
        Since $I^j$ is an injective $B^e$-module and $B/A$ is a flat extension, $\Hom_{A^e}(-,I^j)$ is an exact functor by the isomorphisms of functors:
        \[
        \Hom_{A^e}(-,I^j)\cong \Hom_{A^e}(-,\Hom_{B^e}(B^e,I^j))\cong \Hom_{B^e}(B^e\otimes_{A^e} -, I^j).
        \]
        Hence, $I^j$ is an injective $A^e$-module, so $I^j$ is $\Hom_{A^e}(A,-)$-acyclic. And by \cite[Lemma 3.6]{LRW25}, one has
        \begin{align*}
            \Hom_H(\bfk,-)\circ\Hom_{A^e}(A,-)
            &=\Hom_H(\bfk,\Hom_{A^e}(A,-))\\
            &\cong \Hom_\Gamma(A,-)\\
            &\cong \Hom_\Gamma(A,\Hom_{B^e}(B^e,-))\\
            &\cong \Hom_{B^e}(B^e\otimes_\Gamma A, -)\\
            &\cong \Hom_{B^e}(B,-),
        \end{align*}
        which implies that $I^j$ is $(\Hom_H(\bfk,-)\circ\Hom_{A^e}(A,-))$-acyclic since $I^j$ is an injective $B^e$-module.
        Finally, by \cite[Proposition 3.2]{Ste95}, one has $\Ext_H^q(\bfk,\Hom_{A^e}(A,I^j))=0$ for $q\ge 1$, which implies that $\Hom_{A^e}(A,I^j)$ is $\Hom_H(\bfk,-)$-acyclic.

        We have shown that  $I^\bullet$ is $(\Hom_{A^e}(A,-), \Hom_H(\bfk,-))$-acyclic.
    \end{proof}

So, Theorem \ref{Theorem: kunzer thm 31} applies to  faithfully flat Hopf-Galois extensions.

\begin{theorem}[{\cite[Theorem 34]{Kunzer08}}]\label{Theorem: kunzer thm 34}
    Let  $\calA, \calB', \mathcal{B}$ and $\calC$  be  abelian categories  such that $\calA$ and $\calB$ have   enough projective objects.
    Let $G':  \calA^{\op} \rightarrow  \mathcal{B}'$ be an additive functor and  $F':\mathcal{B}^{\op}\times \calB' \rightarrow \mathcal{C}$ be a biadditive functor.
    Let $X\in \mathcal{A}$, $Y\in \mathcal{B}$ and take a projective resolution $L_\bullet$   of $Y$.
   Assume that the following properties hold:
    \begin{itemize}
        \item[(a)]the functor $G': \calA^{\op} \rightarrow \mathcal{B}'$ is left exact;

        \item[(b)] the functor $ F'(Y,-) : \mathcal{B}' \rightarrow  \mathcal{C}$ is left exact;

        \item[(c)] the object $X$ has a $(G',F'(Y,-))$-acyclic resolution $K_\bullet$;

        \item[(d)] the functor $F'(L_i,-)$ is exact for all $i\ge 0$;

        \item[(e)] the functor $F'(-,I)$ is exact for any injective object $I\in \calB'$.
    \end{itemize}
    Then the Grothendieck spectral sequence for the functors $G'$ and $F'(Y,-)$, evaluated at $X$, is isomorphic to the  spectral sequence induced by the column filtration of the first quadrant bicomplex $F'(L_\bullet,G'(K_\bullet))$.
\end{theorem}

\begin{theorem}\label{theorem: two contructions are isomorphic}  Let $H$ be a Hopf algebra with a bijective antipode, $B/A$ be a faithfully flat $H$-Galois extension and $N$ be a $(B, B)$-bimodule.
    The spectral sequence defined in Theorem~\ref{theorem: A new proof of Stefan spectral sequence} is isomorphic to the \c{S}tefan spectral sequence from the second page.
\end{theorem}

\begin{proof}
    By \cite[Theorem 3.3]{Ste95},  the \c{S}tefan spectral sequence is  the Grothendieck spectral sequence
    \[
    \rmE^{p,q}_2=R^p(F)(R^q(G_1)(N))\Rightarrow R^{p+q}(F\circ G_1)(N),
    \]
    for $F=\Hom_H(\bfk,-): H\Mod\to \bfk\Mod, G_1=\Hom_{A^e}(A,-): B^e\Mod\to H\Mod$ and $$F\circ G_1\cong \Hom_{B^e}(B,-): B^e\Mod\to \bfk\Mod.$$
    By Lemmas~\ref{lemma:  K is (Hom(-,N), Hom_H(k,-))-acyclic} and \ref{lemma: I is (Hom(A,-), Hom_H(k,-))-acyclic},   Theorem \ref{Theorem: kunzer thm 31} applies to our setup, and   as a consequence, from the second page,  the
    \c{S}tefan spectral sequence is isomorphic to
    \[
    \rmE^{p,q}_2=R^p(F)(R^q(G_2)(A))\Rightarrow R^{p+q}(F\circ G_2)(A).
    \]

    In our setup, let $F':=\Hom_H(-,-):{(H \Mod)}^\op \times  H\Mod \rightarrow \bfk\Mod$, $G':=\Hom_{A^e}(-,N): (\Gamma\Mod)^{\op} \rightarrow H\Mod$,
$L_\bullet$ be a projective resolution of the trivial right $H$-module $\bfk$ and  $K_\bullet$ be a  projective resolution of the left $\Gamma$-module $A$.
Note that $F'(\bfk,-)=F=\Hom_H(\bfk,-)$ and $G'=G_2=\Hom_{A^e}(-,N)$.  Thus, by Lemma \ref{lemma:  K is (Hom(-,N), Hom_H(k,-))-acyclic}, $K_\bullet$ is $(\Hom_{A^e}(-,N),\Hom_H(\bfk,-))$-acyclic. So Theorem \ref{Theorem: kunzer thm 34} applies to our setup.   Hence,  from the second page, the spectral sequence \[
    \rmE^{p,q}_2=R^p(F)(R^q(G_2)(A))\Rightarrow R^{p+q}(F\circ G_2)(A)
    \] is isomorphic to the spectral sequence induced by the column  filtration of the bicomplex $F'(L_\bullet,G'(K_\bullet))$ which is exactly the spectral sequence \eqref{equation: the spectral sequence}.
    \end{proof}

Hence, we obtain the main result of this paper.
\begin{theorem}\label{Stefan spectral sequence is multiplicative}
    Let $H$ be a Hopf algebra with a bijective antipode, $B/A$ be a faithfully flat $H$-Galois extension. Then the \c{S}tefan spectral sequence is    a functorial multiplicative convergent spectral sequence.
\end{theorem}

\begin{proof}
   This is directly based on Theorem \ref{theorem: A new proof of Stefan spectral sequence} and Theorem \ref{theorem: two contructions are isomorphic}.
\end{proof}

\subsection{Examples}

\begin{example}[\textbf{\c{S}tefan spectral sequence for strongly graded algebras}]\label{ex: Stefan spectral Sequence for strongly graded algebras}
Let $G$ be a group with identity element $e$ and $B$ be a strongly $G$-graded algebra. 
By Example~\ref{Ex: Strongly G-graded algebras are faithfully flat}, $B/B_e$ is a faithfully flat $\bfk G$-Galois extension. Hence, by Theorem \ref{Stefan spectral sequence is multiplicative}, there is   a functorial multiplicative convergent spectral sequence:
    \[
    \calE^{p,q}_2=\rmH^p( G, \HH^q(B_e,-))\Rightarrow \HH^{p+q}(B ,-).
    \]
     It seems that  this multiplicative property is new even for strongly graded algebras.
\end{example}

\begin{example}
    [\textbf{\c{S}tefan spectral sequence for smash product algebras}]\label{ex: Stefan spectral Sequence for smash product algebras}
    Let $H$ be a Hopf algebra with a bijective antipode, $A$ be an $H$-module algebra and $B=A\# H$ be a smash product algebra. 
    By Example~\ref{ex: smash product is faithfully flat},  $B/A$ is a faithfully flat $H$-Galois extension. Hence, by Theorem \ref{Stefan spectral sequence is multiplicative}, there is   a functorial multiplicative convergent spectral sequence:
    \[
    \calE^{p,q}_2=\rmH^p(H, \HH^q(A,-))\Rightarrow \HH^{p+q}(A\# H ,-), 
    \]
    which is consistent with   \cite{Neg15}, although \cite{Neg15} uses the narrower Definitions \ref{def: multiplicative spectral sequence} and \ref{definition: Classical multiplicative convergence}.
\end{example}

\begin{example}[\textbf{\c{S}tefan spectral sequence for crossed product algebras}]\label{ex: Stefan spectral Sequence for crossed product algebras}
\label{Ex: Stefan spectral sequence for a crossed product algebra}
Let $H$ be a Hopf algebra with a bijective antipode, $A$ be a  $\bfk$-algebra  such that $H$ measures $A$.
Let  $\sigma \in \Hom_\bfk(H\ot H,A)$ be an invertible map with respect to the convolution product.  Assume that
$A$ is a twisted $H$-module and that  $\sigma$ is a Hopf 2-cocycle.   By Example~\ref{Ex: Crossed product algebras are faithfully flat}, the crossed product algebra
    $A\#_\sigma H$ is  a faithfully flat $H$-Galois extension over $A$.

    Hence, by Theorem \ref{Stefan spectral sequence is multiplicative}, there is   a functorial multiplicative convergent spectral sequence:
    \[
    \calE^{p,q}_2=\rmH^p(H, \HH^q(A,-))\Rightarrow \HH^{p+q}(A\#_\sigma H, -).
    \]
   
    It seems that this multiplicative property is even new in this example.
\end{example}

The following two examples are in fact special cases of Example~\ref{Ex: Stefan spectral sequence for a crossed product algebra}.

\begin{example}[\textbf{Lyndon-Hochschild-Serre spectral sequence for a group extension}]\label{ex: Lyndon-Hochschild-Serre spectral sequence for a group extension}

    Let  $G$ be  a group, $N$  be a normal subgroup of $G$ and $M$ be a $\bfk$-representation of  $G$. By \cite[Example 7.1.6]{Mont93},  $\bfk G$ is isomorphic to a crossed product $\bfk N\#_\sigma \bfk(G/N)$,  which is a faithfully flat $\bfk (G/N)$-Galois extension over $\bfk N$ by Example~\ref{Ex: Crossed product algebras are faithfully flat}.

    Recall that a representation $M$ of $G$ is exactly a left module over $\bfk G$ and the group cohomology $\rmH^*(G, M)$ is exactly the Hopf algebra cohomology $\rmH^*(\bfk G, M)$.
    Moreover,  $M$ can be considered as a $(\bfk G, \bfk G)$-bimodule $M$ where the right   $G$-action on $M$ is the trivial action and the group cohomology $\rmH^\bullet(G, M)$ is isomorphic to the Hochschild cohomology $\rmHH^\bullet(\bfk G,   M)$.   In this case, the \c{S}tefan spectral sequence coincides with the famous Lyndon-Hochschild-Serre spectral sequence:
    \[
    \rmE^{p,q}_2=\rmH^p(G/N, \rmH^q(N,M))\Rightarrow \rmH^{p+q}(G,M).
    \]
   By Theorem~\ref{Stefan spectral sequence is multiplicative}, the Lyndon-Hochschild-Serre spectral sequence
    is   a functorial multiplicative convergent spectral sequence, but this is a well-known fact due to Beyl \cite{Beyl81}, although  the narrower Definitions \ref{def: multiplicative spectral sequence} and \ref{definition: Classical multiplicative convergence} are used in \cite{Beyl81}.
\end{example}

\begin{example}[\textbf{Hochschild-Serre spectral sequence for an extension of Lie algebras}]\label{ex: Chevalley-Eilenberg spectral sequence for an extension of Lie algebras}
    Let $$0\rightarrow \mathfrak{h}\rightarrow \mathfrak{g}\rightarrow \mathfrak{g/\mathfrak{h}}\rightarrow 0$$  be an extension of Lie algebras over the field $\bfk$ and $M$ be a representation of $\frakg$.
    Recall that a representation $M$ of $\frakg$ is exactly a left module over $U(\mathfrak{g})$, the universal enveloping algebra of $\frakg$,  and the  Lie algebra  cohomology $\rmH^*(\frakg, M)$ is exactly the Hopf algebra cohomology $\rmH^*(U(\frakg), M)$.
    Moreover,   $M$ can be considered as a $(U(\mathfrak{g}), U(\mathfrak{g}))$-bimodule $M$ where  the right $U(\mathfrak{g})$-module structure is trivial and  the Lie algebra  cohomology $\rmH^\bullet( \mathfrak{g} ,  M) $ is isomorphic to the Hochschild cohomology $\rmHH^\bullet(U ( \mathfrak{g} ),  M)$.

    Similar to the case of group extensions,
    by \cite[Corollary 7.2]{Mont93}, $U(\mathfrak{g})$ is isomorphic to a crossed product $U(\mathfrak{h)}\#_\sigma U(\mathfrak{g/h})$, which is a faithfully flat $U(\mathfrak{g/h})$-Galois extension  by Example~\ref{Ex: Crossed product algebras are faithfully flat}.  Hence, we obtain a  functorial multiplicative convergent spectral sequence:
    \[ \calE^{p,q}_2=\rmH^p(\mathfrak{g/\mathfrak{h}},\rmH^q(\mathfrak{h},-))\Rightarrow \rmH^{p+q}(\mathfrak{g},-),
    \]
    which coincides with  the Hochschild-Serre spectral sequence for an extension of Lie algebras.  
 
However, this is a well-known fact; see, for instance, \cite[Theorem 1.5.1]{Fuks86}, although  the narrower Definitions \ref{def: multiplicative spectral sequence} and \ref{definition: Classical multiplicative convergence} are   used in \cite{Fuks86}.
\end{example}

\begin{example}[\textbf{\c{S}tefan spectral sequence for an extension of Hopf algebras}]\label{ex: Stefan spectral sequence for an extension of Hopf algebras}
  Assume that we are given  an extension of Hopf algebras over the field $\bfk$:
    \[
    \bfk \longrightarrow A \stackrel{\iota}{\longrightarrow} B \stackrel{\pi}{\longrightarrow} C \longrightarrow \bfk
    \]
    that is,
    \begin{itemize}
\item $\iota$ is injective;
\item $\pi$ is surjective;
\item $\ker \pi = A^{+}B$;
\item $A = \{ x \in B \mid \sum x_{(1)} \otimes \pi(x_{(2)}) = x \otimes 1_C \}$.
    \end{itemize}
    If $B$ is faithfully flat as an $A$-module, then $B/A$ is a faithfully flat $C$-Galois extension.
    By Nichols--Zoeller theorem, $B$ is necessarily a free $A$-module whenever $B$ is finite-dimensional.
    Let $M$ be a left $B$-module.
    By the fact that the Hopf algebra cohomology $\rmH^\bullet(B,  M)$ is isomorphic to the Hochschild cohomology $\rmHH^\bullet(B, M)$, where the right $B$-action on $M$ is the trivial action, one obtains a functorial multiplicative convergent spectral sequence:
    \[
    \calE^{p,q}_2=\rmH^p(C, \rmH^q(A, -))\Rightarrow \rmH^{p+q}(B, -).
    \] 
    If $A$ is central in $B$, then Adams also constructed a similar multiplicative convergent spectral sequence through an alternative approach, see \cite[Theorem 9.12]{McC01}. 
    This spectral sequence may be used to attack the finite generation conjecture in \cite[Question 1.1]{AN26}.
\end{example}

\bigskip

\textbf{Acknowledgements} 
The  second and the third authors   were supported by  the National Key R  $\&$ D Program of China (No. 2024YFA1013803) 
   and    by Shanghai Key Laboratory of PMMP (No. 22DZ2229014).


\textbf{Conflict of Interest}

None of the authors has any conflict of interest in the conceptualization or publication of this
work.

\textbf{Data availability}

Data sharing is not applicable to this article as no new data were created or analyzed in this study.

\bigskip

\bibliographystyle{alpha}
\bibliography{reference}

@article {Tak77,
    AUTHOR = {M. Takeuchi},
     TITLE = {Groups of algebras over {$A\otimes \overline A$}},
   JOURNAL = {J. Math. Soc. Japan},
  FJOURNAL = {Journal of the Mathematical Society of Japan},
    VOLUME = {29},
      YEAR = {1977},
    NUMBER = {3},
     PAGES = {459--492},
      ISSN = {0025-5645,1881-1167},
   MRCLASS = {16A24 (13D15 14L15 18G30)},
  MRNUMBER = {506407},
       DOI = {10.2969/jmsj/02930459},
       URL = {https://doi.org/10.2969/jmsj/02930459},
}

@article {Lu96,
    AUTHOR = {J.-H. Lu},
     TITLE = {Hopf algebroids and quantum groupoids},
   JOURNAL = {Internat. J. Math.},
  FJOURNAL = {International Journal of Mathematics},
    VOLUME = {7},
      YEAR = {1996},
    NUMBER = {1},
     PAGES = {47--70},
      ISSN = {0129-167X,1793-6519},
   MRCLASS = {16W30 (16S40 20N02 81R50)},
  MRNUMBER = {1369905},
MRREVIEWER = {Yvette\ Kosmann-Schwarzbach},
       DOI = {10.1142/S0129167X96000050},
       URL = {https://doi.org/10.1142/S0129167X96000050},
}

@article {AN26,
    AUTHOR = { N. Andruskiewitsch  and  S. Natale },
     TITLE = {On the finite generation of the cohomology of abelian
              extensions of {H}opf algebras},
   JOURNAL = {Doc. Math.},
  FJOURNAL = {Documenta Mathematica},
    VOLUME = {31},
      YEAR = {2026},
    NUMBER = {2},
     PAGES = {401--433},
      ISSN = {1431-0635,1431-0643},
   MRCLASS = {16T05 (16E40)},
  MRNUMBER = {5009501},
       DOI = {10.4171/dm/1025},
       URL = {https://doi.org/10.4171/dm/1025},
}

@article{BCM86,
 author = {R. J. Blattner and M. Cohen and S. Montgomery},
 journal = {Trans.   Amer. Math. Soc.},
 number = {2},
 pages = {671--711},
 publisher = {American Mathematical Society},
 title = {Crossed Products and Inner Actions of {Hopf} Algebras},
 volume = {298},
 year = {1986}
}

@article {Beyl81,
    AUTHOR = {F. R. Beyl},
     TITLE = {The spectral sequence of a group extension},
   JOURNAL = {Bull. Sci. Math.},
  FJOURNAL = {Bulletin des Sciences Math\'ematiques. 2e S\'erie},
    VOLUME = {105},
      YEAR = {1981},
    NUMBER = {4},
     PAGES = {417--434}
}

@article{BH04,
  author  = {T. Brzezi{\'n}ski  and  P. M. Hajac},
  title   = {The {Chern}--{Galois} character},
  journal = {C. R. Math. Acad. Sci. Paris},
  volume  = {338},
  year    = {2004},
  pages   = {113--116}
}

@book{BW03,
  author    = {T. Brzezi{\'n}ski and R. Wisbauer},
  title     = {Corings and Comodules},
  series    = {London Mathematical Society Lecture Note Series},
  volume    = {309},
  publisher = {Cambridge University Press},
  address   = {Cambridge},
  year      = {2003}
}

@inproceedings {BW07,
    AUTHOR = {S. M. Burciu and  S. Witherspoon},
     TITLE = {Hochschild cohomology of smash products and rank one {H}opf algebras},
 BOOKTITLE = {Proceedings of the {XVI}th {L}atin {A}merican {A}lgebra
              {C}olloquium ({S}panish)},
    SERIES = {Bibl. Rev. Mat. Iberoamericana},
     PAGES = {153--170},
 PUBLISHER = {Rev. Mat. Iberoamericana, Madrid},
      YEAR = {2007},
}

@article{DT86,
 author = {Y. Doi and M. Takeuchi},
 journal = {Comm. Algebra},
 pages = {801-817},
 title = {Cleft comodule algebras for a
bialgebra},
 volume = {14},
 issue={5},
 year = {1986}
}

@book{Fuks86,
  author     = {D. B. Fuks},
  title      = {Cohomology of Infinite-Dimensional Lie Algebras},
  series     = {Contemporary Soviet Mathematics},
  publisher  = {Consultants Bureau},
  address    = {New York and London},
  year       = {1986},
  note       = {Translated from the Russian}
}

@article{Ger63,
 author = {M. Gerstenhaber},
 journal = {Ann. of Math.},
 number = {2},
 pages = {267--288},
 title = {The Cohomology Structure of an Associative Ring},
 volume = {78},
 year = {1963}
}

@article{HS53,
  author  = {G. Hochschild and J.-P. Serre},
  title   = {Cohomology of Group Extensions},
  journal = {Trans. Amer. Math.  Soc.},
  volume  = {74},
  number  = {1},
  pages   = {110--134},
  year    = {1953},
}

@article{KT81,
  author  = {H. F. Kreimer and M. Takeuchi},
  title   = {Hopf algebras and {Galois} extensions of an algebra},
  journal = {Indiana Univ. Math. J.},
  volume  = {30},
  number  = {5},
  pages   = {675--692},
  year    = {1981},
}

@article{Kunzer08,
  author  = { M. K{\"u}nzer},
  title   = {Comparison of Spectral Sequences Involving Bifunctors},
  journal = {Doc. Math.},
  volume  = {13},
  year    = {2008},
  pages   = {677--737}
}

@article{LRW25,
  author  = {L. Liu and W. Ren and S. Wang},
  title   = {Cup Products on {Hochschild} Cohomology of
 {Hopf}–{Galois} Extensions},
  journal = {arXiv:2502.01967},
  year    = {2025},
  eprint  = {2502.01967},
  archivePrefix = {arXiv},
  primaryClass = {cs.LG}
}

@book{McC01,
  title={A User's Guide to Spectral Sequences (Second Edition)},
  author={J. McCleary},
  year={2001},
  publisher={Cambridge University Press}
}

@book{Mont93,
  author    = {S. Montgomery  },
  title     = {Hopf Algebras and Their Actions on Rings},
  series    = {CBMS Regional Conference Series in Mathematics},
  volume    = {82},
  publisher = {American Mathematical Society},
  year      = {1993}
}

@article{Neg15,
  author        = {C. Negron},
  title         = {Spectral sequences for the cohomology rings of a smash product},
  journal       = {J. Algebra},
  volume        = {433},
  pages         = {73--106},
  year          = {2015},
  doi           = {10.1016/j.jalgebra.2015.02.021},
  eprint        = {1401.3551},
  archivePrefix = {arXiv
::contentReference[oaicite:4]{index=4}
},
  primaryClass  = {math.KT}
}

@article{SS05,
  author  = {P. Schauenburg and H.-J. Schneider  },
  title   = {On generalized {Hopf} {Galois} extensions},
  journal = {J. Pure Appl. Algebra},
  volume  = {202},
  year    = {2005},
  pages   = {168--194}
}

@article{Sch98,
  author  = {P. Schauenburg},
  title   = {Bialgebras over noncommutative rings and a structure theorem for {Hopf} bimodules},
  journal = {Appl. Categ. Struct.},
  volume  = {6},
  pages   = {193--222},
  year    = {1998}
}

@article{Sch90Princialhomospace,
  author  = {H.-J. Schneider},
  title   = {Principal homogeneous spaces for arbitrary {Hopf} algebras},
  journal = {Israel J.    Math.},
  volume  = {72},
  year    = {1990},
  pages   = {167--195}
}

@article{Ste95,
  author  = { D. {\c S}tefan},
  title   = {{Hochschild} cohomology on {Hopf} {Galois} extensions},
 JOURNAL = {J. Pure Appl. Algebra},
 fuljournal = {Journal of Pure and Applied Algebra},
  volume  = {103},
  number  = {2},
  pages   = {221--233},
  year    = {1995},
  doi     = {10.1016/0022-4049(95)00101-2}
}

@article{Wang2018,
  author  = {Z.W. Wang},
  title   = {Note on the {Hochschild} cohomology of {Hopf}-{Galois} extension},
  journal = {Acta Math. Hung.},
  volume  = {154},
  number  = {1},
  year    = {2018},
  pages   = {223--230}
}

@article{Zhu26,
  author  = {R. Zhu},
  title   = {Skew {Calabi-Yau} property of faithfully flat {Hopf} {Galois} extensions},
  journal = {J. Algebra},
  volume  = {691},
  pages   = {597--647},
  year    = {2026}
}

@article{Sch90,
  author  = {H.-J. Schneider},
  title   = {Representation theory of {Hopf} {Galois} extensions},
  journal = {Israel J. Math.},
  volume  = {72},
  year    = {1990},
  pages   = {196–231}
}

@article{Ul81,
  author  = {K. H. Ulbrich},
  title   = {Vollgraduierte {Algebren}},
  journal = {Abh. Math. Sem.   Univ. Hamb.},
  volume  = {51},
  year    = {1981},
  pages   = {136--148}
}

@book{Weibel94,
  author    = {C. A. Weibel},
  title     = {An Introduction to Homological Algebra},
  series    = {Cambridge Studies in Advanced Mathematics},
  volume    = {38},
  publisher = {Cambridge University Press},
  address   = {Cambridge},
  year      = {1994},
}

@article{Douady59,
  author       = {A. Douady},
  title        = {La suite spectrale d'{A}dams: structure multiplicative},
  journal      = {S\'eminaire Henri Cartan},
  volume       = {11},
  number       = {2},
  year         = {1959},
  note         = {Expos\'e 19, Secr\'etariat math\'ematique, Paris},
  url          = {https://www.numdam.org/item/SHC_1958-1959__11_2_A10_0/}
}

@phdthesis{Hedenlund20,
  author      = {A. Hedenlund},
  title       = {Multiplicative {T}ate Spectral Sequences},
  school      = {University of Oslo},
  year        = {2020},
  type        = {Ph.{D}. Thesis},
  eprint      = {2008.09095},
  archivePrefix = {arXiv},
  primaryClass  = {math.AT}
}

@book{CE56,
  author       = {H. Cartan and S. Eilenberg},
  title        = {Homological Algebra},
  publisher    = {Princeton University Press},
  year         = {1956},
  series       = {Princeton Mathematical Series},
  number       = {19},
  isbn         = {978-0-691-07976-1}
}

@masterthesis{dePotter23,
  author      = {T. de Potter},
  title       = {An {$\infty$}-categorical perspective on spectral sequences,  {M}aster's Thesis, {Utrecht} {University}},
  year        = {2023}
}

@book{LurieHA,
  author       = {J. Lurie},
  title        = {Higher Algebra},
  year         = {2017},
  note         = {Available at \url{https://www.math.ias.edu/~lurie/}}
}

@article{Massey52,
  author  = {Massey, W. S.},
  title   = {Exact Couples in Algebraic Topology ({Parts I and II})},
  journal = {Ann. of Math.},
  series  = {Second Series},
  volume  = {56},
  number  = {2},
  pages   = {363--396},
  year    = {1952},
  publisher = {Annals of Mathematics},
  doi     = {10.2307/1969805},
  url     = {https://www.jstor.org/stable/1969805}
}

@article{Massey53,
  author       = {W. S. Massey},
  title        = {Exact couples in algebraic topology  ({Parts III, IV, and V})},
  journal      = {Ann. of Math.},
  series       = {2},
  volume       = {57},
  pages        = {248--286},
  year         = {1953},
  doi          = {10.2307/1969859}
}

@article{Massey54,
  author       = {W. S. Massey},
  title        = {Products in exact couples},
  journal      = {Ann. of Math.},
  series       = {2},
  volume       = {59},
  pages        = {558--569},
  year         = {1954},
  doi          = {10.2307/1969719}
}

@incollection{Day70,
  author       = {B. Day},
  title        = {On closed categories of functors},
  booktitle    = {Reports of the Midwest Category Seminar IV},
  series       = {Lecture Notes in Math.},
  volume       = {137},
  pages        = {1--38},
  year         = {1970},
  publisher    = {Springer},
  doi          = {10.1007/BFb0060438}
}

@article{HochschildSerre1953,
  author  = {G. Hochschild  and J. P. Serre},
  title   = {Cohomology of {Lie} algebras},
  journal = {Annals of Mathematics},
  series  = {Second Series},
  year    = {1953},
  volume  = {57},
  number  = {3},
  pages   = {591--603},
  doi     = {10.2307/1969740},
  url     = {https://www.jstor.org/stable/1969740},
  mrnumber = {MR0054581},
}

@article{Glasman16,
  author       = {S. Glasman},
  title        = {Day convolution for {$\infty$}-categories},
  journal      = {Math. Res. Lett.},
  volume       = {23},
  number       = {5},
  pages        = {1369--1385},
  year         = {2016},
  doi          = {10.4310/MRL.2016.v23.n5.a8}
}

\end{document}